\documentclass[oneside,reqno,10pt]{amsart}
\usepackage{graphicx,amsfonts,amssymb,amsmath,amsthm,url,amscd,comment}
\usepackage{color}
\usepackage[usenames,dvipsnames]{xcolor}
\usepackage[normalem]{ulem}
\usepackage{pdfsync}
\usepackage{graphicx, amsmath, amssymb, amsfonts, amsthm, stmaryrd, tikz, amscd}
\usepackage[all]{xy}
\usetikzlibrary{matrix,arrows,decorations.pathmorphing}

 \usepackage[hyperfootnotes=false, colorlinks, citecolor=	 RoyalBlue, urlcolor=blue, linkcolor=blue         ]{hyperref}

\theoremstyle{plain} %--default
\newtheorem{theorem}             {Theorem}  % [section]
\newtheorem{lemma}      [theorem]{Lemma}

\newtheorem{proposition}[theorem]{Proposition}

\theoremstyle{definition}

\theoremstyle{remark}

\renewcommand{\Re}{\mathrm{Re}}
\renewcommand{\Im}{\mathrm{Im}}

\renewcommand{\geq}{\geqslant}
\renewcommand{\leq}{\leqslant}

\DeclareMathOperator{\sgn}{sgn}
\DeclareMathOperator{\SL}{SL}
\DeclareMathOperator{\Mp}{Mp}
\DeclareMathOperator{\GL}{GL}

\DeclareMathOperator{\height}{height}

\DeclareMathOperator{\SO}{SO}
\DeclareMathOperator{\ad}{ad}
\DeclareMathOperator{\Ad}{Ad}
\DeclareMathOperator{\disc}{disc}

\DeclareMathOperator{\Hom}{Hom}
\DeclareMathOperator{\End}{End}
\DeclareMathOperator{\JL}{JL}

\def\eps{\varepsilon}

\DeclareMathOperator{\PGL}{PGL}
\DeclareMathOperator{\Jac}{Jac}

\DeclareMathOperator{\diag}{diag}

\def\O{\operatorname{O}}

\DeclareMathOperator{\Fix}{Fix}
\DeclareMathOperator{\Opp}{Op}

\DeclareMathOperator{\sym}{sym}

\DeclareMathOperator{\ram}{ram}
\DeclareMathOperator{\fin}{fin}

\DeclareMathOperator{\Eis}{Eis}
\DeclareMathOperator{\reg}{reg}

\DeclareMathOperator{\res}{res}

\DeclareMathOperator{\trace}{trace}

\DeclareMathOperator{\vol}{vol}
\DeclareMathOperator{\nr}{nr}

\DeclareMathOperator{\tr}{tr}

\DeclareMathOperator{\supp}{supp}

\author{Paul D. Nelson}
\address{ETH Zurich, Department of Mathematics, R{\"a}mistrasse 101, CH-8092, Zurich, Switzerland}
\email{paul.nelson@math.ethz.ch}

\date{\today}
\title{Quantum variance on quaternion algebras, I}
\hypersetup{
  pdfkeywords={},
  pdfsubject={},
  pdfcreator={Emacs 24.5.1 (Org mode 8.2.10)}}
\begin{document}

\maketitle
\begin{abstract}
We determine the quantum variance of a sequence of families of automorphic forms on 
a compact quotient arising from a non-split quaternion algebra.
Our results compare to those obtained 
on $\SL_2(\mathbb{Z}) \backslash \SL_2(\mathbb{R})$
in work of
Luo, Sarnak and Zhao,
whose method required a cusp.
Our
method uses the  theta correspondence to
reduce the problem to the estimation of metaplectic Rankin--Selberg
convolutions.
We apply it here to the first non-split case.
\end{abstract}

\setcounter{tocdepth}{1}
\tableofcontents
\section{Introduction}
\label{sec-1}
\subsection{Overview}\label{sec:overview}
Let $M$ be a negatively curved compact Riemannian manifold
with
$d := \dim(M) \geq 2$.
It is known in various senses that a \emph{random} Laplace
eigenfunction on $M$ of large eigenvalue is uniformly
distributed, but very rarely known just \emph{how} uniformly.
To explain,
denote by $\mu$ the probability measure on $M$ that
is a multiple of the volume measure
and
let $T \rightarrow \infty$ be a positive parameter.
For each $t \in [T, 2 T]$, choose an orthonormal basis
for the $t^2$-eigenspace of the Laplacian
acting on $L^2(M,\mu)$.
Let $\mathcal{F}_T$
denote the union of these bases; it is finite set of cardinality
$|\mathcal{F}_T| \asymp T^{d}$.
To each $\varphi \in \mathcal{F}_T$,
attach the $L^2$-mass
\[
\mu_\varphi := |\varphi|^2 \mu,
\quad \mu_\varphi(\Psi)
:=
\int_{x \in M}
\Psi(x) |\varphi|^2(x) \, d \mu(x).
\]
It is a probability measure on $M$.
A consequence of the quantum ergodicity theorem
is that the \emph{mean} $|\mathcal{F}_T|^{-1} \sum_{\varphi \in
  \mathcal{F}_T} \mu_\varphi$
converges weakly to $\mu$.
The \emph{quantum variance problem}
concerns the asymptotic determination
of the quadratic form
\begin{align}\label{eq:variance-sums-general}
  \mathcal{V}_T(\Psi) &:=
\frac{1}{|\mathcal{F}_T|}
\sum_{\varphi \in \mathcal{F}_T}
\left\lvert
  \mu_\varphi(\Psi)
  -
  \mu(\Psi)
                        \right\rvert^2
                        \\ \label{eq:quadrilinear-obvious}
  &=
\frac{1}{|\mathcal{F}_T|}
\sum_{\varphi \in \mathcal{F}_T}
\int_{x,y \in M}
\Psi(x) \overline{\Psi(y)}
(|\varphi(x)|^2 - 1)
(|\varphi(y)|^2 - 1)
\, d \mu(x) \, d \mu(y)
\end{align}
and its bilinearization $\mathcal{V}_T(\Psi_1,\Psi_2)$, which
quantify the \emph{correlations of the fluctuations} of the
$\mu_\varphi$. They extend naturally to the phase space
$\mathbf{X} = S^* M$ by replacing $\mu_\varphi$ with its
microlocal lift
$\omega_\varphi(\Psi) := \langle \Opp(\Psi) \varphi, \varphi
\rangle$
and $\mu$ with the Liouville measure.

Zelditch
introduced these sums and showed using semiclassical techniques that $\mathcal{V}_T(\Psi) =
O(1/\log T)$
(see \cite{MR916129,MR1262192,MR2267060}).
A prediction of
Feingold--Peres (see \cite{MR848319},
\cite[\S15.6]{2009arXiv0911.4312Z}, \cite[\S4.1.3]{MR3204186})
suggests that for ``generic''
$M$, the renormalized quantum variance $T^{d-1} 
\mathcal{V}_T$
is asymptotic to the classical variance % $\mathcal{V}_{cl}$
of the geodesic flow.
% (cf. \cite{MR2103474,MR2651907,2013arXiv1303.6972S}).
The quantum variance problem
is thus extremely delicate:
a solution in the expected form
\begin{equation}\label{eq:expected-soln-qv}
  T^{d-1}  \mathcal{V}_T(\Psi)
  = \mathcal{V}_\infty(\Psi) + o(1)
\end{equation}
demands not only the enormous improvement
$\mathcal{V}_T(\Psi) = O(T^{1-d})$ upon the semiclassical estimate,
but the further extraction of a main term.

Luo--Sarnak \cite{MR2103474} insightfully observed that on
the modular curve $\SL_2(\mathbb{Z}) \backslash \mathbb{H}$,
the quantum variance problem for Hecke--Laplace eigenfunctions
may be profitably attacked
using arithmetic techniques stemming from the theory of Hecke operators. 
The solution, 
obtained in sharpest form by Sarnak--Zhao \cite{2013arXiv1303.6972S}
after the increasingly sophisticated works
\cite{MR1361757,MR1465794,luo-sarnak-mass,MR2103474,MR2651907},
is arithmetically interesting:
central values
of $L$-functions quantify the deviation of the generic prediction from the truth.

This article introduces a method for determining the quantum
variance of families of automorphic forms on \emph{compact}
arithmetic quotients attached to \emph{non-split} quaternion
algebras $B$.  The \emph{non-compact} quotient
$\SL_2(\mathbb{Z}) \backslash \mathbb{H}$ considered by
Luo--Sarnak--Zhao arises algebraically from the \emph{split}
quaternion algebra $B = M_2(\mathbb{Q})$.  The problem had been
open in every non-split case prior to this work.  We aim here to
introduce our method by application to the simplest non-trivial
non-split case.

The first step in the method of Luo--Sarnak--Zhao
reduces the expression \eqref{eq:quadrilinear-obvious},
which
is \emph{quadrilinear} in $\varphi$,
to a complicated \emph{bilinear} expression
to which trace formulas apply
and from which a colossal limiting bilinear form eventually emerges.
This is achieved
by combining the close relation
between Fourier coefficients at the cusp
$\infty$
of $\SL_2(\mathbb{Z}) \backslash \mathbb{H}$
and Hecke eigenvalues $\lambda(n)$
with the Hecke multiplicativity
$\lambda(m) \lambda(n) = \sum_{d \mid m,n} \lambda(m n/d^2)$.
Unfortunately, such an approach is unavailable
on compact arithmetic quotients,
which have no cusps
and
hence no notion of Fourier expansion that interacts suitably
with the Hecke theory.

In our method, the theta
correspondence is decisive. The well-known seesaw diagram
\begin{equation}\label{eqn:seesaw}
  \xymatrix{\O(B)\ar @{-}[d]\ar @{-}[dr]&
    \Mp_2\times \Mp_2 \ar @{-}[dl]\ar @{-} [d]\\ 
    \O(1) \times \O(B^0) & \SL_2}
\end{equation}
substitutes for the linearizing role played by Fourier
coefficients and Hecke multiplicativity in the method of
Luo--Sarnak--Zhao.  
The basic input
powering convergence to the limit
is decay of matrix coefficients for the
dihedral\footnote{
  In this article,
  only ``split dihedral'' forms (i.e., unitary Eisenstein
  series)
  contribute.
}
spectrum
of $\SL_2$.
%We do not use the triple product formula.
The Rallis inner product formula and the
Maass--Shintani--Waldspurger theta lift
are shown to give natural interpretations to the main terms obtained here and
in earlier works.
We discover ``secondary main terms''
to the quantum variance sums
at the square-root cancellation threshold.
The primary novelty is that we prove asymptotic formulas
for quantum variance sums on a non-split quaternion
algebra for the first time.\footnote{
  We refer to \cite{MR2150390, MR2356110}
  for some analogous results in a simpler arithmetic setting.
}

A general tool available
on any quaternionic congruence quotient $\Gamma \backslash
\mathbb{H}$
is the \emph{triple product formula}
(see \cite{harris-kudla-1991,watson-2008,MR2449948}).
When both $\phi$ and $\Psi$ are Hecke eigenfunctions,
it relates $|\mu_\varphi(\Psi)|^2$ to the $L$-value $L(\varphi
\times \varphi \times \Psi, 1/2)$
and $\mathcal{V}_T(\Psi)$ to a weighted average 
whose analysis succumbs to well-developed
techniques
(see \cite{MR2103474,MR2651907,2013arXiv1303.6972S})
giving
Lindel{\"o}f-consistent asymptotics of the form \eqref{eq:expected-soln-qv}.
Unfortunately, the triple product
formula is non-linear; it does not apply when,
for instance,
$\Psi$ is the sum of two \emph{inequivalent}
mean zero
Hecke eigenfunctions $\Psi_1, \Psi_2$.
For the correlations
\begin{align*}
  \mathcal{V}_T(\Psi_1,\Psi_2)
  &=
\frac{1}{|\mathcal{F}_T|}
\sum_{\varphi \in \mathcal{F}_T}
\mu_\varphi(\Psi_1)
    \overline{\mu_\varphi(\Psi_2)}
    \\
&=
\frac{1}{|\mathcal{F}_T|}
\sum_{\varphi \in \mathcal{F}_T}
\int_{x,y \in M}
\Psi_1(x) \overline{\Psi_2(y)}
|\varphi(x)|^2
|\varphi(y)|^2
\end{align*}
one then expects (following \cite{MR2103474})  the further
cancellation $\mathcal{V}_T(\Psi_1,\Psi_2) = o(1/T)$ coming
from the \emph{independence of variation in sign} of the
quantities $\mu_\varphi(\Psi_1)$, $\mu_\varphi(\Psi_2)$ beyond
the prediction $\mu_\varphi(\Psi_i) = O(T^{-1/2+o(1)})$ 
of the Lindel{\"o}f
hypothesis for their \emph{magnitude}.
Detecting such cancellation is the fundamental difficulty overcome
in the works of Luo--Sarnak--Zhao.
It is achieved by lengthy analysis
of symmetry properties of their emergent bilinear form,
which is found (remarkably) to be Hecke self-adjoint.
The analogous difficulty in our approach is ultimately addressed by the
unramified case of the local
theta correspondence for $(\Mp_2, \O(B^0))$.

% The solutions to known cases of the quantum variance problem
% involve notorious technical difficulties.
% To mitigate as
% much as possible,

% for attacking a problem
% whose notoriously technical problem.

Our main result concerns automorphic forms with ramification
varying in a \emph{non-archimedean} aspect.  The setup differs
superficially from what was discussed above, but retains the
fundamental difficulty while allowing us to introduce the core
ideas of the method as accessibly as we can.
% The
% essential prerequisite is some familiarity with the classical
% holomorphic theta functions attached to positive definite quadratic forms.

% Most importantly,
% it allows us to present

% while allowing for a clearer presentation
% of the core ideas of the new method.  !!!For example, it allows us

% while allowing for a clearer presentation
% of the core ideas of the new method.  !!!For example, it allows us
% to formulate almost all of the arguments of the paper in terms
% of classical holomorphic theta functions

\subsection{Statement of main result}
Terminology to follow is standard, and will be reviewed
more thoroughly
in \S\ref{sec:preliminaries}.
Set $G := \GL_2(\mathbb{Q}_2)$, $K := \GL_2(\mathbb{Z}_2)$.
We fix a discrete cocompact subgroup $\Gamma < G$ arising
from a maximal order in the quaternion algebra ramified at
$\{\infty, 23 \}$ (see \S\ref{sec:generalities}).
(The significance
of the pair of numbers $(2,23)$
is that it is the lexicographically smallest
for which the problem to be discussed has all essential
features;
we have focused our discussion in this way
to simplify notation and
exposition.)
Let
$N$ be a positive integral parameter tending off to $\infty$.
We consider the
family $\mathcal{F}_N$ of $L^2$-normalized newforms (see
\S\ref{sec:famil-balanc-newv}) on a sequence of congruence
covers
\[
  \mathbf{Y}_N := \Gamma \backslash G / K_{-N..N}
  \text{ of }\mathbf{Y} := \Gamma \backslash G / K,
  \]
where more generally
\begin{equation}\label{eq:defn-K-neg-N-pos-N}
  K_{-N_1..N_2} := K \cap \begin{pmatrix}
  \mathbb{Z}_2 & 2^{N_1} \mathbb{Z}_2 \\
  2^{N_2} \mathbb{Z}_2 & \mathbb{Z}_2
\end{pmatrix}.
\end{equation}
The sets $\Gamma \backslash G,
\mathbf{Y}_N, \mathbf{Y}$
come equipped with compatibly-defined Hecke
correspondences
$T_n$
for odd integers $n \geq 1$ (see \S\ref{sec:hecke-operators});
on $\mathbf{Y}$,
$T_n$ is defined for all $n \geq 1$.

The sets $\mathbf{Y}$ and $\Gamma \backslash G$ are
roughly $2$-adic analogues of a compact arithmetic hyperbolic
surface and its cotangent bundle, respectively.
The base space $\mathbf{Y}$
turns out to have cardinality three.
One may regard it as
% the
a
$(2+1)$-regular
directed multigraph on three vertices
% \footnote{
%   We produced the image
%   using the ``Graph'' and ``BrandtModule'' functions in SAGE \cite{sage2015}.}
% \begin{center}
%   \includegraphics[width=8cm,height=1.5cm]{b23-5.png}
% \end{center}
and its cover $\mathbf{Y}_{N}$ as (with suitable interpretation,
taking into account torsion) the set of
non-backtracking paths
$y = (y_{-N} \rightarrow \dotsb \rightarrow y_{N})$ of length
$2 N + 1$ on $\mathbf{Y}$, with the covering map
$\mathbf{Y}_{N} \rightarrow \mathbf{Y}$ given by the projection
$y \mapsto y_0$.
One may also identify $\mathbf{Y}$ with the set of isomorphism
classes
of supersingular elliptic curve $E$ in characteristic $23$
(see \cite[\S2]{MR894322})
and $\mathbf{Y}_{N}$
with the  cover obtained by considering
level structure
consisting of pairs $C_1, C_2$ of cyclic subgroups
of order $2^{N}$
with $C_1 \cap C_2 = \{0\}$;
in fact,
it will be convenient to adopt this
algebraic perspective
for some calculations (\S\ref{sec:appendix-fluctuations}).

The space $\Gamma \backslash G$
has a natural invariant measure (see \S\ref{sec:measures})
which induces a natural measure on $\mathbf{Y}$.
The space $L^2(\mathbf{Y})$ is three-dimensional,
% corresponding
% to the three vertices in the above depiction of $\mathbf{Y}$,
and decomposes as
\[L^2(\mathbf{Y}) = \mathbb{C} \oplus \mathbb{C} \Psi_1 \oplus
  \mathbb{C} \Psi_2,
\]
where $\Psi_1, \Psi_2$ are orthonormal mean zero real-valued
Hecke eigenfunctions, well-defined up to permutation and sign.
For orientation,
we record that under the Eichler/Jacquet--Langlands correspondence,
$\Psi_1, \Psi_2$
correspond
to the weight $2$ newforms  $\Psi_1^{\JL}, \Psi_2^{\JL}$ on $\Gamma_0(23)$ and
$\mathcal{F}_N$ to the set of weight $2$ newforms
on $\Gamma_0(2^{2 N} \cdot 23)$.

To each $\varphi \in \mathcal{F}_N$ we attach
the ``harmonic
weight'' $\iota_{\varphi} := L^{(2)}(\ad \varphi,1)$ of size
$2^{o(N)}$ (see \S\ref{sec-1-3-7}); these mild weights play a
role (non-obviously) similar to that in
\cite{2013arXiv1303.6972S} and earlier works, and should be
ignored on a first reading.

The fundamental object
of study in this article
is the real symmetric $2 \times 2$ matrix
$V_N := (V_N^{kl})_{k,l=1,2}$
with entries
\[
  V_N^{k l}
  := \frac{1}{2^{2 N}}
  \sum _{\varphi \in \mathcal{F}_N}
  \iota_{\varphi}
  \int_{x,y \in \Gamma \backslash G}
  \Psi_k(x) \Psi_l(y) |\varphi|^2(x) |\varphi|^2(y).
\]
Equivalently,
let $\mu_\varphi$ denote the measure
on $\mathbf{Y}$ given by
$\mu_\varphi(\Psi) := \int_{\Gamma \backslash G} \Psi
|\varphi|^2$;
then
\[
  V_{N}^{k l}
  = \frac{1}{2^{2 N}} \sum_{\varphi \in \mathcal{F}_N}
  \iota_{\varphi}
  \mu_\varphi(\Psi_k)
  \mu_\varphi(\Psi_l).
\]
For the discussion of
\S\ref{sec:overview}
specialized to the case
of surfaces,
the numerical analogue here of the quantity $T$
is $2^N$; for instance,
the family cardinality is $|\mathcal{F}_N| \asymp 2^{2 N}$
(see \S\ref{sec:mean-stats} for an exact formula).
To justify interpreting $V_N^{k l}$ as a variance,
we verify in \S\ref{sec:mean-stats}
that
\begin{equation}\label{eq:linear-statistics-justification}
    \frac{1}{|\mathcal{F}_N|} \sum_{\varphi \in \mathcal{F}_N}
  \mu_\varphi = \mu
\end{equation}
with $\mu$ the probability measure
on $\mathbf{Y}$ that is a multiple of the pushforward
of $\int_{\Gamma \backslash G}$.
The entries of $V_N$ quantify the correlations of the
fluctuations of the $L^2$-masses $\mu_\varphi$ of a random newform
$\varphi \in \mathcal{F}_N$ on the congruence cover
$\mathbf{Y}_N$ when tested against the mean zero
Hecke eigenfunctions $\Psi_1,\Psi_2$ on the base space
$\mathbf{Y}$.
The setup is thus the natural $2$-adic
analogue of that in \S\ref{sec:overview}.

We
normalize the Hecke eigenvalues
of $\Psi_1,\Psi_2$
by
$T_n \Psi_i = \sqrt{n} \lambda_{\Psi_i}(n) \Psi_i$
and the standard $L$-functions $L(\Psi_k,s)$
by
analytic continuation
of the Dirichlet series
$\sum_{n \geq 1}
\lambda_{\Psi_k}(n)/ n^{s}$
from the half-plane
$\Re(s)
> 1$, so that $s = 1/2$ is the central point.

\begin{theorem}[Main result]
  \label{thm:special-main}
  \label{thm:3-by-3}~
  \begin{enumerate}
  \item[(i)]
    {\bf Existence of the limiting quantum variance.}
    
    The $2
    \times 2$ matrix limit $\lim_{N \rightarrow
      \infty}
    2^N V_{N}$ exists.
  \item[(ii)]
    {\bf Diagonalization and determination of the limit.}
    
    That limit is given by\footnote{
      One has
      $\{L(\Psi_1,\tfrac{1}{2} ), L(\Psi_2, \tfrac{1}{2} )\}
      \approx \{0.552, 0.450\}$
      according to
      \cite[\href{http://www.lmfdb.org/L/ModularForm/GL2/Q/holomorphic/23/2/1/a/0/}{23.2.1a.0}]{lmfdb},
      \cite[\href{http://www.lmfdb.org/L/ModularForm/GL2/Q/holomorphic/23/2/1/a/1/}{23.2.1a.1}]{lmfdb}}
    \[
      V_{\infty}
      :=
      \begin{pmatrix}
        \tilde{L}(\Psi_1,\tfrac{1}{2}) & 0  \\
        0 & \tilde{L}(\Psi_2,\tfrac{1}{2}) \\
      \end{pmatrix}
    \]
    with
    $\tilde{L}(\Psi_k,\tfrac{1}{2}) := P(\lambda_{\Psi_k}(2)) L(\Psi_k,\tfrac{1}{2})$,
    $P(x) := \pi^2 (15 - 4 \sqrt{2} x) / 69$.
  \item[(iii)]
    {\bf Effective rate of convergence to the limit.}
    
    $|2^N V_N^{kl} - V_{\infty}^{kl}| \leq C N 2^{-N}$ for some absolute effective $C > 0$.
  \end{enumerate}
\end{theorem}

Each assertion in  Theorem \ref{thm:special-main} 
is highly non-trivial -- for instance, 
the existence of the limit
already implies somewhat more than
the strong improvement $V_N^{kl} \ll 2^{-N}$ over the trivial bound
$V_N^{kl} \ll 1$ -- but
the most novel assertion is that for $k \neq l$,
\begin{equation}\label{eq:most-significant-assertion}
  \lim_{N \rightarrow \infty} 2^N V_N^{k l} = 0
  % \quad \quad
\end{equation}
\[
  \text{(in the strong quantitative form $2^N V_N^{k l} = O(N
    2^{-N})$)}
\]
which is the first of its kind in any non-split/cocompact
setting and lies genuinely beyond existing methods such as
Fourier expansions at cusps, trace formulas, and triple product
formulas, which fail even to reduce its proof to a technical
problem.
One can interpret
\eqref{eq:most-significant-assertion}
as reflecting
a non-obvious
Hecke symmetry
enjoyed by the fluctuations of the measures
$\mu_\varphi$.

\subsection{Application to moments
  of (square roots of) $L$-functions}
The values $\mu_\varphi(\Psi_k)$
are basic examples of triple product periods, which have been
extensively studied.\footnote{See for instance the articles
  \cite{MR1145805,BoSP96,MR2740724, MR2018269}, which focus on definite
  quaternion algebras as we do here, or \cite{MR3110797},
  whose considerations apply
  to the analogous periods obtained
  from newforms on the tower
  of covers $\Gamma_0(2^{2
    N}) \backslash \mathbb{H} \rightarrow \SL_2(\mathbb{Z})
  \backslash \mathbb{H}$ of the modular curve.}
The triple product formula
relates the \emph{squares}
(but not the \emph{signs}) of these periods
to
the central values of triple product $L$-functions.
Unlike in the work of Luo--Sarnak--Zhao,
we do not use the triple product formula in the proof of Theorem
\ref{thm:special-main};
nevertheless, it is instructive
to indicate
briefly
how our results translate thereunder.
The formula in question
(see \cite{MR2449948, MR3110797, 2014arXiv1409.8173H})
has the shape
\[
  \iota_{\varphi}
  |\mu_\varphi(\Psi_k)|^2
  =
  2^{-N}
  \frac{\tilde{L}(\varphi \times \varphi \times \Psi_k,\tfrac{1}{2}) }{ L(\ad \varphi,
    1)}
\]
where
$\tilde{L}(\varphi \times \varphi \times \Psi_k,\tfrac{1}{2}) :=
c_0(\Psi_k,\varphi) L(\varphi \times \varphi \times \Psi_k,\tfrac{1}{2})$
for some explicit nonnegative quantity $c_0(\Psi_k,\varphi)$ that we
expect\footnote{ By \cite[Thm 34]{nelson-padic-que}, this
  expectation holds when $\varphi$ is principal series.  It
  should follow in general from the methods of
  \cite{2014arXiv1409.8173H}.
  One could rewrite this article with the role of $p=2$
  replaced by that of some odd prime.
  We expect then that $c_0(\Psi_k,\varphi)$
  would be independent of $\Psi_k,\varphi$.
  This expectation holds when $\varphi$ is principal
  series by
  \cite[Thm 34]{nelson-padic-que}.
}
is uniformly bounded as $\varphi$ varies.
(The reader will lose little in what follows by
ignoring the factor $c_0(\Psi_k,\varphi)$,
whose explication is beyond the scope of this paper.)
The Lindel{\"o}f hypothesis thus predicts
that the \emph{magnitude}
of the real number
$\mu_\varphi(\Psi_k)$ is typically $\approx 2^{-N/2}$,
but says nothing about its \emph{sign}.
Theorem \ref{thm:special-main} tells us that the signs of
$\mu_\varphi(\Psi_k)$ and $\mu_\varphi(\Psi_l)$
are \emph{highly unbiased} for $k \neq l$;
its diagonal and off-diagonal $k \neq l$ cases 
translate to
\begin{equation}\label{eq:diag-estimate}
  \frac{1}{2^{2 N}}
  \sum_{\varphi \in \mathcal{F}_N}
  \frac{\tilde{L}(\varphi \times \varphi \times \Psi_k,\tfrac{1}{2})}{L(\ad
    \varphi, 1)}
  =
  \tilde{L}(\Psi_k,\tfrac{1}{2}) + O \left( \frac{\log |\mathcal{F}_N|}{\sqrt{|\mathcal{F}_N|}} \right),
\end{equation}
\begin{equation}\label{eq:offdiag-estimate}
  \frac{1}{|\mathcal{F}_N|}
  \sum_{\varphi \in \mathcal{F}_N}
  \frac{
    \sqrt{\tilde{L}(\varphi \times \varphi \times \Psi_k,\tfrac{1}{2})}
    \sqrt{\tilde{L}(\varphi \times \varphi \times \Psi_l,\tfrac{1}{2})}
  }{L(\ad \varphi, 1)}
  = O \left( \frac{\log |\mathcal{F}_N|}{\sqrt{|\mathcal{F}_N|}} \right)
\end{equation}
with
the choice of
square-root in \eqref{eq:offdiag-estimate}
given by the sign of the real number $\mu_\varphi(\Psi_k)$.
The diagonal estimate \eqref{eq:diag-estimate}
is consistent with the Lindel{\"o}f hypothesis
on average.\footnote{
  A direct proof of \eqref{eq:offdiag-estimate} using the approximate functional equation,
  Petersson formula and Voronoi summation is heuristically straightforward;
  technical complications arise because of ``weight $2$.''
  The proof of Theorem \ref{thm:special-main} does not use
  the triple product formula.}
The off-diagonal estimate \eqref{eq:offdiag-estimate}
represents \emph{square-root cancellation beyond the Lindel{\"o}f hypothesis},
and does not follow from standard conjectures
such as GRH.
We discuss the error terms further in \S\ref{sec-optimal-error-term}.
For orientation,
we record that $|\mathcal{F}_N| \asymp 2^{2 N}$ and typically
$C(\varphi \times \varphi \times \Psi_k) \asymp 2^{4 N}$.\footnote{
  We remark also that
  by replacing the $\varphi, \Psi_k$
  with some
  modifications
  belonging to the automorphic representation
  that they generate,
  it should be possible to derive explicit analogues of
  \eqref{eq:diag-estimate}, \eqref{eq:offdiag-estimate}
  without
  the factors $c_0(\Psi_k,\varphi)$.
  For instance,
  it should be possible to achieve this
  by leaving $\Psi_k$ as it is
  and replacing $\varphi$ by
  the $2$-adic microlocal lift
  introduced in \cite{nelson-padic-que}
  when $\varphi$ is principal series
  and by an analogous construction involving non-split tori
  when $\varphi$ is supercuspidal.
}

\subsection{Application to quantum unique ergodicity}
We touch briefly on
the relationship of our results
to the study of worst-case
behavior of the measures $\mu_\varphi$.
It is expected that $\mu_{\varphi_N} \rightarrow \mu$ for any
sequence of $\varphi_N \in \mathcal{F}_N$ with
$N \rightarrow \infty$, i.e., that
$\mu_{\varphi_N}(\Psi_k) = o(1) := o_{N \rightarrow \infty}(1)$ for
$k=1,2$.  This expectation
is
much weaker than
the prediction
$|\mu_{\varphi_N}(\Psi_k)| \leq 2^{-(1 + o(1))N/2}$ of the
Lindel{\"o}f hypothesis,
but remains an open variant of the arithmetic
quantum unique ergodicity conjecture as in
\cite{PDN-HQUE-LEVEL,MR3110797,2014arXiv1409.8173H}
which can be shown to follow from subconvexity.  The special case
in which $\varphi_N$ generates a principal series representation
of $\GL_2(\mathbb{Q}_2)$ was recently confirmed unconditionally in
\cite{nelson-padic-que}.  Theorem
\ref{thm:special-main} and Chebyshev's inequality imply that the
Lindel{\"o}f prediction is essentially sharp, that it holds for
a density $1 - o(1)$ subset of $\mathcal{F}_N$, and that all but
a density $O(|\mathcal{F}_N|^{-1/2+o(1)})$ subset of
$\varphi_N \in \mathcal{F}_N$ satisfy
$\mu_{\varphi_N} \rightarrow \mu$ as $N \rightarrow \infty$
(compare with
\cite{MR1361757,luo-sarnak-mass,MR2103474}).

\subsection{Method}
We now indicate in high-level terms why is it natural to study
the quantum variance problem using theta functions and the
diagram \eqref{eqn:seesaw}; a concrete implementation of this
discussion may be found in \S\ref{sec:reduction-proof}.

The arithmetic quotients $\mathbf{X}$ of interest to us are
parametrized by quaternion $\mathbb{Q}$-algebras $B$; one can
take for $\mathbf{X}$ the adelic quotient
$B^\times(\mathbb{Q}) \backslash B^\times(\mathbb{A})$ or a
further quotient thereof.  Given a pair of cusp forms
$\Psi_1, \Psi_2$ on $\mathbf{X}$ with trivial central character,
we would like to understand quantum variance sums of the shape
\begin{equation}\label{eq:section-method-desired-qv-sums-nice-F}
  \sum_{\varphi \in \mathcal{F}} \langle |\varphi|^2, \Psi_1
\rangle \langle \Psi_2, |\varphi|^2 \rangle
\end{equation}
for some ``nice enough family of automorphic forms''
$\mathcal{F}$.

Our first observation is that sums of this general shape are
related via the theta correspondence to four-fold integrals of
theta functions:
Suppose
given a pair of elementary theta functions
$\theta_1, \theta_2$
and ternary theta lifts $h_1,h_2$ of the cusp forms
$\Psi_1,\Psi_2$.
More precisely,
$\theta_1,\theta_2,h_1,h_2$
are functions
on
$\SL_2(\mathbb{Q}) \backslash \Mp_2(\mathbb{A})$ 
attached to Schwartz--Bruhat functions
$\phi_1', \phi_2'$ on $\mathbb{A}$
and
$\phi_2'', \phi_2''$  on $B^0(\mathbb{A})$,
where $B^0$ is the trace zero subspace of $B$.
The Parseval
formula and seesaw duality \eqref{eqn:seesaw}
then
give identities
roughly of the shape
\begin{equation}\label{eq:general-qv-sums-seesaw}
    \langle \theta_1 h_1, \theta_2 h_2 \rangle
  =
  \left(
  \begin{array}{lr}
    \text{ generalized quantum variance sums } \\
    % \, \, \, B^\times(\mathbb{Q})
    % \backslash B^\times(\mathbb{A}),
    \quad \quad \quad \text{ tested against } \Psi_1, \Psi_2
  \end{array}
  \right).
\end{equation}
We explain this shortly
(\S\ref{sec:quant-vari-sums})
in a ``toy example'' 
and
in the body of this article very concretely.
% and in the sequel \cite{nelson-variance-II} systematically.

% More precisely, let $\nr$ denote the reduced
% norm on $B$.  The quadratic space $(B,\nr)$ decomposes
% orthogonally as $\mathbb{Q} \oplus B^0$, where $B^0$ is the
% trace zero subgroup, inducing a factorization
% $\mathcal{S}(B(\mathbb{A})) \cong \mathcal{S}(\mathbb{A})
% \otimes \mathcal{S}(B^0(\mathbb{A}))$
% of adelic Schwartz--Bruhat spaces.
% Then $\theta_1,\theta_2,h_1,h_2$
% are attached to Schwartz--Bruhat functions
% $\phi_1', \phi_2'$ on $\mathbb{A}$
% and
% $\phi_2'', \phi_2''$  on $B^0(\mathbb{A})$.
% The Parseval
% formula and seesaw mechanism \eqref{eqn:seesaw} give identities
% roughly of the shape

The identity \eqref{eq:general-qv-sums-seesaw}
(together with its extension to non-pure tensors
$\phi_i = \sum_{\alpha} \phi_{i,\alpha}' \otimes
\phi_{i,\alpha}''$) suggests a
natural strategy for attacking the quantum variance problem.
It is not \emph{a priori} clear that this strategy should
succeed,
and its implementation
requires further novelties:
\begin{enumerate}
\item
  The precise sums appearing on the RHS
  of \eqref{eq:general-qv-sums-seesaw} depend heavily
  upon the local data $\phi_1', \phi_2', \phi_1'', \phi_2''$.
  One must thus confront the \emph{inversion
    problem} for a family $\mathcal{F}$ of automorphic forms
  which consists of
  exhibiting local data $\phi_i$ so
  that the RHS of \eqref{eq:general-qv-sums-seesaw} resembles the
  quantum variance sums
  \eqref{eq:section-method-desired-qv-sums-nice-F}
over $\mathcal{F}$.
\item
  There is then the
unprecedented \emph{analytic problem} of asymptotically
evaluating the integral of four-fold products of theta functions
on the LHS of \eqref{eq:general-qv-sums-seesaw}
as the local data vary;
we address this by proving asymptotic formulas of the shape
\begin{equation}\label{eq:asymp-inn-prod-expectation}
    \langle \theta_1 h_1, \theta_2 h_2 \rangle
  \approx
  \langle \theta_1 , \theta_2 \rangle 
  \langle h_1 , h_2 \rangle.
\end{equation}
\end{enumerate}
The two problems are
intertwined in that one does not expect the nice
asymptotic behavior
\eqref{eq:asymp-inn-prod-expectation}
for completely general variation of the
local data.

\subsection{Quantum variance sums with Eisenstein observables
  via $L$-functions}\label{sec:quant-vari-sums}
We  include this subsection to convey the flavor of
the first step \eqref{eq:general-qv-sums-seesaw} of our method
to readers having some familiarity with integral representations
of $L$-functions,
but not necessarily with the theta correspondence
and seesaw duality.

We aim to illustrate 
\eqref{eq:general-qv-sums-seesaw}
for a ``toy problem''
(unrelated
to the main result of this article)
involving the Eisenstein series $E_s$ on $\SL_2(\mathbb{Z})
\backslash \mathbb{H}$,
defined for a complex
parameter $s$ by meromorphic continuation of the sums
\[
E_s(z)
:=
\frac{1}{2}
\sum_{ \substack{
    (c,d) \in \mathbb{Z}^2 - \{(0,0)\}  : \\
    \gcd(c,d) = 1
  }
}
\frac{y^s}{|c z + d|^{2 s}}
= y^s + \dotsb.
\]
For $\Re(s)$ large enough
and up to (important) normalizing factors,
Rankin--Selberg theory
and Shimura's symmetric
square integral representation
give
\begin{align*}
  L(\varphi \times \varphi, s)
  &:=
    \zeta(2 s) \sum_{n=1}^{\infty} \frac{|\lambda_\varphi(n)|^2}{n^{s}}
    \approx \mu_{\varphi}(E_s), \\
  L(\sym^2 \varphi, s)
  &:=
    \zeta(2 s) \sum_{n=1}^{\infty} \frac{\lambda_\varphi(n^2)}{n^{s}}
    \approx
    \langle \theta \tilde{E}_s,  \varphi \rangle,
\end{align*}
for some Jacobi theta function $\theta$ and half-integral
weight Eisenstein series
$\tilde{E}_s$.
Combining these identities with the factorization\footnote{
  This is the special case $|\lambda(n)|^2 = \sum_{d|n}
\lambda(n^2/d^2)$
of the Hecke multiplicatively exploited by Luo--Sarnak--Zhao.
}
$L(\varphi \times \varphi, s)
= \zeta(s) L(\sym^2 \varphi, s)$
gives with $\tilde{E}_{s}' := \zeta(s)
\tilde{E}_{s}$ that
\begin{equation}\label{eqn:linearize-rs-shim}
  \mu_\varphi(E_s)
  \approx
  \langle
  \theta
  \tilde{E}_{s}',
  \varphi 
  \rangle.
\end{equation}
Note that
$\varphi \mapsto \langle \theta \tilde{E}_{s}',
\varphi \rangle$
is \emph{linear}, unlike $\varphi \mapsto \mu_\varphi(E_s)$,
and that the non-linear identity \eqref{eqn:linearize-rs-shim},
like the triple product formula, applies only when
$\varphi$ is an eigenfunction.\footnote{
  The analysis of triple product averages mentioned above
  exploits
  the related identity
  $L(\varphi \times \varphi \times \Psi,1/2)
  = L(\sym^2 \varphi \times \Psi,1/2)
  L(\Psi,1/2)$.
}

Consider now the pair of unitary Eisenstein series
$\Psi_1, \Psi_2 := E_{1/2+i t_1}, E_{1/2+i t_2}$
attached to $t_1, t_2 \in \mathbb{R}$.
Let
$\mathcal{F}$ be an orthonormal basis of eigenfunctions for the
discrete spectrum of
$L^2(\SL_2(\mathbb{Z}) \backslash \mathbb{H})$; write $\dotsb$
for the corresponding contribution of the continuous spectrum,
and ignore (for the purposes
of this formal discussion) that most of the sums/integrals written below are
divergent.
From \eqref{eqn:linearize-rs-shim} and Parseval, one obtains with
$h_1, h_2 := \tilde{E}_{1/2+i t_1}', \tilde{E}_{1/2+i t_2}'$
that
\begin{equation}\label{eq:qv-vs-fourfold}
  \sum_{\varphi \in \mathcal{F}}
\mu_\varphi(\Psi_1) \overline{\mu_\varphi(\Psi_2)} + \dotsb 
\approx 
\sum_{\varphi \in \mathcal{F}}
\langle \theta h_1, \varphi
\rangle
\langle \varphi, \theta h_2
\rangle
+ \dotsb
=
\langle \theta h_1,  \theta h_2
\rangle.
\end{equation}
Formally,
quantum variance sums equal 
integrals of
four-fold products of half-integral
weight theta functions.

The theta correspondence and seesaw duality
give non-formal analogues
\eqref{eq:general-qv-sums-seesaw}
of 
\eqref{eq:qv-vs-fourfold};
the point of this article is thus to
determine
quantum variance asymptotics
by deriving some such analogues
and then proving
asymptotic formulas
(\ref{eq:asymp-inn-prod-expectation})
for the inner products that arise.

\subsection{Further perspectives}
We conclude this introduction by noting some reasons to
have anticipated an approach to the quantum variance problem using
the theta correspondence and half-integral weight forms:
\begin{enumerate}
\item The
quantum variance theorems of Luo--Sarnak--Zhao give new proofs
that $L(\pi,\tfrac{1}{2}) \geq 0$ (see \cite[p773,
(2)]{MR2103474}); earlier proofs as in \cite{MR629468,MR1244668}
used half-integral weight forms.  
\item The Rallis inner product
formula implies that the composition of the
Maass--Shintani--Waldspurger lift (for fixed local data) with
the Petersson inner product has similar properties to those of
the limiting bilinear form identified in the works of
Luo--Sarnak--Zhao;
the present work
confirms that this
similarity is no coincidence.  
\item Any progress on the quantum
variance problem presupposes some progress towards the
multiplicity one theorem; the known constructive proof of that,
due to Eichler and Shimizu, used the theta correspondence much
as we do here.
\end{enumerate}

\subsection*{Acknowledgements}
We are grateful to many people for their contributions to this
project.  P. Sarnak introduced us to his work with W. Luo and
emphasized the problem of obtaining analogous results on compact
quotients; we thank him also for much encouragement, inspiration
and helpful feedback concerning earlier versions of this work.
Ph. Michel patiently entertained and offered helpful feedback
and encouragement on some presentations of earlier forms of the
method.  Conversations with K. Prasanna in connection with the
work \cite{MR3356036} were ultimately
helpful also
in connection with the present work.
A. Venkatesh offered helpful encouragement, feedback and
comments which have improved our exposition.  We thank also
{\"O}. Imamoglu, H. Iwaniec, E. Kowalski, D. Ramakrishnan,
A. Saha, J. Tsimerman for their encouragement and interest and
W.T. Gan, B. Gross, R. Holowinsky, N. Templier, Z. Rudnick,
C. Skinner, K. Soundararajan for their feedback, interest and/or
questions on various aspects of this work.

We gratefully acknowledge the support of
NSF grant OISE-1064866
and SNF grant SNF-137488 during the work leading
to this paper.

\section{Preliminaries}\label{sec:preliminaries}
We record definitions, some postponed from \S\ref{sec-1}.

\subsection{Generalities}\label{sec:generalities}
Recall that $G := \GL_2(\mathbb{Q}_2)$, $K := \GL_2(\mathbb{Z}_2)$.
Let $B$ be the quaternion algebra
ramified at $\{\infty, 23\}$.
Let $R \leq B$ be a maximal order.
Choose an embedding of $B$ into the matrix algebra
$M_2(\mathbb{Q}_2)$
under which $R$ embeds in $M_2(\mathbb{Z}_2)$.
Take for $\Gamma$ the image of $R[1/2]^\times$ under this
embedding.
Then $\Gamma < G$ is a discrete cocompact subgroup.
The quotient $\mathbf{X} := \Gamma \backslash G$ is compact.

The reduced norm and trace on $B$ and its extensions
are denoted $\nr, \tr$.
A superscripted $0$ denotes ``trace zero subgroup.''
A subscripted $2$ denotes ``$2$-adic completion.''

\subsection{Hecke operators}\label{sec:hecke-operators}
For an odd positive integer $n$,
the Hecke operator $T_n$
acts on functions $\varphi : \mathbf{X} \rightarrow \mathbb{C}$
by
$T_n \varphi(x) := \sum _{\alpha \in M_n / \Gamma }
\varphi(\alpha^{-1} x)$
where
$M_n := R[1/2] \cap \nr^{-1}(n \mathbb{Z}[1/2]^\times)$.
The group $G$ acts on such functions by right translation, and
commutes with the Hecke operators.

\subsection{Automorphic forms}\label{sec:automorphic-forms}
Denote by
$\mathcal{A}(\mathbf{X})$ the space of \emph{smooth} functions
$\varphi : \mathbf{X} \rightarrow \mathbb{C}$, i.e., those that
are right invariant by some open subgroup of $G$.
For a subgroup $S \leq G$,
denote by $\mathcal{A}(\mathbf{X})^S$ the subspace
of right $S$-invariant functions,
or equivalently, those that factor through
the quotient $\mathbf{X}/S
= \Gamma \backslash G / S$.
We may and shall identify $\mathcal{A}(\mathbf{X})^S$
with a space of functions on $\mathbf{X}/S$.
The Hecke operators act on it.

\subsection{Measures}\label{sec:measures}
Equip $G$ with the Haar measure assigning volume $2$
to $K$,
% Equip $\GL_2(\mathbb{Q}_2)$ with the Haar measure assigning volume $2$
% to $\GL_2(\mathbb{Z}_2)$,
and $\mathbf{X}$ with the quotient measure,
denoted simply $\int_{\mathbf{X}}$;
define $L^2(\mathbf{X})$ with respect to that measure.

Recall that $\mathbf{Y} := \Gamma \backslash G / K = \mathbf{X}
/ K$.
By the convention of \S\ref{sec:automorphic-forms},
functions on $\mathbf{Y}$ are identified with
right $K$-invariant functions on $\mathbf{X}$.
For each $E \in \mathbf{Y}$,
choose a representative $g_E \in G$.
Set $R_E := R[1/2] \cap g_E R_2 g_E^{-1}$; it is a maximal order
in $B$.
Set $w_E := \# R_E^\times / \mathbb{Z}^\times
= (1/2) \# R_E^\times$.
For $\Psi : \mathbf{Y} \rightarrow \mathbb{C}$,
one then has
$\int_{\mathbf{X}} \Psi = \sum_{E \in \mathbf{Y}} \Psi(E) / w_E$.

\subsection{Families of balanced newvectors}\label{sec:famil-balanc-newv}
Let $N$ be a positive integer.
Recall the definition \eqref{eq:defn-K-neg-N-pos-N}
of $K_{-N_1..N_2}$.
By the multiplicity one theorem and local newvector theory,
there is a unique (up to signs) maximal subset $\mathcal{F}_N
\subset \mathcal{A}(\mathbf{X})^{K_{-N..N}}$
with the properties:
\begin{itemize}
\item Each $\varphi \in \mathcal{F}_N$ is an eigenfunction
  for the Hecke operator $T_n$ for all odd natural numbers $n$.
  We accordingly write
  $T_n \varphi
  = \sqrt{n}\lambda_\varphi(n) \varphi$.
\item Each $\varphi \in \mathcal{F}_N$ generates an irreducible
  representation of $G$ under right translation.
\item Each $\varphi \in \mathcal{F}_N$ is real-valued and
  orthonormal,
  and any two $\varphi, \varphi ' \in \mathcal{F}_N$ are
  orthogonal
  to one another.
\item Each $\varphi \in \mathcal{F}_N$
is orthogonal
to any function on $\mathbf{X}$
that is $K_{-N_1..N_2}$-invariant
from some ordered pair $(N_1, N_2)$
with $N_1 \leq N$ and $N_2 \leq N$ and $(N_1,N_2) \neq (N,N)$.
\end{itemize}
% consisting
% of orthonormal real-valued Hecke eigenfunctions
% with the property:
The multiplicity one theorem implies moreover
that each $\varphi \in \mathcal{F}_N$
is determined by its system of Hecke eigenvalues
$\lambda_\varphi(n)$.
The family $\mathcal{F}_N$ is analogous
to (and in Hecke-equivariant bijection with; see \S\ref{sec:eichl-langl-lifts})
the set of normalized
newforms of weight $2$ on $\Gamma_0(2^{2 N} \cdot 23)$.

\subsubsection{Remark}
We consider here the families $\mathcal{F}_N$ arising from the
``balanced'' subgroups $K_{-N..N}$.  Our method applies to the
other subgroups, such as the unbalanced ones $K_{0..N}$ more
commonly denoted``$K_0(2^{N})$'', but the results
obtained are nicer for those considered here.

\subsection{Conventions on modular forms}\label{sec:some-conventions}
We denote by $z := x + i y$ a typical element of the upper
half-plane
and write $q := e^{2 \pi i z}$.
``Modular'' always means ``modular with respect to
some congruence subgroup $\Gamma'$ of
$\Gamma_0(4)$.''
We analytically
normalize holomorphic modular forms $\Phi$ of weight
$k \in \tfrac{1}{2} \mathbb{Z}_{\geq 0}$ by the factor
$y^{k/2}$,
so that $q$-expansions read
$\Phi(z) = y^{k/2} \sum  a_n q^n$
and
dilations $\Phi(z) \mapsto \Phi(a z)$
for $a \in \mathbb{Q}^\times_+$
are unitary for the Petersson inner product, which we
normalize by
$\left\langle \Phi_1, \Phi_2 \right\rangle := \int_{z \in \Gamma'
  \backslash \mathbb{H}} \Phi_1(z) \overline{\Phi_2(z)} \, d
\nu(z)$
for $\Gamma'$ small enough
in terms of $\Phi_1, \Phi_2$
and $\nu = \nu_{\Gamma'}$ the probability measure
that is a multiple of $y^{-2} \, d x \, d y$.
Thus, for instance, $\langle 1,1 \rangle = 1$,
regardless of the congruence quotient $\Gamma' \backslash
\mathbb{H}$
on which the constant function $1$ is regarded as living.
Write $\|\Psi\| := \langle \Psi, \Psi \rangle^{1/2}$.

\subsection{Eichler/Jacquet--Langlands lifts}\label{sec:eichl-langl-lifts}
For $N \geq 1$ and $\varphi \in \mathcal{F}_N$, set
\[
\Phi_\varphi (z) := y \sum _{\substack{
    n \geq 1 : \\ \gcd(n,2) = 1 
  }
}
\sqrt{n} \lambda_\varphi(n) q^n
\]
It is known\footnote{
  This can be deduced from results of Eichler
  (see \cite{MR0485698}, \cite[\S2]{MR579066}) and Atkin--Lehner
  \cite[Thm 3 (iii)]{MR0268123}.
}
that
$\Phi_\varphi$
defines a weight $2$ newform on $\Gamma_0(2^{2N} \cdot 23)$
with Hecke eigenvalues
$\sqrt{n} \lambda_{\varphi}(n)$ for odd natural numbers $n$.

\subsection{Harmonic weights}
\label{sec-1-3-7}
Recall from \S\ref{sec-1} that $\iota_{\varphi} := L^{(2)}(\ad \varphi,1)$.
It is known \cite{HL94} that $\iota_{\varphi} = 2^{o(N)}$ as $N
\rightarrow \infty$.
By the theory of Eisenstein series,\footnote{
  see for instance \cite[p138]{MR2061214},
  \cite[\S5.1]{MR2504745}}
$\|\Phi_{\varphi}\|^2$ is the residue
as $s \rightarrow 1^+$
for some
sufficiently divisible $M \in \mathbb{Z}_{\geq 1}$
of
the integral
\[\int_{\Gamma_1(M) \backslash \mathbb{H}}
(\sum_{\gamma \in \Gamma_\infty\backslash \Gamma_1(M) }
\Im(\gamma z)^s )
|\Phi_\varphi|^2(z)
\, y^{-2} \, d x \, d y\]
which unfolds to
$\int_{y =0}^{\infty}
y^s
\sum_n
n |\lambda_\varphi(n)|^2
y^2 e^{- 4 \pi n y}
\, y^{-2} \, d y$
and then simplifies
to
\[
\int_{y=0}^\infty y^{s+1} e^{- 4 \pi y} \, \frac{d y}{y}
\sum_{\substack{
    n \in \mathbb{Z}_{\geq 1} :  \\
    (n,2) = 1
  }
}
\frac{|\lambda_\varphi(n)|^2}{n^{s}}
=
\frac{  \Gamma(s+1)}{(4 \pi)^{s+1}}
\zeta_{23}(s + 1)
\frac{L^{(S)}(\ad \varphi,s)\zeta^{(S)}(s)}{\zeta^{(S)}(2 s)}
\]
with $S := \{2, 23\}$
and $\zeta_p(s) := (1 - p^{-s})^{-1}, \zeta^{(S)}(s) := \zeta(s) / \prod_{p \in S} \zeta_p(s)$.
Taking residues
and using that $L_{23}(\ad \varphi,s) = \zeta_{23}(s+1)$,
we obtain
$\|\Phi_\varphi\|^2
=
\kappa_1
\iota_{\varphi}$
with
\begin{equation}\label{eq:kappa-1-defn}
  \kappa_1 :=
  \frac{
    1
  }
  {
    (4 \pi)^2
    \zeta^{(S)}(2)
    \zeta_2(1)
    \zeta_{23}(1)
  }.
\end{equation}
Define $\kappa_0 > 0$
by requiring that $\kappa_1 = \kappa_0^{-2} 2^{-2}$.

\section{Definitions of some theta functions}
\label{sec:defn-ternary-thetas}
The purpose of this section is to define some specific
weight $3/2$ cuspidal theta functions $h_k$ ($k=1,2$)
belonging to the Shintani--Waldspurger lifts of the
automorphic forms $\Psi_k : \mathbf{Y} \rightarrow \mathbb{C}$
defined in \S\ref{sec-1}.
The precise definition of the $h_k$ is
not immediately enlightening;
it is the output of local computations
to be discussed later.

Recall that for each element $E$ of the three-element set
$\mathbf{Y} = \Gamma \backslash G / K = \mathbf{X} / K$
we have defined a representative group element $g_E \in G$
and a maximal order $R_E$ (see \S\ref{sec:measures}).
Set $S_E := \mathbb{Z} + 2 R_E$.
Recall the reduced norm and trace
$\tr, \nr : R_E, S_E \rightarrow \mathbb{Z}$
and the trace zero subgroups $R_E^0, S_E^0$.
The latter are rank three lattices
which we regard as ternary quadratic forms with respect to
$\nr$.
Because $B$ splits at $2$, 
there are isomorphisms
$R_E \otimes_{\mathbb{Z}} \mathbb{Z}_2 \xrightarrow{\cong }
M_2(\mathbb{Z}_2)$
taking
$(\tr, \nr)$ to $(\tr,\det)$ and 
$S_E^0 \otimes_{\mathbb{Z}} \mathbb{Z}_2$ to
$\left\{ \begin{pmatrix}
    a & 2 b \\
    2c & - a
  \end{pmatrix} : a, b, c \in \mathbb{Z}_2 \right\}$.
There are thus
isomorphisms $\ell_E : \mathbb{Z}^3 \rightarrow   S_E^0$
so that the quadratic forms
$Q_E : \mathbb{Z}^3 \rightarrow \mathbb{Z}$,
$Q_E(a,b,c) := \nr(\ell_{E}(a,b,c))$
satisfy
$Q_E(a,b,c) \equiv \det (\begin{pmatrix}
  a & 2 b \\
  2 c & -a
\end{pmatrix})$
mod $4$.
For each $E \in \mathbf{Y}$, define quadratic characters
$\chi_1^E, \chi_2^E, \chi_3^E : S_E^0/2 S_E^0 \rightarrow \{\pm
1\}$
by writing a given element $\beta \in S_E^0$ in coordinates
$\beta = \ell_{E}(a,b,c)$ and setting
$\chi_1^{E}(\beta), \chi_2^{E}(\beta), \chi_3^{E}(\beta) :=
(-1)^{b + c}, (-1)^{a+c}, (-1)^{a+b}$;
some coordinate-free definitions of the $\chi_i^E$ are recorded
in \S\ref{sec:some-quadratic-characters}.
For $k=1,2$, we define
the theta function $h_k$ by
the equivalent formulas
\begin{align}\label{eqn:ternary-cusp}
  h_k(16 z)
  &:=
    \kappa_0
    y^{3/4}
    \sum_{E \in \mathbf{Y}}
    \frac{\Psi_k(E)}{w_E}
    \sum_{\beta \in S_E^0}
    \sum_{i =1,2,3} \chi_i^E(\beta) q^{\nr(\beta)} \\
  &= \nonumber
    \kappa_0
    y^{3/4}
    \sum_{E \in \mathbf{Y}}
    \frac{\Psi_k(E)}{w_{E}}
    \sum_{a,b,c \in \mathbb{Z}}
    \{(-1)^{b+c} + (-1)^{a+c} + (-1)^{a+b}\}
    q^{Q_E(a,b,c)} \\
  &= \nonumber
    \kappa_0
    y^{3/4}
    \sum_{D \geq 0}
    \mu_D(\Psi_k) q^D
\end{align}
where in the final expression,
$\mu_D$ is the measure on $\mathbf{Y}$ given by
\begin{align*}
\mu_D(E)
&:=
w_E^{-1}
\sum_{\beta \in S_E^0 : \nr(\beta) = D}
  \sum_{i=1,2,3} \chi_i^E(\beta)
  \\
  &= w_{E}^{-1}
\sum_{(a,b,c) \in \mathbb{Z}^3 : Q_E(a,b,c) = D}
\{
(-1)^{b+c}
+
(-1)^{a+c}
+
(-1)^{a+b}
    \}.
\end{align*}
We record
for future reference an equivalent definition of the $h_k$.
Let $E \in \mathbf{Y}$ correspond to the identity coset,
so that $R_E = R$ and $S_E = S := \mathbb{Z} + 2 R$.
We thereby obtain characters $\chi_i := \chi_i^E : S^0
\rightarrow \{\pm 1\}$.
They extend by continuity to $S_2^0$
and then extend by zero to elements
$\chi_i$ of the Schwartz--Bruhat space $\mathcal{S}(B_2^0)$.
Define $\phi'' \in \mathcal{S}(B_2^0)$
by
\begin{equation}\label{eq:defn-of-phi-double-prime}
  \phi''(\beta) := 2^{-3} \kappa_0 \sum_{i=1,2,3} \chi_i(4 \beta).
\end{equation}
Define the theta kernel
$\theta '' : \mathbb{H} \times \mathbf{X} \rightarrow \mathbb{C}$
by
\[
\theta ''(z,g) := y^{3/4}
\sum_{\beta \in R[1/2]^0}
\phi ''(g^{-1} \beta g)
q^{\nr(\beta)}.
\]
It is modular of weight $3/2$ in the first variable.
Recalling our measure normalizations (\S\ref{sec:measures}),
we find easily that
\[
h_k(z) = \int_{g \in \mathbf{X}} \Psi_k(g) \theta ''(z,g).
\]

We refer to \cite[\S12]{MR894322} for a lucid discussion
of the basic properties of
some theta functions
defined by analogy to the $h_k$ but using
the simpler measures $\mu_D^0(E) := w_E^{-1} \#\{\beta \in S_E^0 : \nr(\beta) = D\}$
instead of the $\mu_D$.
% somethe theta functions
% defined by analogy to the $h_k$ but using $\mu_D^0$
% instead of $\mu_D$.

\begin{comment}
\subsection*{Numerics}
{\small
We record some numerics,
carried out with the help of SAGE \cite{sage2015},
for concreteness and illustration. They are irrelevant
to our proofs.
There is a labeling
$\mathbf{Y} = \{E_1, E_2, E_3\}$
with\footnote{
  The corresponding supersingular $j$-invariants are $1728, -4,
  0$.
}
$(w_{E_1}, w_{E_2}, w_{E_3}) = (2,1,3)$
so that for each $k \in \{1,2\}$
there is a normalizing scalar $c_k \in \mathbb{R}$ so that
\[
(c_k \Psi_k(E_1),c_k \Psi_k(E_2),c_k \Psi_k(E_3))
= (2, -1- \phi_k, 1 + \phi_k)
\]
where
$\{\phi_1, \phi_2\} =
\{\frac{1 \pm \sqrt{5}}{2}\}
\approx
\{1.618, 0.618\}$ are the roots of $\phi^2 - \phi -
1 = 0$.
We may choose the
isomorphisms $\ell_E : \mathbb{Z}^3 \rightarrow   S_E^0$
so that the quadratic forms
$Q_E$
are given by
\begin{align*}
  Q_{E_1}(a,b,c) &= 23 a^2 + 4 b^2 - 4 b c + 2 4 c^2, \\
  Q_{E_2}(a,b,c) &= 23 a^2 + 12 b^2 - 4 b c + 8 c^2, \\
  Q_{E_3}(a,b,c) &= 3 a^2 + 8 a b + 8 a c + 3 6 b^2 - 20 b c + 36 c^2.
  % Q_2(a,b,c) &= 8a^2 - 4 a b + 1 2 b^2 + 2 3c^2, \\
  % Q_3(a,b,c) &= 3 a^2 + 2 a b + 2 a c + 3 1 b^2 - 30 b c + 3 1 c^2.
\end{align*}
}
\end{comment}

\section{Comparison with arithmetic variance}
\label{sec:appl-comp}
% In this section
% we formulate and indicateTheorem
% \ref{thm:special-main}.

% This section is irrelevant to the logical purposes of the paper,
% but motivates some ideas relevant for the proof of Theorem
% \ref{thm:special-main}.This section is irrelevant to the logical purposes of the paper,
% but motivates some ideas relevant for the proof of 

% This section may be regarded as optional for the logical
% purposes of the paper.
% We include it to motivate
% some ideas relevant for the
% proof of Theorem \ref{thm:special-main}.
% Some readers
% may wish to skip directly to
% \S\ref{sec:reduction-proof}.

% We explain this point not only for its
% interest, but because it provides a precious opportunity to
% introduce and motivate
Luo--Rudnick--Sarnak \cite{MR2504745} determined the limiting
variance of the family of measures on
$\SL_2(\mathbb{Z}) \backslash \SL_2(\mathbb{R})$ defined by
closed geodesics ordered by discriminant.  They termed this the
\emph{arithmetic variance} (see also
\cite{MR2776068,MR2890490}).
It turns out to be remarkably
close to the limiting quantum variance of microlocal lifts (see
\cite[Remark 2]{2013arXiv1303.6972S}).

The analogue in our
setting concerns representation numbers of ternary quadratic
forms.
Recall the definition of the constant $\kappa_0$ from
\S\ref{sec-1-3-7}, that of the $2 \times 2$ matrix $V^\infty$
from Theorem \ref{thm:special-main}, and that of the measures
$\mu_D$ from \S\ref{sec:appl-comp}.
\begin{theorem}[Arithmetic variance]\label{thm:arith-var}
  For $k,l \in \{1,2\}$,
  \[
  \lim_{x \rightarrow \infty}
  \frac{1}{x}
  \sum_{0 < D < x}
  \frac{\mu_D(\Psi_k)}{D^{1/4}} \frac{\mu_D(\Psi_l)}{D^{1/4}}
  =
  \frac{2}{\kappa_0^2 (4 \pi)^{-3/2} \Gamma(3/2)}  V_\infty^{k l}.
  \]
\end{theorem}
Despite the superficial similarity that both involve
limits of quadratic sums of measures on $\mathbf{Y}$, Theorem \ref{thm:arith-var} is
much simpler than Theorem \ref{thm:special-main}: it follows as
in \cite{MR2504745,MR2776068,MR2890490} from the Rankin--Selberg
method applied to $h_k$
followed by a Tauberian argument
and the inner product calculation
\begin{equation}\label{eq:inner-product-hkhl}
  \langle h_k, h_l \rangle = 2 V_{\infty}^{k l}
\end{equation}
to be discussed shortly.
The factor $(4 \pi)^{-3/2} \Gamma(3/2)$
arises as 
$\int_{y=0}^{\infty}
|y^{3/4} e^{- 2 \pi y}|^2
\, \frac{d y}{y}$.

The ``coincidence'' that the same limiting matrix
$V_\infty$
appears
in Theorems
\ref{thm:special-main}
and \ref{thm:arith-var}
% should be compared with
%  \cite[\S1.4.6]{MR2504745}.
may be understood as
an instance of the correspondence principle
(compare with \cite[\S1.4.6]{MR2504745}).

\section{Reduction of the proof of the main result\label{sec:reduction-proof}}
\label{sec-1-10}
We now reduce the proof of Theorem \ref{thm:special-main}
to that of some independent assertions whose proofs
may be studied in any order.
Retain
the notation and conventions of \S\ref{sec:appl-comp},
particularly the definition \eqref{eqn:ternary-cusp}
of $h_k$.
Define the following Jacobi theta function, which is modular of weight $1/2$:
\[
\theta(z) := y^{1/4} \sum_{m \in \mathbb{Z}: \gcd(m,2)=1} q^{m^2}.
\]

\subsection{The reduction}\label{sec:the-reduction}
To prove Theorem \ref{thm:special-main},
it suffices to show more precisely
for $k,l \in \{1,2\}$
and large enough $N$
that
\begin{align}\label{eq:sketch-1}
    2^N V_{N}^{kl}
    &=
      \int_{\Gamma' \backslash \mathbb{H}}
      \theta(z)
      h_k(2^{2N} z)
      \overline{
      \theta(z)  h_l(2^{2N} z)
      }
      \, d \nu(z)
    \\
    \label{eq:sketch-2}
    &=
      \int_{\Gamma' \backslash \mathbb{H}}
      |\theta|^2(z)
      h_k \overline{h_l}(2^{2N} z)
      \, d \nu(z)
    \\
    \label{eq:sketch-3}
    &=
      \int_{\Gamma' \backslash \mathbb{H}}
      |\theta|^2 \, d \nu
      \int_{\Gamma' \backslash \mathbb{H}}
      h_k \overline{h_l} \, d \nu
      + O(N 2^{-N})
    \\
    \label{eq:sketch-4}
    &=
      V_{\infty}^{kl} + O(N 2^{-N})
\end{align}
where
$\Gamma'$
denotes
a sufficiently small
congruence subgroup of $\SL_2(\mathbb{Z})$.
(Recall from \S\ref{sec:some-conventions}
that $d \nu(z) = d \nu_{\Gamma '}(z)$ is always a probability measure.)
The above steps
encapsulate the heart of our method.
The novelty lies  primarily in
the algebraic input \eqref{eq:sketch-1}, the analytic input \eqref{eq:sketch-3}, and the overall framework of the argument.

\subsection{Inner products of theta lifts}\label{sec:inner-products-theta}
We first discuss the easy steps of 
\S\ref{sec:the-reduction},
or more precisely,
those that follow readily from
extensive machinery developed by others.
The second step
\eqref{eq:sketch-2} is trivial.
The fourth
step \eqref{eq:sketch-4} holds in the more precise form
\begin{equation}\label{eqn:sketch-4-precise}
  \|\theta\|^2
  \langle h_k, h_l \rangle
  = V_\infty^{k l}.
\end{equation}
The identities \eqref{eq:inner-product-hkhl} and
\eqref{eqn:sketch-4-precise}
are equivalent because $\|\theta\|^2 = 1/2$ (see \S\ref{sec-3-10}, \eqref{eq:theta-l2-norm}).
The proof of \eqref{eqn:sketch-4-precise}
divides according to whether $k = l$ or not.  If
$k \neq l$, so that $V_\infty^{k l} := 0$, the strong
multiplicity one theorem on $\mathbf{Y}$ furnishes an odd prime
$p$ for which $\lambda_{\Psi_k}(p) \neq
\lambda_{\Psi_l}(p)$.\footnote{
  SAGE confirms that one may take $p=3$.
}
The local
data defining the lift $\Psi_k \mapsto h_k$ is unramified at
$p$, which is known
by Eichler's commutation relations\footnote{
  see for instance \cite[Prop 12.10]{MR894322}
}
to imply that
$h_k$ is an eigenfunction of
Shimura's $T_{p^2}$ Hecke
operator with eigenvalue proportional to $\lambda_{\Psi_k}(p)$.  Since $T_{p^2}$
is self-adjoint for the Petersson inner product, the vanishing
$\langle h_k, h_l \rangle = 0$ follows.  This argument parallels
the verification
in
\cite{MR2103474,MR2651907,2013arXiv1303.6972S} of 
symmetry properties of the limiting bilinear form arising from
the off-diagonal Kuznetsov terms.
Both arguments require a sort of
``fundamental lemma,'' given here by the unramified case of the
local theta correspondence.

The $k=l$ case of \eqref{eqn:sketch-4-precise} specializes an
inner product formula whose paradigm was introduced by
Rallis \cite{MR743016}.
It is now known in great generality
thanks to the work of several authors (see
\cite{2012arXiv1207.4709T, MR2837015} and references).
See \S\ref{sec-7} for further discussion.
% mentioned in \S\ref{sec-1-8}.

\subsection{The algebraic input}\label{sec:alg-input}
We turn to the more difficult steps of
\S\ref{sec:the-reduction}.
The identity \eqref{eq:sketch-1},
to be proved in
\S\ref{sec-2},
may be understood roughly as the synthesis of
\begin{enumerate}
\item the Parseval formula on
$\Gamma' \backslash \mathbb{H}$;
\item  an explicit seesaw identity
  proved using a weighted pretrace formula;
\item the local computation of the pushforward of a theta
  kernel,
  proved here using geometric arguments involving the Bruhat--Tits tree.
\end{enumerate}

\subsection{The analytic input}\label{sec:intro-analytic-input}
The estimate \eqref{eq:sketch-3}
should be compared
with a standard consequence of the spectral theorem: for
\emph{square-integrable} automorphic functions $f_1, f_2$,
\begin{equation}\label{eq:spectral-motivation}
  \int_{\Gamma' \backslash \mathbb{H}}
  f_1(z)
  f_2(2^{2N} z)
  \, d \nu(z)
  =
  \int_{\Gamma' \backslash \mathbb{H}}
  f_1 \, d \nu
  \int_{\Gamma' \backslash \mathbb{H}}
  f_2 \, d \nu
  + O(N 2^{-(1 - 2 \vartheta) N} \|f_1\| \|f_2\|)
\end{equation}
with $\vartheta \in [0,7/64]$ the best known bound
for the Hecke eigenvalues of Maass cusp forms (and unitary Eisenstein series)
at the prime $2$.
The hypotheses of \eqref{eq:spectral-motivation} do not apply to
the setup of \eqref{eq:sketch-3} because
$|\theta|^2 \notin L^2$.
This problem is
addressed by a regularized
spectral decomposition of $|\theta|^2$
which gives the stronger
estimate \eqref{eq:sketch-3}
-- like
\eqref{eq:spectral-motivation} but with ``$\vartheta := 0$'' --
by showing that $|\theta|^2$ is orthogonal to every cusp form.
See \S\ref{sec:analytic-input}
and \cite{nelson-theta-squared}.

\section{Further remarks}
\label{sec-1-11}

\subsection{}
The proof by \cite{MR2504745,2013arXiv1303.6972S} that the
arithmetic and quantum variance on
the modular curve
$\SL_2(\mathbb{Z}) \backslash \SL_2(\mathbb{R})$ are related
reduces to an explicit evaluation of each obtained by very
different means; see \cite[Remark 2]{2013arXiv1303.6972S}.  By
contrast, the first three steps
of \S\ref{sec:the-reduction}
provide a direct mechanism linking their analogues here.

\subsection{}\label{sec:rmk-weights}
Weights like the $\iota_{\varphi}$ (see \S\ref{sec-1-3-7})
arose in the works of Luo--Sarnak--Zhao
from an application of the Petersson/Kuznetsov formulas.  We do
not use such formulas here.  The weights arise in the
proofs of those formulas for the same reason they arise in our
treatment; see \S\ref{sec-2}.
The weights $\iota_{\varphi}$ could be removed
in our treatment
by the technique of
\cite{2013arXiv1303.6972S},
but the statements and proofs of our results
are simplified by retaining them.

\subsection{}
The spectral identity resulting from the
rearrangement
\eqref{eq:sketch-2}
may be understood, in the language of Reznikov
\cite{MR2373356}, as arising from the strong Gelfand
configuration formed by the various diagonal embeddings of $\Mp_2 \times \Mp_2$
inside $\Mp_2 \times \Mp_2 \times \Mp_2 \times \Mp_2$,
but for one caveat:
local multiplicity one fails
for some of the trilinear functionals on $\Mp_2$ that we consider (implicitly).
Compare with
\cite[\S1.1.3]{michel-2009}.

\subsection{}
\label{sec-1-5}
The error bound
in part (iii) of Theorem \ref{thm:special-main}
may be written $O( |\mathcal{F}_N|^{-1/2}\log|\mathcal{F}_N|)$.
The analogous bounds
obtained by Luo--Sarnak--Zhao (see \cite[Thm 1]{MR2103474}, \cite[Thm 2]{MR2651907},
\cite[\S4]{2013arXiv1303.6972S})
are
$O(|\mathcal{F}|^{-1/4+\eps})$,
roughly half as strong in the exponent as that obtained
here.
The present improvement seems to be a feature of the method rather than
of the specific aspect considered.  
The quantitative strength of our method is of secondary importance;
what matters most is that we have  proven qualitative assertions as in  \eqref{eq:most-significant-assertion}
in a non-split setting for the first time.

\subsection{}
\label{sec-optimal-error-term}
The error bound in part (iii) of Theorem \ref{thm:special-main}
is not optimal: by taking into account
that the
Ramanujan-type bound
$2^{N i t} + 2^{(N-2) i t} + \dotsb + 2^{- N i t} = O(N)$ for
the $2^{2 N}$th normalized Hecke eigenvalue of the unitary
Eisenstein series is rarely sharp, one can refine
$N 2^{-N}$ down to $2^{-N}$.  In particular,
\eqref{eq:diag-estimate} and \eqref{eq:offdiag-estimate} should hold
with errors $O(|\mathcal{F}_N|^{-1/2})$;
see
\S\ref{sec:remark-optimality}
for some details.
In principle, our
method permits determination of
``secondary main terms''
$V_{k l}^{\infty 2}$ for which
\[
  2^N V_{k l}^N
  =
  V_{k l}^\infty
  +
  2^{-N}
  V_{k l}^{\infty 2}
  + O(N^{-1} 2^{-N}),
\]
or perhaps even asymptotic expansions
up to $O(N^{-A} 2^{-N})$ and beyond; such refinements
are well beyond the scope of this paper.
The quantities
$V_{k l}^{\infty 2}$ should admit numerical evaluation.
We expect (but have not checked) that $V_{k l}^{\infty 2} \neq
0$
for all $k,l$;
if so, then the estimate
$2^N V_{k l}^N = V_{k l}^{\infty} + O(2^{-N})$
is best possible
in that it cannot be improved to
$2^N V_{k l}^N = V_{k l}^{\infty} + o(2^{-N})$.

\subsection{}
It would be interesting to
\begin{enumerate}
\item  understand any sense in which Theorem
  \ref{thm:special-main} holds after passing to algebraic
  parts modulo a suitable prime $\lambda$ of
  $\overline{\mathbb{Q}}$ (cf. \cite{MR2215136, MR2501296, MR1896237});
\item consider $(p,23)$ for an odd prime $p$ instead of $(2,23)$;
\item shrink the family
  $\mathcal{F}_N$ as in \cite{luo-sarnak-mass},
  perhaps using \cite[Thms 17 and 33]{nelson-padic-que};
\item implement the refinements
  suggested in \S\ref{sec-optimal-error-term},
  \S\ref{sec:remark-optimality}.
\end{enumerate}

\section{Proof of the algebraic input}
\label{sec-2}
We now prove \eqref{eq:sketch-1}.
Our argument is inspired by those of Eichler \cite{MR0080768} and
Shimizu \cite{MR0333081}
as well as
Gross's formulation \cite[(12.13)]{MR894322} of Eichler's trace
formula.

\subsection{Weighted pretrace formula}
\label{sec-2-1}
A convolution kernel $f \in C_c^\infty(G)$
acts on an automorphic form $\varphi \in \mathcal{A}(\mathbf{X})$
by the regular representation $\rho_{\reg}(f) \varphi(x) := \int_{g \in G} \varphi(x g)
f(g)
= \int_{y \in \mathbf{X}} \varphi(y) k_f(x,y)$
with kernel $k_f$ having the geometric and spectral expansions
\[
\sum_{\gamma \in \Gamma} f(x^{-1} \gamma y)
= k_f(x,y) =
\sum_{\varphi \in \mathcal{B}(\mathcal{A}(\mathbf{X}))}
\overline{\varphi(x)} (\rho_{\reg}(f) \varphi)(y).
\]
For nonnegative integers $N_1, N_2$,
denote by $e_{-N_1..N_2}$
the multiple of the characteristic
function
of $K_{-N_1..N_2}$
for which $\int_G e_{-N_1..N_2} = 1$.
Thus
$\rho_{\reg}(e_{-N_1..N_2})$ defines the orthogonal projection onto
$\mathcal{A}(\mathbf{X})^{K_{-N_1..N_2}}$.
Set
\[
  f := e_{{-N..N}}
  -
  e_{{-N+1..N}} - e_{{-N..N-1}}
  + e_{{-N+1..N-1}} \in C_c^\infty(G).
\]
It was shown in \cite{nelson-newform-isolation}
that for $N \geq 2$,  $\rho_{\reg}(f)$ defines
the orthogonal projection onto the span of
the orthonormal set $\mathcal{F}_N$ defined in \S\ref{sec:famil-balanc-newv}.
Thus
\[
  \sum_{\varphi \in \mathcal{F}_N}
  \overline{\varphi}(x) \varphi(y) = \sum_{\gamma \in \Gamma} f(x^{-1} \gamma y).
\]
For an odd natural number $n$,
we apply the Hecke operator $T_n$
to both sides and restrict to $x = y =: g$, giving
\[
  \sum _{\varphi \in \mathcal{F}_N}
  \sqrt{n} \lambda_\varphi(n) |\varphi|^2(g)
  = \sum_{\gamma \in M_n} f(g^{-1} \gamma g).
\]
We fix $k \in \{1,2\}$ and integrate against $\Psi_k$
to obtain\footnote{
  The works \cite{MR633667, MR871663, MR1671637, MR2783966} also
  consider
  such weighted trace fromulas.
}
\begin{equation}\label{eq:basic-trace-kernel-identity}
    \sum _{\varphi \in \mathcal{F}_N}
  \sqrt{n} \lambda_\varphi(n)
  \mu_\varphi(\Psi_k)
  =
  \int_{g \in \mathbf{X}}
  \Psi_k(g)
  \sum_{\gamma \in M_n} f(g^{-1} \gamma g).
\end{equation}

\subsection{Introduction of theta functions}\label{sec:intr-theta-funct}
The discussion thus far has been general:
it applies with inessential modification
to any compact quotient $\Gamma \backslash G$
involving any group $G$.
Recalling now that $\Gamma \backslash G$ arose from the multiplicative
structure on a quaternion algebra $B$,
we prepare to exploit the \emph{additive} structure
of $B$.
Set $B_2 := B \otimes_{\mathbb{Q}} \mathbb{Q}_2 \cong
M_2(\mathbb{Q}_2)$.
Denote by $\mathcal{S}(B_2)$ the Schwartz--Bruhat space,
consisting of locally constant compactly supported functions.
Define  $\phi \in \mathcal{S}(B_2)$
as follows:
For nonnegative integers $N_1, N_2$,
denote by $e'_{-N_1..N_2}$
the multiple of the characteristic
function
of the $\mathbb{Z}_2$-order
\[
\begin{pmatrix}
  \mathbb{Z}_2  & 2^{N_1} \mathbb{Z}_2 \\
  2^{N_2} \mathbb{Z}_2 & \mathbb{Z}_2
\end{pmatrix}
\]
for which
$e'_{-N_1..N_2} |_{\nr^{-1}(\mathbb{Z}_2^\times)} =
e_{-N_1..N_2}$.
Thus, e.g., $e_{0..0}'(1) = 1/2$, $e_{-N..N}'(1) = 3 \cdot 2^{2N - 2}$.
Set
\[
  \phi :=
e'_{{-N..N}}
  -
  e'_{{-N+1..N}} - e'_{{-N..N-1}}
  + e'_{{-N+1..N-1}} \in \mathcal{S}(B_2).
\]
Thus $\phi$ is to orders as $f$
is to their unit groups.\footnote{
  The choice
  $\phi := f$ would also work;
  the indicated choice of $\phi$
  turns out to be computationally convenient.
  We remark that in an adelic setting, the analogous passage
  at almost all unramified places
  from unit groups to orders 
  is responsible for the harmonic weights
  $\iota_{\varphi}$.
}
Observe that $\nr(\supp(f)) \subseteq \mathbb{Z}_2^\times$
and that $n \mathbb{Z}[1/2]^\times \cap \mathbb{Z}_2^\times = \{n\}$.
Thus if $\gamma \in M_n$ satisfies $f(g^{-1} \gamma g) \neq 0$,
then $\nr(\gamma) = n$.
Moreover, $f(g^{-1} \gamma g) = \phi(g^{-1} \gamma g)$.
Therefore
\begin{equation}\label{eq:trace-vs-theta-kernels}
  \sum_{\gamma \in M_n} f(g^{-1} \gamma g) = \sum_{\alpha
    \in R[1/2] \cap \nr^{-1}(n)}
  \phi(g^{-1} \alpha g).
\end{equation}
Summing \eqref{eq:basic-trace-kernel-identity}, \eqref{eq:trace-vs-theta-kernels}
over odd natural numbers $n$ gives
for $z \in \mathbb{H}, q := e^{2 \pi i z}$
with\footnote{
  $\Theta$ defines a variant of
  the diagonal restriction of the Eichler/Shimizu theta kernel.
}
\[
\Theta(\phi,z,g,g)
:=
y
\sum _{\alpha \in R[1/2]}
\phi (g^{-1} \alpha g)
1_{\mathbb{Z}_2^\times}(\nr(\alpha))
q^{\nr(\alpha)}
\]
and $\Phi_{\varphi}$ as in \S\ref{sec:eichl-langl-lifts}
that
\begin{equation}\label{eq:explicated-seesaw}
  \sum_{\varphi \in \mathcal{F}_N} \mu_\varphi(\Psi_k)
\Phi_\varphi(z)
=
\int_{g \in \mathbf{X}}
\Psi_k(g)
\Theta(\phi,z,g,g).
\end{equation}

\subsection{Pushforward of a theta kernel}\label{sec:pushf-theta-kern}
The function $\Psi_k$ is right $K$-invariant,
so the RHS of \eqref{eq:explicated-seesaw}
is unchanged by replacing $\phi$ with its average
$\phi^K \in \mathcal{S}(B_2)$
defined by
\[
\phi^K(b) := \frac{1}{\vol(K)}
\int_{s \in K}
\phi(s^{-1} b s).
\]
\begin{proposition}\label{prop:theta-kernel-pushforward}
  Let $\phi''$ be as in \S\ref{sec:defn-ternary-thetas}.
  For $m \in \mathbb{Q}_2$ and $\beta \in B_2^0$,
  one has
  $\phi^K(m + \beta) =  \kappa_0^{-1} 2^{2 N-1}
  1_{\mathbb{Z}_2}(m)
  \phi ''(2^{-N} \beta)$.
\end{proposition}
The proof of Proposition \ref{prop:theta-kernel-pushforward}
is a local computation
of what might be called ``partial orbital integrals.''
We postpone it to \S\ref{sec:fluctuations}.
We see now using the orthogonal decomposition
$R[1/2] = \mathbb{Z}[1/2] \oplus R[1/2]^0$
that
(cf. \S\ref{sec:defn-ternary-thetas}, \S\ref{sec:reduction-proof})
\[
\Theta(\phi^K,z,g,g)
=
c_N
\theta(z)
\theta ''(2^{2 N} z, g),
\]
with
$c_N := \kappa_0^{-1}
2^{2 N - 1} (2^{2 N})^{-3/4}$,
hence upon integrating against $\Psi_k$ that
\begin{equation}\label{eq:refined-seesaw-identity}
  \sum _{\varphi \in \mathcal{F}_N}
\mu_\varphi(\Psi_k)
\Phi_\varphi(z)
= c_N
\theta(z) h_k(2^{2 N} z).
\end{equation}

\subsection{Parseval}\label{sec:parseval}
The multiplicity one theorem on $\mathbf{X}$
and the self-adjointness of the Petersson inner product
for the classical Hecke operators
implies that $\Phi_\varphi, \Phi_{\varphi '}$ are orthogonal for $\varphi \neq
\varphi '$.
Recall from \S\ref{sec-1-3-7} that $\|\Phi_\varphi\|^2
= \kappa_1 \iota_{\varphi}$.
For $k,l \in \{1,2\}$, we obtain
\[
\kappa_1
\sum _{\varphi \in \mathcal{F}_N}
\iota_{\varphi}
\mu_\varphi(\Psi_k)
\mu_\varphi(\Psi_l)
= |c_N|^2
\int_{z \in \Gamma ' \backslash \mathbb{H}}
\theta(z) h_k(2^{2 N} z)
\overline{\theta(z) h_l(2^{2 N} z)}
\, d \nu (z).
\]
Dividing through by $2^N \kappa_1$
and verifying that $|c_N|^2 = 2^N \kappa_1$,
we obtain \eqref{eq:sketch-1}.

\subsection{Remark}
Several authors\footnote{ See for instance \cite[\S10]{MR1096467},
  \cite[p230]{MR1292728}, \cite[Lemma 2]{watson-2008},
  \cite[\S5]{MR1628792}, \cite[\S5.4]{MR2249532},
  \cite[\S3.2]{MR2215136}.
  See also
  \cite[\S11]{MR2198222} and \cite{MR2783966}
  for further variants of \eqref{eq:explicated-seesaw}, \eqref{eq:refined-seesaw-identity}
  proved on \emph{split} quotients
  by different means.
} have established explicit seesaw
identities using strong multiplicity one and newvector theory on
$\mathbf{X}$ to write some unknown constant multiple of
$\varphi \otimes \overline{\varphi}$ as an explicit theta lift
$\int_{z \in \Gamma ' \backslash \mathbb{H}}
\overline{\Phi_{\varphi}}(z) \Theta(\phi,z,\cdot,\cdot) \, d
\nu(z)$ and then seesaw duality to determine the constant.  The
approach developed above differs in that it avoids direct
analysis of the theta lift from $\SL_2$ to $\O(B)$; this was
achieved by unfolding a small part of the proof of that case of
the global theta correspondence.  The present approach is far
more direct for our purposes because of subtleties arising from
oldforms; moreover, it generalizes to more
complicated families and to archimedean aspects.
% We postpone
% a systematic discussion to
% \cite{nelson-variance-II}.

% Fluctuations of fixed lines
\section{Computation of partial orbital integrals\label{sec:fluctuations}}
\label{sec-3}
In this section we carry out the local calculation (Proposition
\ref{prop:theta-kernel-pushforward}) postponed in
\S\ref{sec:pushf-theta-kern}.
The shape of this calculation
(specifically
the ``separation''
of the factors $1_{\mathbb{Z}_2}(m)$
and
$\theta ''(2^{-N} \beta)$ by the multiplicative
dilation $2^N$)
is crucial to the success of the method.
On the other hand,
it may be instructive to note that
Proposition
\ref{prop:theta-kernel-pushforward}, and hence the contents
of this section,
become unnecessary
if one is willing to settle
for a weaker and (much) less natural
variant of Theorem \ref{thm:special-main}
involving sums of the shape
\begin{equation}\label{eq:extra-averaging-weaker}
  \sum_{\pi \subseteq \mathcal{A}(\mathbf{X})}
  \iota_{\pi}
  \sum_{\varphi_1,\varphi_2 \in \mathcal{B}(\pi^{K[N]})}
  \mu_{\varphi_1}(\Psi_k)
  \mu_{\varphi_2}(\Psi_l),
\end{equation}
where $\pi$ traverses the irreducible submodules,
$\iota_\pi := L^{(2)}(\ad \pi,1)$
and $\mathcal{B}(\pi^{K[N]})$
is an orthonormal basis for the vectors invariant by the
principal congruence subgroup $K[N]$ of $K$ consisting
of those elements congruent modulo $N$ to a scalar.  One may
understand the purpose of this section as to reduce the average
within $\pi$ in \eqref{eq:extra-averaging-weaker} to an
individual newvector $\varphi \in \mathcal{F}_N \cap \pi$.

In fact,
we will give two complementary proofs of Proposition
\ref{prop:theta-kernel-pushforward}:
% although it
% is ``just
% a local
% computation,'' its correctness is essential for us, and we are not
% aware of analogous computations in the literature against which
% to compare it.
\begin{enumerate}
\item In this section, we record
a direct proof based on the matrix Fourier transform and
analysis of conjugacy classes in
$\GL_2(\mathbb{Z}_2)$; this first proof is
computationally involved, but concrete.
\item In \S\ref{sec:appendix-fluctuations}, we record a geometric proof involving the
Bruhat--Tits tree; that second proof is free of computational
difficulties, but requires
additional setup.
\end{enumerate}
% Reassuringly,
% both
% arguments give the same result.

% Recall the notation of \S\ref{sec:preliminaries}.
% Set $S := \mathbb{Z}
% + 2 R$.
% Retain the notation $\phi$ from \S\ref{sec:intr-theta-funct}.
\subsection{Reduction to an identity of functions
on a finite matrix ring}
Observe first that the identity in Proposition
\ref{prop:theta-kernel-pushforward} holds in the special case
$m = 1, \beta = 0$; indeed, both sides specialize to the same
nonzero quantity $3 \cdot 2^{2 N - 4}$.  It will thus suffice to
verify that identity up to an unspecified constant
multiple;
this purely technical reduction
frees us from worrying about
proportionality
factors in what follows.

Observe next that the identity in question is one of functions of
the variable $b = m + \beta \in B_2$ supported on
$b \in M_2(\mathbb{Z}_2)$ and invariant under translation by
$M_2(2^N \mathbb{Z}_2)$; it is thus equivalent
to an identity
between functions on $M_2(\mathfrak{o})$, where
$\mathfrak{o} := \mathbb{Z}_2 / 2^N \mathbb{Z}_2 \cong
\mathbb{Z} / 2^N \mathbb{Z}$.
Expanding the definitions,
we reduce to establishing the following
explicit identity of functions on $M_2(\mathfrak{o})$:
\begin{lemma}\label{lem:ugly-computation-partial-orbital-integral}
  Set $\varpi := 2 \hookrightarrow \mathfrak{o}$
  and $\mathfrak{p} := \varpi \mathfrak{o}$.
  For integers $m < 0 < m'$,
  let $\eta_{m..m'} : M_2(\mathfrak{o}) \rightarrow \mathbb{C}$
  denote the characteristic
  function
  of
  the subset
  $\begin{pmatrix}
    \mathfrak{o}^\times  & \mathfrak{p}^{-m} \\
    \mathfrak{p}^{m'} & \mathfrak{o}^\times 
  \end{pmatrix}$ of $M_2(\mathfrak{o})$.
  Define $\Phi^0 : M_2(\mathfrak{o}) \rightarrow \mathbb{C}$
  by \[\Phi^0 := \eta_{-N..N} - (1/2) \eta_{-N..N-1}
  - (1/2) \eta_{-N+1..N}
  + (1/4) \eta_{-N+1..N-1}.\]
  Define $\Phi : M_2(\mathfrak{o}) \rightarrow \mathbb{C}$
  by
  \[
  \Phi(x) := \sum_{g \in \GL_2(\mathfrak{o})}
  \Phi^0(g^{-1} x g).
  \]
  Let $\Phi' : M_2(\mathfrak{o}) \rightarrow \mathbb{C}$
  denote the function
  supported on elements of the form
  \[
  t =
  \begin{pmatrix}
    v + \varpi^{N-2} x & \varpi^{N-1} y \\
    \varpi^{N-1} z & v - \varpi^{N-2} x
  \end{pmatrix}
  \text{ with }
  v \in \mathfrak{o}^\times \text{ and }
  x,y,z \in \mathfrak{o}
  \]
  and given on such elements
  by
  \[\Phi'(t) := (-1)^{x+y} + (-1)^{y + z} + (-1)^{x+z}.\]
  Then the functions $\Phi, \Phi'$ are constant multiples of one another.
\end{lemma}

\subsection{Application of the Fourier transform}
To prove \eqref{lem:ugly-computation-partial-orbital-integral},
we use the Fourier transform on $M_2(\mathfrak{o})$.  Let
$\zeta$ be a primitive $2^N$th root of unity.  For
$a \in \mathfrak{o}$, the quantity $\zeta^a$ is well-defined.
For $f : M_2(\mathfrak{o}) \rightarrow \mathbb{C}$, define its
Fourier transform
$\mathcal{F} f : M_2(\mathfrak{o}) \rightarrow \mathbb{C}$ by
the formula
$\mathcal{F} f(x) := \sum_{y \in M_2(\mathfrak{o})} f(y)
\zeta^{(x,y)}$,
where
$(x,y) := \det(x + y) - \det(x) - \det(y) = x_{1 1} y_{2 2} +
x_{2 2} y_{1 1} - x_{1 2} y_{2 1} - x_{2 1} y_{1 2}$.
The Fourier transform is equivariant for conjugation by
$\GL_2(\mathfrak{o})$.  It is also injective, since it satisfies
an inversion formula.  We compute the Fourier transform of
$\Phi^0$ by applying the inclusion-exclusion identity (for
$-N \leq m \leq -1, 1 \leq m' \leq N$)
\[
\eta_{m..m'}
=
1 _{\begin{pmatrix}
    \mathfrak{o}  & \mathfrak{p}^{-m} \\
    \mathfrak{p}^{m'} & \mathfrak{o} 
  \end{pmatrix} 
}
-
1 _{\begin{pmatrix}
    \mathfrak{p}  & \mathfrak{p}^{-m} \\
    \mathfrak{p}^{m'} & \mathfrak{o} 
  \end{pmatrix} 
}
-
1 _{\begin{pmatrix}
    \mathfrak{o}  & \mathfrak{p}^{-m} \\
    \mathfrak{p}^{m'} & \mathfrak{p} 
  \end{pmatrix} 
}
+
1 _{\begin{pmatrix}
    \mathfrak{p}  & \mathfrak{p}^{-m} \\
    \mathfrak{p}^{m'} & \mathfrak{p} 
  \end{pmatrix} 
}
\]
followed by the Fourier identity
\[
\mathcal{F}
1 _{\begin{pmatrix}
    \mathfrak{p}^a & \mathfrak{p}^b \\
    \mathfrak{p}^c & \mathfrak{p}^d
  \end{pmatrix}}
=
2^{a+b+c+d}
1 _{\begin{pmatrix}
    \mathfrak{p}^{N-d} & \mathfrak{p}^{N-c} \\
    \mathfrak{p}^{N-b} & \mathfrak{p}^{N-a}
  \end{pmatrix}
}
\text{ for $0 \leq a,b,c,d \leq N$}
\]
% 1 _{\begin{pmatrix}
%   \mathfrak{a}  & \mathfrak{b}  \\
%   \mathfrak{c}  & \mathfrak{d} 
% \end{pmatrix}
% }
%   =
%   |\mathfrak{a}| \, |\mathfrak{b}| \,
%   |\mathfrak{c}| \,
%   |\mathfrak{d}| \,
%   1 _{\begin{pmatrix}
%     \ann(\mathfrak{d})  & \ann(\mathfrak{c})  \\
%     \ann(\mathfrak{b})  & \ann(\mathfrak{a})
%   \end{pmatrix}
% }
%   \]
%   for ideals $\mathfrak{a},\mathfrak{b},\mathfrak{c},\mathfrak{d}$
%   of $\mathfrak{o}$,
%   where $\ann(\mathfrak{a}) := \{x \in \mathfrak{o} : x
%   \mathfrak{a} = 0\}$,
%   so that $\ann(\mathfrak{p}^n) = \mathfrak{p}^{m-n}$
%   for $0 \leq n \leq m$.
to see that
\begin{equation}\label{eq:initial-formula-for-fourier-transform-of-Phi-in-messy-computation}
  2^{-2 N}
\mathcal{F} \Phi^0
=
\sigma_{N,N} - (1/2) \sigma_{N-1,N}
- (1/2) \sigma_{N,N-1}
+ (1/4) \sigma_{N-1,N-1},
\end{equation}
where
$\sigma_{n,n'}$
denotes the characteristic function
of
$S(n,n') := \begin{pmatrix}
  \mathfrak{p}^{n} & \mathfrak{o}^\times  \\
  \mathfrak{o}^\times  & \mathfrak{p}^{n'}
\end{pmatrix}$.
By a similar calculation,
we see that
\[
\mathcal{F} \Phi'
=
?\left(
  1_{\tr^{-1}(\mathfrak{p}^{N})}
  - \frac{1}{2} 
  1_{\tr^{-1}(\mathfrak{p}^{N-1})}
\right)
1_E
\]
for some unimportant scalar $?$,
where $1_{\tr^{-1}(\mathfrak{p}^{n})}$
denotes the characteristic function
of $\{x \in M_2(\mathfrak{o}) : \trace(x) \in \mathfrak{p}^n\}$
and $E := E_1 \sqcup E_2 \sqcup E_3$
with
$E_i := X_i + \mathfrak{p} M_2(\mathfrak{o})$
and
$X_1 := \left(
  \begin{smallmatrix}
    &1\\
    1&
  \end{smallmatrix}
\right),
X_2 := \left(
  \begin{smallmatrix}
    1&1\\
    &-1
  \end{smallmatrix}
\right),
X_3 := \left(
  \begin{smallmatrix}
    1&\\
    1&-1
  \end{smallmatrix}
\right)$.
% \[
% X_1 := \begin{pmatrix}
%   & 1 \\
%   1 & 
% \end{pmatrix},
% \quad 
% X_2 := \begin{pmatrix}
%   1 & 1 \\
%   & -1
% \end{pmatrix},
% \quad 
% X_3 := \begin{pmatrix}
%   1 &  \\
%   1 & -1
% \end{pmatrix}.
% \]
Set
$H := \{\GL_2(\mathfrak{o}) : g \equiv 1 (\mathfrak{p})\}$.
Since the $E_i$ are conjugate under
$\GL_2(\mathfrak{o})$
and $E_1 = S(1,1)$,
we reduce to verifying that
\begin{equation}\label{eq:identity-between-fourier-transforms-finite-rings-want-it}
  \sum_{g \in H}
  \mathcal{F} \Phi^0(g^{-1} x g)
  =
  ?
  \left(
    1_{\tr^{-1}(\mathfrak{p}^{N})}
    - \frac{1}{2} 
    1_{\tr^{-1}(\mathfrak{p}^{N-1})}
  \right)
  1_{S(1,1)}(x)
\end{equation}
for all $x \in M_2(\mathfrak{o})$
and some scalar $?$ not depending upon $x$.

\subsection{Summing over orbits}
Consider the map
$\kappa : S(1,1) \rightarrow \mathfrak{p} \times
\mathfrak{o}^\times$
given by $\kappa(X) := (\tr(X),\nr(X))$.
By Hensel's lemma
and an (omitted) orbit-stabilizer argument,
the map $\kappa$ is surjective, its fibers all have the same
cardinality,
and each fiber is an orbit for the conjugation action
of $H$.
It follows
for $-0 < n,n' \leq N$
that $\kappa|_{S(n..n')}$ surjects
onto $\mathfrak{p}^{\min(n,n')} \times \mathfrak{o}^\times$
with
fibers of equal cardinality
and $H$ acting transitively
on each fiber.
With the notation $n_0 := \min(n,n')$
and
$n_1 := \max(n,n')$,
we deduce that
$\sum_{g \in H}
1 _{g S(n,n') g^{-1}}$
is a constant multiple of
the characteristic
function of $S(1,1) \cap \tr^{-1}(\mathfrak{p}^{n_0})$.
Since $S(n,n')$ has
$2^{4 N-2 -n - n'}$ elements
and $S(1,1) \cap \tr^{-1}(\mathfrak{p}^{n_0})$
has $2^{4 N - 3 - n_0}$ elements,
we obtain
\[
\sum_{g \in H}
1_{g S(n,n') g^{-1}}
=
|H|
2^{1 - n_1}
1 _{S(1,1) \cap \tr^{-1}(\mathfrak{p}^{n_0})}.
\]
Substituting this into
\eqref{eq:initial-formula-for-fourier-transform-of-Phi-in-messy-computation},
we deduce that
the LHS of
\eqref{eq:identity-between-fourier-transforms-finite-rings-want-it}
is a constant multiple
of the value taken at $x$ by the function
% \[
% e_{G_1}
% (1_{S(m,m)}
% - \frac{1_{S(m-1,m)}}{2}
% - \frac{1_{S(m,m-1)}}{2}
% + \frac{1_{S(m-1,m-1)}}{4}
% )
% \]
\[
% 2^{1-N}
\left( 
  1_{\tr^{-1}(\mathfrak{p}^N)}
  -
  \frac{1}{2}
  1_{\tr^{-1}(\mathfrak{p}^{N-1})}
  -
  \frac{1}{2}
  1_{\tr^{-1}(\mathfrak{p}^{N-1})}
  +
  \frac{2}{4}
  1_{\tr^{-1}(\mathfrak{p}^{N-1})}
\right)
1_{S(1,1)}
\]
which simplifies to
the RHS
of \eqref{eq:identity-between-fourier-transforms-finite-rings-want-it}.

% \subsection{Remarks}
% Proposition
% \ref{prop:theta-kernel-pushforward}, and hence the contents
% of this section,
% become unnecessary
% if one is willing to settle
% for a weaker and less natural
% variant of Theorem \ref{thm:special-main}
% involving sums of the shape
% \begin{equation}\label{eq:extra-averaging-weaker}
%   \sum_{\pi \subseteq \mathcal{A}(\mathbf{X})}
%   \iota_{\pi}
%   \sum_{\varphi_1,\varphi_2 \in \mathcal{B}(\pi^{K[N]})}
%   \mu_{\varphi_1}(\Psi_k)
%   \mu_{\varphi_2}(\Psi_l),
% \end{equation}
% where $\pi$ traverses the irreducible submodules,
% $\iota_\pi := L^{(2)}(\ad \pi,1)$
% and $\mathcal{B}(\pi^{K[N]})$
% is an orthonormal basis for the vectors invariant by the
% principal congruence subgroup $K[N]$ of $K$ consisting
% of those elements congruent modulo $N$ to a scalar.  One may
% thus understand the purpose of this section as to reduce the average
% within $\pi$ in \eqref{eq:extra-averaging-weaker} to an
% individual newvector $\varphi \in \mathcal{F}_N \cap \pi$.

% \subsection{Proof by brute-force calculation}
% TODO: eliminate?

% We shall give a geometric proof of Propositions
% \ref{prop:theta-kernel-pushforward}
% which seems simpler
% and easier to check than our original proof,
% which involved
% the Fourier transform on $M_2(\mathbb{Q}_2)$
% as in 
% \S\ref{sec:classical-correlations} and the
% conjugacy classes of $\GL_2(\mathbb{Z}_2)$.

\section{Fluctuations of fixed
  lines\label{sec:appendix-fluctuations}}
The main purpose of this section is to record the geometric
proof of Proposition \ref{prop:theta-kernel-pushforward}
promised in \S\ref{sec:fluctuations}.  We also verify the mean
statistics \eqref{eq:linear-statistics-justification}
used to justify interpreting $V_N$ as a variance.

\subsection{Translation to a geometric problem}
Recall the notation of \S\ref{sec:preliminaries}.
Set $S := \mathbb{Z}
+ 2 R$.
Retain the notation $\phi$ from \S\ref{sec:intr-theta-funct}.

The function $\phi$ and its conjugates under $K$
are supported on $R_2 \cong M_2(\mathbb{Z}_2)$,
and the order $R$ is dense in $R_2$, so
our task reduces (for notational convenience) to determining $\phi^K(\alpha)$ for all $\alpha$ in $R$.
Recall that we have fixed an embedding
$R \hookrightarrow R_2  \cong M_2(\mathbb{Z}_2) = \End(\mathbb{Z}_2^2)$.
Denote by $E[2^\infty]$ the abelian group
$(\mathbb{Q}_2/\mathbb{Z}_2)^2$.
It is thus a module for $R$ under left multiplication.
It is a direct limit of the submodules $E[2^N] := (2^{-N}
\mathbb{Z}_2/\mathbb{Z}_2)^2$.

Our notation reflects that $R$ may be identified
with the endomorphism ring $\End(E)$
of a supersingular elliptic curve $E$ in characteristic $23$
(see \cite[\S2]{MR894322})
and the groups $E[2^N], E[2^\infty]$ defined above with the
corresponding
torsion subgroups.  We have found this perspective helpful in
forming intuition for the calculations to follow.

For a pair of nonnegative integers $N_1, N_2 \geq 0$,
denote by $\mathcal{L}_{N_1,N_2}$ the set of ordered pairs
$(C_1,C_2)$
consisting of cyclic subgroups $C_1, C_2 \leq E[2^\infty]$
of respective orders $2^{N_1}, 2^{N_2}$ satisfying $C_1 \cap C_2
= \{0\}$.
The group $K$ acts transitively on $\mathcal{L}_{N_1,N_2}$.  The
subgroup $K_{-N_1..N_2}$ is the stabilizer of some ``standard''
pair $(C_1,C_2)
\in \mathcal{L}_{N_1,N_2}$ tailored to the standard basis.
Recall from \S\ref{sec:measures} that the Haar on $G$ assigns
volume $2$ to $K$.
It follows for $\alpha \in R$ that
\[
\frac{1}{\vol(K)}
\int_{s \in K} e_{-N_1..N_2}'(s^{-1} \alpha s)
= (1/2)\Fix_{N_1,N_2}(\alpha),
\]
where $\Fix_{N_1,N_2} : R \rightarrow \mathbb{Z}_{\geq 0}$
is given
by
\begin{equation}\label{eq:defn-fix}
  \Fix_{N_1,N_2}(\alpha) := \# \{(C_1,C_2) \in
  \mathcal{L}_{N_1,N_2}:
  \alpha C_1 \leq C_1, \alpha C_2 \leq C_2\}.
\end{equation}
Thus for $N \geq 1$,
\begin{equation}\label{eq:geometric-interpretation-of-local-pushforward}
  \phi^K(\alpha)
  = (1/2)\Fix_{N,N}^\sharp(\alpha)
\end{equation}
where more generally for $N_1, N_2 \geq 1$,
\[
\Fix_{N_1,N_2}^\sharp
:=
\Fix_{N_1,N_2}
-
\Fix_{N_1-1,N_2}
-
\Fix_{N_1,N_2-1}
+
\Fix_{N_1-1,N_2-1}.
\]
Our task is thus equivalent to evaluating
the functions $\Fix_{N,N}^\sharp$:
\begin{proposition}\label{prop:local-pushforward}
  Let $N \geq 2$.
  Let $\alpha \in R$ with
  $\Fix_{N,N}^{\sharp}(\alpha) \neq 0$.
  Then $\alpha \in \mathbb{Z} \oplus 2^{N-2} S^0$.
  For $m \in \mathbb{Z}$
  and $\beta \in S^0$,
  \begin{equation}\label{eqn:local-pushforward}
    \Fix_{N,N}^\sharp(m + 2^{N - 2} \beta)
    =
    2^{2 N - 3}
    \sum_{i=1,2,3} \chi_i(\beta).
  \end{equation}
\end{proposition}
The proof is given below.
To see how
Proposition
\ref{prop:local-pushforward}
implies
Proposition \ref{prop:theta-kernel-pushforward},
write $\alpha = m + \beta$.
By
\eqref{eq:geometric-interpretation-of-local-pushforward}
and \eqref{eqn:local-pushforward},
we have $\phi^K(\alpha) = 2^{2 N - 4}
1_{\mathbb{Z}_2}(m)
\sum_{i=1,2,3}
\chi_i(2^{2 - N} \beta)$.
We conclude that
$\phi^K(\alpha) = \kappa_0^{-1}
2^{2N-1}
1_{\mathbb{Z}_2}(m) \phi ''(2^{-N} \beta)$
upon recalling the definition \eqref{eq:defn-of-phi-double-prime}
of $\phi ''$.

\subsection{The tree}
\label{sec:tree}
We record a realization of the Bruhat--Tits tree $\mathcal{T}$
of $\PGL_2(\mathbb{Q}_2)$
(see \cite{MR580949}, \cite{MR1954121}, \cite[\S1.2]{MR2729264})
relative to a basepoint
and fix some terminology.

\subsubsection{}
The \emph{vertices} of $\mathcal{T}$
are the cyclic subgroups $L$ of
$E[2^\infty] = (\mathbb{Q}_2/\mathbb{Z}_2)^2$.  Two vertices
$L, L'$ are connected by an \emph{edge} if one contains the
other with index $2$.  The undirected graph $\mathcal{T}$ is
then a $3$-regular tree.

\subsubsection{}
The trivial subgroup $\{0\}$ belongs
to $\mathcal{T}$; we call it the \emph{origin}.  For $n \geq 0$,
denote by $\mathcal{T}_n \subset \mathcal{T}$ the
set of vertices at distance $n$ from
the origin; these are the cyclic $2^n$-subgroups of
$E[2^\infty]$.
The nearest vertex
to the origin in the convex hull of a pair of vertices $L$ and
$L'$ is $L \cap L'$.  We call two vertices $L, L'$
\emph{independent} if $L \cap L' = \{0\}$, or equivalently, if
their convex hull contains the origin.  The convex hull of $L$
and the origin consists of those $L'$ for which $L' \leq L$.
The \emph{children} of a vertex $v \in \mathcal{T}_n$
are those $v' \in \mathcal{T}_{n+1}$ neighboring it;
the \emph{parent} of a vertex $v \in \mathcal{T}_n$, if $n > 0$,
is the vertex $v' \in \mathcal{T}_{n-1}$ neighboring it.

\subsubsection{}
The multiplicative monoid $\{\alpha \in R : (\nr(\alpha),2)=1\}$
acts in an evident way ($(\alpha, L) \mapsto \alpha L$) on $\mathcal{T}$
by isometries that fix the origin, hence stabilize each sphere
$\mathcal{T}_n$.
The set $\mathcal{L}_{N_1,N_2}$
consists of the ordered pairs of
independent vertices $v_1,v_2 \in \mathcal{T}_{N_1},
\mathcal{T}_{N_2}$;
for $\alpha \in R$ with $(\nr(\alpha),2)=1$,
the set of such pairs
fixed by $\alpha$
is convex and has cardinality
$\Fix_{N_1,N_2}(\alpha)$.

\subsubsection{}
More generally, the multiplicative monoid $R$ acts (by  $(\alpha, L) \mapsto \alpha L$) on
$\mathcal{T}$, although not by isometries in general.
Say that $\alpha \in R$ \emph{contracts} a
vertex $L$ if $\alpha L$ belongs to the convex hull of $L$ and
the origin, or equivalently, if $\alpha L \leq L$.
For $\alpha \in R$,
the set of ordered pairs
of independent vertices
$v_1, v_2 \in \mathcal{T}_{N_1}, \mathcal{T}_{N_2}$
contracted by $\alpha$ is convex
and has cardinality
$\Fix_{N_1,N_2}(\alpha)$.

\subsection{Some quadratic characters\label{sec:some-quadratic-characters}}
\label{sec:some-quadratic-chars}
We introduce a geometric interpretation
of the characters
$\chi_1, \chi_2, \chi_3 : S^0 \rightarrow \{\pm 1\}$
defined in \S\ref{sec:defn-ternary-thetas}
(see \S\ref{sec:four-analyt-interpr}
for a further Fourier-analytic interpretation).
Denote by $v_0$ the origin in $\mathcal{T}$ and
by $v_1,v_2,v_3$ its children.
For $\beta \in S^0$,
let $\alpha \in S$
be such that $(\nr(\alpha),2)=1$
and $\alpha - \beta \in \mathbb{Z}$.
Then $\alpha \equiv 1 \pod{2}$, so it fixes $v_1,v_2,v_3$
and hence permutes each of their children.
For $i \in \{1,2,3\}$, set $\eta_i(\beta) := 1$
if $\alpha$ fixes the children of $v_i$
and $\eta_i(\beta) := -1$ if it swaps them.
Define $\chi_i' := \eta_j \eta_k$
where
$\{i,j,k\} = \{1,2,3\}$.
\begin{lemma}\label{lem:triples-agree}
  The triples of characters $\chi_1,\chi_2,\chi_3$
and $\chi_1',\chi_2',\chi_3'$
coincide up to a permutation of indices.
\end{lemma}
\begin{proof}
  There
is a canonical conjugacy class of ring maps
$\rho : R \rightarrow M_2(\mathbb{Z}/4)$, given by the action
of $R$ on $E[4]$, under which elements $\beta \in S^0$ have
the form
\[
\rho(\beta) = \begin{pmatrix}
  a & 2 b \\
  2 c & -a
\end{pmatrix}.
\]
In that optic, 
$\chi_1(\beta), \chi_2(\beta),\chi_3(\beta)
:=
(-1)^{b+c},
(-1)^{a+c},
(-1)^{a+b}$.
Using a basis for $\mathbb{Z}_2^2$ compatible with $\rho$,
we may assume after relabeling indices
that $v_1,v_2,v_3$
correspond to the respective cyclic subgroups
$\tfrac{1}{2}(1,1)^t + \mathbb{Z}_2^2,
\tfrac{1}{2}(1,0)^t + \mathbb{Z}_2^2,
\tfrac{1}{2}(0,1)^t + \mathbb{Z}_2^2$
of $E[2]$.
Let $\beta \in S^0$.
Set
\[
\alpha := \begin{pmatrix}
  1 + 2 a & 2 b \\
  2 c & 1
\end{pmatrix} \in S.
\]
Then $(\nr(\alpha),2) = 1$ and $\alpha - \beta \in \mathbb{Z}$.
Using this choice of $\alpha$,
we compute directly that
$\eta_1(\beta), \eta_2(\beta), \eta_3(\beta) = (-1)^{a+b+c},
(-1)^c, (-1)^b$;
for example, 
$\eta_1(\beta) = +1$
iff $\alpha$ stabilizes
$\tfrac{1}{4} (1,1)^t
+ \mathbb{Z}_2^2$
iff $(\tfrac{a + b}{2}, \tfrac{c}{2})^t \in \tfrac{1}{4}(1,1)^t + \mathbb{Z}_2^2$
iff $a + b + c$ is even.
It follows as required that $\chi_i = \chi_i'$.
\end{proof}
We assume henceforth that $v_1,v_2,v_3$ have been
ordered so that $\chi_i = \chi_i'$.

\subsection{Combinatorial arguments\label{sec:local-pushforward}}
\label{sec-3-3}
\begin{lemma}\label{lem:fix-n-basic-properties}
  Let $N_1, N_2$ be nonnegative integers
  and $\alpha \in R$.
  \begin{enumerate}
  \item[(i)] For $t \in \mathbb{Z}$,
    one has $\Fix_{N_1,N_2}(t + \alpha) =
    \Fix_{N_1,N_2}(\alpha)$.
  \item[(ii)] If $N_1,N_2 \geq 2$, then
    $\Fix_{N_1,N_2}(2 \alpha) = 4 \Fix_{N_1-1,N_2-1}(\alpha)$.
  \end{enumerate}
\end{lemma}
\begin{proof}
  (i): Immediate from the definition.
  (ii):
  For  $N \geq 2$,
  an element
  $\alpha \in R$ contracts some $v \in \mathcal{T}^{N-1}$
  if and only if
  $2 \alpha$ contracts both (equivalently, either) of the children of $v$;
  each independent ordered pair 
  contributing to $\Fix_{N_1-1,N_2-1}(\alpha)$
  thus corresponds to $2^2$ independent ordered pairs
  contributing to $\Fix_{N_1,N_2}(2\alpha)$, and vice-versa.
\end{proof}
\begin{lemma}\label{lem:force-scalar-mod-2}
  Let $\alpha \in R$ with $(\nr(\alpha),2) = 1$.  Suppose
  $\Fix_{N_1,N_2}(\alpha) \neq 0$ for some $N_1,N_2 \geq 1$.
  Then $\alpha \in S$.
\end{lemma}
\begin{proof}
  The hypotheses imply that $\alpha$ fixes at least one pair
  of independent vertices not equal to the origin, hence
  fixes their convex hull, hence fixes some pair of
  children of the origin,
  hence fixes all three children,
  i.e., acts trivially on $E[2]$,
  and so is congruent to a scalar modulo $2$,
  as required.
\end{proof}
\begin{lemma}\label{lem:incl-excl-tricky}
  Let $\alpha \in R$.
  Suppose $\Fix_{N_1,N_2}^\sharp(\alpha) \neq 0$ for some $N_2, N_2 \geq 2$.
  Then $\alpha \in S$.
\end{lemma}
\begin{proof}
  If $(\nr(\alpha),2) = 1$, then the conclusion follows from
  Lemma \ref{lem:force-scalar-mod-2}, so suppose
  $\nr(\alpha) \equiv 0 \pmod{2}$.  If $\alpha \in 2 R$, the
  conclusion is clear.
  It remains to show for $\alpha$ having
  rank $1$ reduction mod $2$ that
  $\Fix_{N_1,N_2}^\sharp(\alpha) = 0$.  The set $V$ of vertices
  $v \in \mathcal{T}$ contracted by $\alpha$ then contains the
  origin, is convex, and contains at least one and at most two
  children of the origin, for else the mod $2$ reduction of
  $\alpha$ would be either invertible or a scalar.
  If $V$
  contains exactly one child of the origin, then
  it contains no pair of independent vertices in the complement of the origin,
  and so
  $\Fix_{N_1',N_2'}(\alpha) = 0$ for all $N_1',N_2' \geq 1$.

  It remains to consider the case that $V$
  contains exactly two children of the origin, say $v_1, v_2$.
  The matrix of $\alpha$ on $E[2]$ with respect to generators of
  $v_1,v_2$ is then diagonal,
  degenerate, and nonzero.  Therefore $\tr(\alpha) \equiv 1 \pod{2}$;
  since $\nr(\alpha) \equiv 0 \pod{2}$,
  we deduce that $\tr(\alpha)^2 - 4 \nr(\alpha)$
  is congruent to $1$ mod $8$ and so is a square in
  $\mathbb{Z}_2^\times$.
  It follows that  $\alpha$ acts on $E[2^N]$ for each $N \geq 1$
  by a diagonal matrix with distinct entries,
  and
  so for each vertex $v \in V$
  other than the origin, $V$ contains exactly
  one of the children of $v$.
  There are thus two infinite
  non-backtracking one-sided paths $Z_1, Z_2$, starting from the
  origin of $\mathcal{T}$ and containing distinct children of the
  origin, so that $V = Z_1 \cup Z_2$.  Thus
  $\Fix_{N_1',N_2'}(\alpha) = 2$ for all $N_1', N_2' \geq 1$.
  The claim $\Fix_{N_1,N_2}^\sharp(\alpha) = 0$ for
  $N_1,N_2 \geq 2$ follows by inclusion-exclusion.
\end{proof}

\begin{proof}[Proof of Proposition \ref{prop:local-pushforward}]
  Suppose first that $N = 2$.
  By Lemma \ref{lem:incl-excl-tricky}, one has
  $\alpha \in S = \mathbb{Z} \oplus S^0$,
  thus $\alpha = m + \beta$ for some $m \in \mathbb{Z}$,
  $\beta \in S^0$.
  By Lemma \ref{lem:fix-n-basic-properties} (i),
  we reduce (adjusting $m$ as necessary) to the terminologically
  simpler case in which $\nr(\alpha)$ is odd.  We compute separately each term
  in $\Fix_{2,2}^\sharp(\alpha)$:
  \begin{itemize}
  \item Since $\alpha$ is a scalar mod $2$, it fixes all children
    $v_1,v_2,v_3$ of the origin $v_0$,
    so $\Fix_{1,1}(\alpha) = 6$ is the number of distinct ordered pairs of such.
  \item 
    $\alpha$ fixes either child of $v_i$
    iff $\eta_i(\beta) = 1$,
    in which case it fixes both children and also both of $v_j,v_k$,
    so $\Fix_{2,1}(\alpha) = \Fix_{1,2}(\alpha) = 2 \sum_{i=1,2,3}
    (1 + \eta_i(\beta))$.
  \item For a distinct ordered pair $i,j \in \{1,2,3\}$,
    $\alpha$ fixes a child of each of $v_i$ and $v_j$
    iff $\eta_i(\beta) = \eta_j(\beta) = 1$,
    in which case it fixes all $4$ such pairs of children,
    hence
    $\Fix_{2,2}(\alpha)
    = 2 \sum _{\{i,j\} \subseteq \{1,2,3\}}
    (1 + \eta_i(\beta)) (1 + \eta_j(\beta))$, with the sum taken over unordered pairs.
  \end{itemize}
  Thus
  \[
  \Fix_{2,2}^\sharp(\alpha)
  =
  2 \sum _{\{i,j\} \subseteq \{1,2,3\}}
  (1 + \eta_i(\beta))(1 + \eta_j(\beta))
  - 4 \sum_{i \in \{1,2,3\}} (1 + \eta_i(\beta))
  + 6.
  \]
  Simplifying,
  we obtain $\Fix_{2,2}^\sharp(\alpha) = 2 \sum_{i=1,2,3} \chi_i(\beta)$,
  as required.

  Suppose now that $N \geq 3$
  and that the conclusion holds
  for smaller values of $N$.
  By Lemma \ref{lem:incl-excl-tricky}, $\alpha \in S = \mathbb{Z} + 2 R$,
  and so there is $t \in \mathbb{Z}$ and  $\gamma \in R$
  for which $\alpha - t  = 2 \gamma$.
  By
  Lemma \ref{lem:fix-n-basic-properties},
  we have $\Fix_{N,N}^\sharp(\alpha) = \Fix_{N,N}^\sharp(2 \gamma)
  =  4 \Fix_{N-1,N-1}^\sharp(\gamma)$,
  and so the conclusion
  follows inductively.
\end{proof}

\subsection{The mean statistics\label{sec:mean-stats}}
\label{sec-2-9}
The following result was promised in \S\ref{sec-1}
(see \eqref{eq:linear-statistics-justification})
to justify interpretating $V_N$ as a variance.
We do not use it otherwise.
\begin{proposition}\label{prop:mean-is-uniform}
  Let $N \geq 2$.
  Then
  \begin{enumerate}
  \item[(i)]
    $|\mathcal{F}_N| = 
    (1 + 2^{-1}) (1 - 2^{-1})^2 2^{2 N} \frac{23-1}{12}$.
  \item[(ii)] $|\mathcal{F}_N|^{-1} \sum_{\varphi \in \mathcal{F}_N}
    \mu_\varphi
    = \mu$.
  \end{enumerate}
\end{proposition}
The proof requires a lemma
similar to the torsion-freeness of $\Gamma_0(4)/\{\pm 1\}$:
\begin{lemma}\label{lem:elliptic-stuff}
  Let $N_1, N_2 \geq 0$ with $N_1 + N_2 \geq 2$.
  Let $\alpha \in R$ with $\nr(\alpha) = 1$.
  Then
  $\Fix_{N_1,N_2}(\alpha) =  1_{\alpha = \pm 1} |\mathcal{L}_{N_1,N_2}|$.
\end{lemma}
\begin{proof}
  If $\alpha = \pm 1$, then it fixes every subgroup.
  Assume otherwise.
  Then $\alpha$ is a non-scalar unit in a definite quaternionic
  order,
  so $\nr(\alpha) = 1$
  and $\tr(\alpha) \in \{-1,0,1\}$.
  Suppose $\Fix_{N_1,N_2}(\alpha) \neq 0$,
  so that $\alpha C_1 = C_1, \alpha C_2 = C_2$
  for some $(C_1, C_2) \in \mathcal{L}_{N_1,N_2}$.
  Consider the
  matrix $\left(
    \begin{smallmatrix}
      a&b\\
      c&d
    \end{smallmatrix}
  \right) \in \GL_2(\mathbb{Z}/4 \mathbb{Z})$
  of $\alpha$ on $E[4]$ with respect to a basis $v_1, v_2$ for
  which $\mathbb{Z} v_i \geq C_i \cap E[4]$.
  From $N_1 + N_2 \geq 2$
  we obtain $b c \equiv 0 \pmod{4}$,
  hence from $\nr(\alpha) = 1$ that
  $a d \equiv 1 \pmod{4}$ and so $\tr(\alpha) \equiv a + d \equiv 2
  \pmod{4}$,
  contradicting that $\tr(\alpha) \in \{-1,0,1\}$.
\end{proof}
We now deduce Proposition \ref{prop:mean-is-uniform}.
Let $\Psi : \mathbf{Y} \rightarrow \mathbb{R}$.
Let $g_E \in G$
represent $E \in \mathbf{Y}$.
Recall from
\S\ref{sec:measures}
that $\int_{\mathbf{X}} \Psi = \sum_{E \in \mathbf{Y}} \frac{\Psi(E)}{w_E}$.
By the $n=1$ case of
\eqref{eq:trace-vs-theta-kernels}
and
\eqref{eq:basic-trace-kernel-identity}
applied
to $\Psi$ rather than $\Psi_k$,
we have
$\sum_{\mathcal{F}_N} \mu_\varphi(\Psi) = \int_{\mathbf{X}} \Psi(g) \sum_{\Gamma} f(g^{-1}
\gamma g)
= \sum_{E \in \mathbf{Y}}
\frac{\Psi(E)}{w_E} \sum_{\Gamma} \phi^K(g_E^{-1} \gamma g_E)$.
If $g_E = 1$, then
\eqref{eq:geometric-interpretation-of-local-pushforward}
and Lemma \ref{lem:elliptic-stuff}
give
$\sum_{\Gamma} \phi^K(\gamma) = (1/2)  \sum_{\alpha \in R :
  \nr(\alpha) = 1}\Fix^\sharp_{N,N}(\alpha)
= |\mathcal{L}_{N,N}|
- |\mathcal{L}_{N-1,N}| 
- |\mathcal{L}_{N,N-1}| 
+ |\mathcal{L}_{N-1,N-1}|
= 2^{2 N}(1 + 2^{-1})(1 - 2^{-1})^2$.
By the proof of Lemma \ref{lem:elliptic-stuff} applied to $R_E := R[1/2] \cap g_E R_2 g_E^{-1}$
rather than $R$, we obtain the same formula for
$\sum_{\Gamma} \phi^K(g_E^{-1} \gamma g_E)$.  Since the
$\mu_\varphi$ are probability measures, we deduce (ii).  By
taking $\Psi = 1$ and applying Eichler's mass formula
$\sum_{E \in \mathbf{Y}} 1/w_E = \frac{23-1}{12}$
\cite{MR894322}, we obtain (i).\footnote{
  One could alternatively apply the Eichler lift and cite known dimension
  formulas for the space of newforms
  on $\Gamma_0(2^{2N} \cdot 23)$ \cite{MR2141534}.  }

\section{Proof of the analytic input}
\label{sec:analytic-input}
\subsection{Overview}
\label{sec-3-1}
We now prove \eqref{eq:sketch-3}.
Recall from \S\ref{sec:reduction-proof} the definition of
$\theta$.
Recall that $|\theta|^2 \notin L^2$.
Fix a modular form $\Phi$ 
on some congruence quotient
satisfying
\begin{equation}\label{eq:postulated-growth-bound}
  \Phi(z) \ll \height(z)^{1/2-\delta} \text{ for some } \delta > 0.
\end{equation}
Let $N \rightarrow \infty$ be a positive integral parameter.
We aim to show that the translates $t(2^{-N}) \Phi(z) := \Phi(2^{2 N} z)$
satisfy
\begin{equation}\label{eq:claimed-translate-estimate}
  \langle |\theta|^2, t(2^{-N}) \Phi  \rangle
= \langle |\theta|^2, 1 \rangle
\langle 1, \Phi  \rangle
+ O_\Phi(N 2^{-N}).
\end{equation}
This gives \eqref{eq:sketch-3} because $h_k, h_l$ are fixed and cuspidal.
The contents
of this section have been developed much more generally
in our preprint \cite{nelson-theta-squared};
we retain this section for completeness,
noting that we require here only a very special case
of the general results of \cite{nelson-theta-squared}.
The reader might profitably consult
\cite[\S3]{nelson-theta-squared} for a toy version
of the argument to follow.

\subsubsection{Remark}
The families $\mathcal{F}_N$ considered in this article consist
of automorphic forms with trivial central character.  If one
instead considers families of forms with central character
$\chi = \prod \chi_v : \mathbb{A}^\times/\mathbb{Q}^\times
\rightarrow \mathbb{C}^{(1)}$
satisfying $\chi_v(-1) = -1$ for at least one place $v$, then
it turns out that a
\emph{cuspidal} elementary theta function takes the place of
$\theta$ in our argument.
The analogue of \eqref{eq:claimed-translate-estimate}
then follows (most directly
with the weaker error $O(N 2^{-(1-2\vartheta) N})$)
from the standard $L^2$-based estimate
\eqref{eq:spectral-motivation}
and the considerations of this section become unnecessary.

\subsection{}
Since $|\theta|^2$ lives on $\Gamma_0(8)$,
we reduce formally
(for mild technical convenience)
to the case that $\Phi$ belongs to the space
\[
\mathcal{A}(2^\infty) := \varinjlim_{n \rightarrow \infty }
\left\{ \text{ smooth } \Phi : \Gamma(2^n) \backslash \mathbb{H} \rightarrow \mathbb{C} \right\}.
\]
We equip $\mathcal{A}(2^{\infty})$ with the normalized Petersson
inner product
$\langle , \rangle$
as in
\S\ref{sec:some-conventions}.
We consider the pairing $\langle \Phi_1, \Phi_2 \rangle$
to be defined whenever $|\Phi_1 \Phi_2|$ is integrable.
\subsection{Change of polarization and Poisson summation}
\label{sec-3-2}
We begin by developing a regularized spectral decomposition of
the function $|\theta|^2 \notin L^2$.
The transformations that
follow amount to regarding $|\theta|^2$ as the restriction to
the first factor of a theta kernel on $\SL_2 \times \O_2$ and
decomposing the latter with respect to the action of the second
factor $\O_2$.
By definition,
\[
  |\theta|^2(z) = y^{1/2} \sum _{\substack{
      m,n \in \mathbb{Z} : \\
      \gcd(m,2) = \gcd(n,2) = 1
    }
  }
  e((m^2-n^2) z)
  \exp(-2 \pi (m^2 + n^2) y).  
\]
We change variables $m,n := (\mu+\nu)/2, (\mu-\nu)/2$
and apply Poisson summation to $\nu$.
To keep track of the $2$-adic summation conditions and weights
that intervene,
it will be technically convenient
to introduce some $2$-adic analysis.
Thus, define the Schwartz--Bruhat function $\phi \in
\mathcal{S}(\mathbb{Q}_2^2)$
by $\phi(m,n) := 1_{\mathbb{Z}_2^\times}(m)
1 _{\mathbb{Z}_2^\times}(n)$,
denote by $\psi : \mathbb{Q}_2 \rightarrow \mathbb{C}^{(1)}$
the standard character for which
$e(x) \psi(x) = 1$ for $x \in \mathbb{Z}[1/2]$,
and introduce the partial Fourier transform
$\mathcal{F} \phi \in \mathcal{S}(\mathbb{Q}_2^2)$
by
\[
  \mathcal{F} \phi(y_1,y_2) := \int _{t \in \mathbb{Q}_2}
  \phi \left( \frac{y_1 + t}{2},
    \frac{y_1 - t}{2}
  \right)
  \psi(y_2 t) \, d t
\]
where $d t$ assigns unit volume to $\mathbb{Z}_2$.
By Poisson summation
for $\mathbb{Z}[1/2] \hookrightarrow \mathbb{R} \times \mathbb{Q}_2$,
\begin{align*}
  |\theta|^2(z)
  &=
    y^{1/2} \sum _{m,n \in \mathbb{Z}[1/2]}
    \phi(m,n) e((m^2-n^2) x)
    \exp(-2 \pi(m^2+n^2)y)
    \\
  &=
    y^{1/2} \sum _{\mu,\nu \in \mathbb{Z}[1/2]}
    \phi \left( \frac{\mu + \nu }{2}, \frac{\mu - \nu }{2}
    \right)
    e(\mu \nu x)
    \exp(-\pi(\mu^2+\nu^2)y)
    \\
  &=
    \sum _{\mu,\nu \in \mathbb{Z}[1/2]}
    \mathcal{F} \phi ( \mu,\nu)
    \exp(-\pi((\mu x + \nu)^2/y+\mu^2y))
    \\
  &=
    \sum _{\mu,\nu \in \mathbb{Z}[1/2]}
    \mathcal{F} \phi ( \mu,\nu)
    \exp(-\pi|\mu z + \nu|^2/y).
\end{align*}
\subsection{Some notation}
\label{sec-3-4}
Introduce the general notation
\[
  n(b) := \begin{pmatrix}
    1 & b \\
    & 1
  \end{pmatrix},
  \quad
  t(a) := \begin{pmatrix}
    a &  \\
    & a^{-1}
  \end{pmatrix}, \quad e_2 := (0,1).
\]
Denote
by $N_2, T_2, B_2$ the subgroups
of $\SL_2(\mathbb{Q}_2)$
consisting of elements of the respective forms
$n(\ast), t(\ast), n(\ast) t(\ast)$.
Denote by
$\widetilde{\Gamma}_\infty <
\SL_2(\mathbb{Z}[1/2])$ and  
$\Gamma_\infty <
\SL_2(\mathbb{Z})$   
the intersections with the upper-triangular Borel subgroups;
the corresponding quotients
are $\mathbb{P}^1(\mathbb{Z}[1/2]) \cong
\mathbb{P}^1(\mathbb{Z})$.
\subsection{Folding up, Mellin expansion}
\label{sec-3-5}
Each $(\mu,\nu) \in \mathbb{Z}[1/2]^\times - \{(0,0)\}$
is uniquely of the form $\lambda e_2 g$
for some $g \in \widetilde{\Gamma}_\infty \backslash
\SL_2(\mathbb{Z}[1/2])$
and $\lambda \in \mathbb{Z}[1/2] - \{0\}$.
Since then $|\mu z + \nu|^2/y = \lambda^2 / \Im(g z)$,
we obtain
\[
  |\theta|^2(z) = \mathcal{F} \phi(0,0)
  + \sum _{\gamma \in \widetilde{\Gamma }_\infty \backslash
    \SL_2(\mathbb{Z}[1/2])}
  F(\gamma z, \gamma)
\]
where $F : \mathbb{H} \times \SL_2(\mathbb{Q}_2) \rightarrow
\mathbb{C}$
is given by
\[
  F(z,g)
  :=
  \sum _{\lambda \in \mathbb{Z}[1/2] - \{0\}}
  \exp(-\pi \lambda^2 / \Im(z))
  F \phi(\lambda e_2 g).
\]
The second term is an incomplete Eisenstein series
which we now study by Mellin expansion.
For fixed $z \in \mathbb{H}$ and $g \in \SL_2(\mathbb{Q}_2)$,
the function of the variable $a = (a_\infty,a_2)
\in \mathbb{R} ^\times \times \mathbb{Q} _2 ^\times$
given by
$a \mapsto F(a_\infty^2 z, t(a_2) g)$
is left-invariant by
the diagonal embedding
of $\mathbb{Z}[1/2]^\times$, right-invariant
by $\{\pm 1\} \hookrightarrow \mathbb{R}_+^\times$,
decays rapidly as $|a| := |a_\infty|_\infty |a_2|_2$
tends to zero, and is $O(|a|^{O(1)})$
as $|a| \rightarrow \infty$.
It thus admits a Mellin expansion
indexed by
\[
  \mathfrak{X} := \left\{
    \chi = (\chi_\infty, \chi_2)
    \in \Hom(\mathbb{R}^\times \times \mathbb{Q}_2^\times,
    \mathbb{C}^\times)
    \Bigg\bracevert
    \begin{array}{lr}
      \chi|_{\mathbb{Z}[1/2]^\times} = 1, \\
      \chi_\infty(-1) = 1
    \end{array}
  \right\}
\]
and given for large enough $c > 2$ by
$F(z,g) = \int_{\chi \in \mathfrak{X} : \Re(\chi) = c}
F_\chi(z,g) \, d \chi$
with
\[
  F_\chi(z,g) := \int_{a \in \mathbb{Z}[1/2]^\times \backslash
    (\mathbb{R}^\times \times \mathbb{Q}_2^\times)}
  \chi^{-1}(a) F(a_\infty^2 z, t(a_2) g) \, d^\times a.
\]
Here we normalize measures by taking on $\mathbb{Z}[1/2]^\times
\backslash (\mathbb{R}^\times \times \mathbb{Q}_2^\times)$
the quotient of the product of the standard Haar measures on
$\mathbb{R}^\times$
and $\mathbb{Q}_2^\times$
assigning unit volume to $(1,e)$ and $\mathbb{Z}_2^\times$,
respectively,
and on each
$\{\chi \in \mathfrak{X} : \Re(\chi) = c \}$ the dual measure
$d \chi$.
Concretely, $\mathfrak{X}$ identifies with the set of pairs
$\chi \leftrightarrow (s_\chi,\omega_\chi)$,
where $s_\chi \in \mathbb{C}$
and $\omega_\chi : \mathbb{Z}_2^\times \rightarrow
\mathbb{C}^{(1)}$
is a character satisfying $\omega_\chi(-1) = 1$.
This identification is determined by requiring that
$\chi_\infty(y) = |y|^{s_\chi}$ for $y \in \mathbb{R}^\times$
and $\chi_2(2^n u) = 2^{- n s_\chi} \omega_\chi(u)$
for $n \in \mathbb{Z}$ and $u \in \mathbb{Z}_2^\times$.
The real part, analytic conductor and dual measure
are given in these coordinates by $\Re(\chi) = \Re(s_\chi)$,
$C(\chi) = (1 + |s_\chi|) C(\omega_\chi)$
and $d \chi = \frac{d s_\chi}{2 \pi i } \,  d \omega_\chi$,
where
$d \omega_\chi$ denotes counting measure.
We remark that only those $\chi$ with $C(\omega_\chi) \leq 2^4$
(say)
are required in our argument.
\subsection{$L$-functions}
\label{sec-3-6}
For $\chi \in \mathfrak{X}$,
write $\chi_p(p) := \chi_\infty(p)^{-1} \chi_2(p)^{-1}$ for $p > 2$
and define
\[
\Lambda(\chi,s) := L(\chi_\infty,s) L(\chi_2,s) \prod_{p>2} (1
- \chi_p(p) p^{-s})^{-1}
\]
for $\Re(s)$ large enough
and in general by meromorphic continuation.
Thus $\Lambda(\chi,s)$ is the completed $L$-function
obtained by regarding $\chi$ as a Hecke character unramified
outside $\{\infty, 2\}$.
For orientation, we record that the element $|.|^s \in
\mathfrak{X}$
with components $(|.|_\infty^s, |.|_2^s)$
has $\Lambda(|.|^s,0) = \xi(s) = \pi^{-s/2} \Gamma(s/2)
\zeta(s)$.
\subsection{Unfolding}
\label{sec-3-7}
We now explicate $F_\chi$
further by opening the sum defining $F$.
The positive odd integers give representatives for the
$\mathbb{Z}[1/2]^\times$-orbits
on $\mathbb{Z}[1/2] - \{0\}$,
so $F_\chi(z,g)$ unfolds to
\[
\sum _{\substack{
    \lambda \geq 1 : \\
    \gcd(\lambda,2) = 1
  }}
\chi^{-1}(\lambda)
\int _{a \in \mathbb{R}^\times \times \mathbb{Q}_2^\times }
\chi^{-1}(a)
\exp \left( \frac{-\pi }{ \Im(z) a_\infty^2}\right)
\mathcal{F} \phi(a_2^{-1} e_2 g) \, d^\times a.
\]
By evaluating the Dirichlet series and local Tate integrals,
we obtain
$F_\chi(z,g) = \Lambda(\chi,0) f_\chi(z,g)$
with
\[
f_\chi(z,g) := \frac{\chi_\infty(\Im(z)^{1/2})}{L(\chi_2,0)}
\int_{a \in \mathbb{Q}_2^\times}
\chi_2(a) \mathcal{F} \phi(a e_2 g) \, d^\times a.
\]
There exists an open subgroup $U \leq \SL_2(\mathbb{Z}_2)$,
independent of everything,
so that $f_\chi(z, g u) = f_\chi(z,g)$ for all $u \in U$.
By the theory of local Tate integrals or direct evaluation,
one has
$\|f_\chi \|\ll 1$ for all $\chi \in \mathfrak{X}$
with $-10c \leq  \Re(\chi) \leq 10c$, say.
\subsection{Induced representations, Eisenstein series}
\label{sec-3-8}
The function $f_\chi$
belongs to the space
\[
\mathcal{I}(\chi) :=
\left\{
  f : \mathbb{H} \times \SL_2(\mathbb{Q}_2) \rightarrow
  \mathbb{C}
  \Bigg\bracevert
  \begin{array}{lr}
    f(z, n(b)  t(a) g)
    = \chi(\Im(z)^{1/2},a) f(i,g) \\
    \text{ for }
    z,a,b,g \in \mathbb{H},\mathbb{Q}_2^\times,
    \mathbb{Q}_2,\SL_2(\mathbb{Q}_2), \\
    f \text{ is smooth}
  \end{array}
\right\}
\]
which
arises naturally as the $\SO(2)$-fixed
subspace of the representation of $\SL_2(\mathbb{R} \times
\mathbb{Q}_2)$ induced by $\chi$
(without normalization).
The space $\mathcal{I}(\chi)$ is stable
for the Laplacian,
which acts by the scalar
$\tfrac{1}{2} s_\chi(\tfrac{1}{2} s_\chi-1)$,
and under right translation by $\SL_2(\mathbb{Q}_2)$.
Denote by $\Eis : \mathcal{I}(\chi) \rightarrow
\mathcal{A}(2^\infty)$
the standard intertwiner
given by
\[
\Eis(f)(z) :=
\sum _{\gamma \in \widetilde{\Gamma }_\infty \backslash
  \SL_2(\mathbb{Z}[1/2])}
f(\gamma z, \gamma)
= \sum _{\gamma \in \Gamma_\infty \backslash
  \SL_2(\mathbb{Z})}
f(\gamma z, \gamma)
\]
for $\Re(\chi) > 2$
and in general by meromorphic continuation along flat sections.
This assignment is equivariant for $\Delta$ and
$\SL_2(\mathbb{Q}_2)$; the latter acts on
$\mathcal{A}(2^\infty)$
by $g \Phi(z) := \Phi(\alpha^{-1} z)$
for $\alpha \in \SL_2(\mathbb{Z}[1/2])$
taken $2$-adically close enough to $g \in \SL_2(\mathbb{Q}_2)$.
Set $\Eis^*(f) := \Lambda(\chi,0) \Eis(f)$;
it is defined for all $\chi \neq |.|^2$
with $\Re(\chi) \geq 1$,
noting that the pole of $\Lambda(\chi,0)$ at $\chi = |.|^1$
cancels the simple zero of $\Eis$ at that parameter.
Then
\[
|\theta|^2(z)
= \mathcal{F} \phi(0,0)
+ \int _{\chi \in \mathfrak{X} : \Re(\chi) = c}
\Eis^*(f_\chi)(z) \, d \chi.
\]

\subsection{Example}
\label{sec:example-eis-full-level}
Suppose $\chi = |.|^s$
and $f^0_\chi \in \mathcal{I}(\chi)$
is defined by requiring that $f^0_\chi(i,k) = 1$ for all $k \in \SL_2(\mathbb{Z}_2)$.
Then
$\Eis(f^0_\chi)(z) = E_{s/2}(z)$
and $\Eis^*(f^0_\chi)(z) = \xi(s) E_{s/2}(z)$
with $E_s(z) = y^s + \dotsb$
as in \S\ref{sec-1}.

\subsection{Measures}
\label{sec-3-9}
Equip $N_2, T_2$ with the measures transported by
the isomorphisms $n : \mathbb{Q}_2 \cong N_2, t :
\mathbb{Q}_2^\times \cong T_2$,
$B_2$ with the left Haar compatible with
$B_2/N_2 \cong T_2$ and the chosen measures on $N_2, T_2$,
and $\SL_2(\mathbb{Q}_2)$ with the Haar inducing a quotient Haar for which $N_2 \backslash \SL_2(\mathbb{Q}_2)
\ni g \mapsto e_2 g \in \mathbb{Q}_2^2$ is measure-preserving.
Equip $B_2 \backslash \SL_2(\mathbb{Q}_2)$ with the quotient Haar.
\subsection{Contour shift}
\label{sec-3-10}
We now shift to $\mathfrak{X}_u := \{\chi \in \mathfrak{X} :
\Re(\chi) = 1\}$,\footnote{
  The unitary axis is at $1$ rather than $1/2$
  because we are working with $\SL_2$ rather than $\PGL_2$
  Eisenstein series, the former being more natural for our
  purposes.
}
passing a pole at $\chi = |.|^2$ of residue
$\int_{g \in N_2 \backslash \SL_2(\mathbb{Q}_2)} \mathcal{F}
\phi(e_2, g)
= \int_{\mathbb{Q}_2^2} \mathcal{F} \phi$;
see \S\ref{sec-3-12} below for
details.
Because the function $1_{\mathbb{Z}_2^\times}$ is even,
we see by Fourier inversion
that $\mathcal{F} \phi(0,0) = \int_{\mathbb{Q}_2^2} \mathcal{F}
\phi = (1/2) \int_{x \in \mathbb{Q}_2} \phi(x,x) = 1/4$.
Thus $|\theta|^2(z) = 1/2 + \int_{\chi \in \mathfrak{X}^u}
\Eis^*(f_\chi)(z) \, d \chi$.
We integrate both sides against the constant function $1$ and
the function
$\Phi$ from \S\ref{sec-3-1},
using standard growth estimates on Eisenstein series
(see \S\ref{sec-3-13} below)
and the boundedness of $f_\chi$ (see \S\ref{sec-3-7})
and the rapid decay of $\Lambda(\chi,0)$
for bounded $C(\omega_\chi)$
to justify changing the order of integration.
We obtain
\begin{equation}\label{eq:theta-l2-norm}
  \|\theta \|^2 = 1/2
\end{equation}
and
\begin{equation}\label{eq:theta-expansion}
  |\theta|^2(z) = \langle |\theta|^2,1 \rangle
  + \int _{\chi \in \mathfrak{X}^u}
  \Eis^*(f_\chi)(z) \, d \chi
\end{equation}
and
\begin{equation}\label{eq:theta-expansion-2}
  \langle |\theta|^2, t(2^{-N}) \Phi  \rangle
  = \langle |\theta|^2,1 \rangle
  \langle 1, \Phi  \rangle
  + \int _{\chi \in \mathfrak{X}^u}
  \langle \Eis^*(f_\chi), t(2^{-N}) \Phi \rangle
  \, d \chi.
\end{equation}
To complete the proof of \eqref{eq:claimed-translate-estimate}
it suffices now to show for $A = 1.01$, so that
$\int_{\mathfrak{X}^u} C(\chi)^{-A} \, d \chi < \infty$,
that
\begin{equation}\label{eq:required-standard-growth-matrix-coef-bound}
  \langle \Eis^*(f_\chi), t(2^{-N}) \Phi \rangle
  \ll_{\Phi} N 2^{-N} C(\chi)^{-A}.
\end{equation}
This follows from standard growth and matrix coefficients
bounds;
see \S\ref{sec-3-18} for details.
\subsection{Residues}
\label{sec-3-12}
When $\chi = |.|^2$,
so that each $f \in \mathcal{I}(\chi)$ transforms on the left
under the modulus character of $B_2$,
the representation $\mathcal{I}(\chi)$ is reducible:
there is an $\SL_2(\mathbb{Q}_2)$-invariant map
$\mathcal{R} : \mathcal{I}(\chi) \rightarrow \mathbb{C}$
given by $\mathcal{R}(f) := \zeta_2(2) \int_{g \in B_2
  \backslash \SL_2(\mathbb{Q}_2)} f(i,g)$,
which we have normalized so that $\mathcal{R}(f^0_\chi) = 1$ for $f^0_\chi$ as in \S\ref{sec:example-eis-full-level}.
Let $f_\chi \in \mathcal{I}(\chi)$ vary in a bounded holomorphic family
in the strip $2 - \eps < \Re(\chi) < 2 + \eps$,
with
$f_\chi$ invariant by some open $U \leq \SL_2(\mathbb{Q}_2)$ independent of $\chi$;
these assumptions hold in the context of \S\ref{sec-3-7}--\S\ref{sec-3-10}.
Then $\Eis(f_\chi)$ is holomorphic away from a simple pole at $\chi = |.|^2$
with residue described by
\begin{equation}\label{eq:residue-of-eisenstein-intertwiner}
  \int_{\chi \in \mathfrak{X} : \Re(\chi) = 2 + \eps } \Eis^*(f_\chi) \, d \chi
  - 
  \int_{\chi \in \mathfrak{X} : \Re(\chi) = 2 - \eps } \Eis^*(f_\chi) \, d \chi
  =
  \mathcal{R}(f_{|.|^2}).
\end{equation}
For example, for $f_\chi = f_\chi^0$, this says that
\[
\res_{s \rightarrow 2} \xi(s) E_{s/2}(z)
= \res_{s \rightarrow 1} 2 \xi(2 s) E_{s}(z) = 1,
\]
as is well-known.
The general case follows either
by noting that both sides
of \eqref{eq:residue-of-eisenstein-intertwiner}
may be interpreted
as defining elements of the one-dimensional space
of equivariant functionals $\mathcal{I}(|.|^2) \rightarrow
\mathbb{C}$
or by applying the general treatment of \cite{MR546600}
to $\SL_2$ instead of $\PGL_2$.
\subsection{Growth bounds}
\label{sec-3-13}
Define $\height: \mathbb{H} \rightarrow \mathbb{C}$ by
$\height(z) := \sup_{\gamma \in \SL_2(\mathbb{Z})} \Im(\gamma
z)$.
It descends to $\height : \Gamma(2^n) \backslash \mathbb{H}
\rightarrow \mathbb{C}$.
Let $U$ be an open subgroup of $\SL_2(\mathbb{Q}_2)$.
Let $\chi \in \mathfrak{X}$ with $\Re(\chi) \geq 1$
and $\Re(\chi) \ll 1$.
Let $f \in \mathcal{I}(\chi)$ be  $U$-invariant.
Then
\begin{equation}\label{eq:lem-growth-bound}
  \Eis^*(f)(z) \ll_{U,A} C(\chi)^{-A}
  \height(z)^{\Re(\chi)/2} \log(3 + \height(z)) \|f\|.
\end{equation}
\begin{proof}
  We may assume $f \neq 0$.
  Then the $U$-invariance of $f$ implies
  $C(\chi_2) \ll_U 1$,
  hence that $\Lambda(\chi,0) \ll_A C(\chi)^{-A}$,
  so it suffices to show that
  $\Eis^*(f)(z) \ll_{U} C(\chi)^{O(1)}
  \height(z)^{\Re(\chi)/2} \log(3 + \height(z)) \|f\|$.
  For this, we estimate Fourier coefficients
  on a Siegel domain as in \cite[(4.12)]{michel-2009}
  and
  \cite[(3.23)]{michel-2009}.
\end{proof}
\subsection{Projections}
\label{sec-3-14}
For $\chi \in \mathfrak{X}^u$,
the norm
\[
\mathcal{I}(\chi) \ni f \mapsto \|f\|^2 :=
\int_{g \in B_2 \backslash \SL_2(\mathbb{Q}_2)} |f(i,g)|^2
\]
is $\SL_2(\mathbb{Q}_2)$-invariant.
By duality, for $\Phi \in
\mathcal{A}(2^\infty)$
satisfying \eqref{eq:postulated-growth-bound}
there is a unique $\Phi_\chi \in \mathcal{I}(\chi)$
so that $\langle \Eis(f), \Phi \rangle
= \langle f, \Phi_\chi  \rangle$ (the first inner product taken in $\mathcal{A}(2^{\infty})$, the second in $\mathcal{I}(\chi)$).
The maps $\Phi \mapsto \Phi_\chi$ are linear and equivariant
for $\Delta$ and $\SL_2(\mathbb{Q}_2)$.
\subsection{Plancherel theorem}
\label{sec-3-15}
We record for the sake of orientation that if $\Phi$ is
square-integrable,
then
$\|\Phi \|^2 = \|\Phi_{\disc}\|^2
+ (1/2) \int_{\mathfrak{X}^u} \|\Phi_\chi \|^2 \, d \chi$,
where $\Phi_{\disc}$ is the orthogonal projection
onto the discrete part of $\mathcal{A}(2^\infty)$
spanned by constants and cusp forms.
\subsection{Bounds for projections}
\label{sec-3-16}
For $\Phi \in \mathcal{A}(2^\infty)$ satisfying
\eqref{eq:postulated-growth-bound} and $A > 0$,
we claim that
$\Lambda(\chi,0) \|\Phi_\chi \| \ll_{\Phi,A} C(\chi)^{-A}$ for
all $\chi \in \mathfrak{X}^u$.  To see this, let $U$ be an open
subgroup of $\SL_2(\mathbb{Q}_2)$ that fixes $\Phi$.  Then
$f := \Phi_\chi \in \mathcal{I}(\chi)$ is $U$-invariant.  By
estimating the integral of $|\Eis^*(f)(z) \Phi(z)|$ over a
Siegel domain using \S\ref{sec-3-13}, it follows
that
  \[
  \Lambda(\chi,0) \langle f, \Phi_\chi  \rangle
  =
  \langle \Eis^*(f), \Phi  \rangle
  \ll_{\Phi,A} C(\chi)^{-A} \|f\|.
  \]
  Cancelling common factors of $\|\Phi_\chi\|$ (if nonzero)
  from $\langle f, \Phi_\chi  \rangle = \|\Phi_\chi\|^2$
  and $\|f\| = \|\Phi_\chi\|$,
  we conclude.
\subsection{Bounds for matrix coefficients}
\label{sec-3-17}
Let $\chi \in \mathfrak{X}^u$,
let $U$ be an open subgroup of $\SL_2(\mathbb{Z}_2)$,
let $f_1, f_2 \in \mathcal{I}(\chi)$ be $U$-invariant, and let
$n \in \mathbb{Z}$.
By explicating \cite{MR946351},
we obtain\footnote{
  This estimate may be understood as a
  general form of the well-known Ramanujan--type bound for
  the $2^{2 n}$th Hecke eigenvalues of unitary Eisenstein
  series on $\SL_2(\mathbb{Z})$.
}
\[|\langle f_1, t(2^n) f_2 \rangle| \leq [\SL_2(\mathbb{Z}_2) : U] (2 |n| + 1) 2^{-|n|}
  \|f_1\| \|f_2\|.\]
\subsection{Completion}
\label{sec-3-18}
We deduce the remaining estimate
\eqref{eq:required-standard-growth-matrix-coef-bound}
from \S\ref{sec-3-17} and \S\ref{sec-3-16}:
\begin{align*}
  \langle \Eis^*(f_\chi), t(2^{-N}) \Phi  \rangle
  &= \Lambda(\chi,0)
    \langle f_\chi, t(2^{-N}) \Phi_\chi  \rangle
    \\
  &\ll_{\Phi}
  N 2^{-N}
  \|f_\chi\|
    \Lambda(\chi,0) \|\Phi_\chi\|
    \\
  &\ll_{\Phi,A}
  N 2^{-N} C(\chi)^{-A}.
\end{align*}
\subsection{Refinements}\label{sec:remark-optimality}
We discuss heuristically
some possible refinements
of the above analysis
(cf. \S\ref{sec-optimal-error-term}),
leaving an actual implementation to the interested reader.
Recall that $(f_\chi)_\chi$ and $\Phi$ are independent of
$N$ and $U$-invariant for some fixed open
$U \leq \SL_2(\mathbb{Z}_2)$.  This subgroup $U$ is \emph{not}
the full maximal compact subgroup $\SL_2(\mathbb{Z}_2)$, but let
us pretend for the sake of illustration that it were.  Then
$f_\chi$ and $\Phi_\chi$ ``are'' spherical vectors.  For
$\chi \in \mathfrak{X}^u$, it follows that $f_\chi = 0$ unless
$\chi = |.|^{1 + 2 i t}$ for some $t \in \mathbb{R}$, in which
case we may write
$\langle \Eis^*(f_\chi), \Phi_\chi \rangle =: H(t)$ for some
Schwartz
function
$H : \mathbb{R} \rightarrow \mathbb{C}$.
By the Macdonald formula for the spherical matrix coefficients
of $\mathcal{I}(\chi)$,
the inner product
$\langle \Eis(f_\chi), t(2^{-N}) \Phi \rangle$ may be written
$2^{-N} H(t) (2^{N i t} + 2^{(N-2) i t} + \dotsb + 2^{- N i t})$
plus a similar term involving $N-2$ in place of $N$.
Thus $\langle  |\theta|^2, t(2^{-N}) \Phi \rangle
- \langle |\theta|^2, 1 \rangle \langle 1, \Phi  \rangle$
is morally
\[
  2^{-N}
  \int_{t \in \mathbb{R}}
  H(t) (2^{N i t} + 2^{(N-2) i t} + \dotsb + 2^{- N i t}) \, d t
  = 2^{-N}
  \sum _{\substack{
      -N \leq n \leq N : \\
      n \equiv N (2)
    }
  } \hat{H}(n)
\]
for
the normalized Fourier transform
$\hat{H}(\xi) := \int_{t \in \mathbb{R}} H(t) 2^{i \xi t} \, d
t$.  Since $\hat{H}$ decays rapidly, the error is thus
$O(2^{-N})$.
Similar arguments should apply
non-heuristically.

\section{Classical correlations\label{sec:classical-correlations}}
\label{sec-6}
In this section we record some local calculations relevant for
specializing the Rallis inner product formula.  The main result
is Proposition \ref{prop:classical-correlations}.  Our proof
uses Fourier analysis as in \S\ref{sec:fluctuations},
which seems more naturally suited to the task at hand;
it should also be possible
to argue geometrically as in \S\ref{sec:appendix-fluctuations}.
We retain throughout this section
the notation and setup of \S\ref{sec:appendix-fluctuations}.

\subsection{Fourier-analytic interpretation of some quadratic characters}\label{sec:four-analyt-interpr}
Recall the definition of $\chi_i$ from
\S\ref{sec:defn-ternary-thetas}.
We record here an
equivalent Fourier-analytic definition.
Recall
$\rho : R \rightarrow M_2(\mathbb{Z}/4)$
from the proof of Lemma \ref{lem:triples-agree}.
Recall that (up to permutation)
$\chi_1(\alpha), \chi_2(\alpha),\chi_3(\alpha)
  =
  (-1)^{b+c},
  (-1)^{a+c},
  (-1)^{a+b}$.
Let $e_1,e_2,e_3$ be the matrices
\[
  e_1 :=
  \begin{pmatrix}
    & 1 \\
    1 & 
  \end{pmatrix},
  \quad
  e_2 :=
  \begin{pmatrix}
    1 & 1 \\
    & -1
  \end{pmatrix},
  \quad
  e_3 :=
  \begin{pmatrix}
    1 &  \\
    1 & -1
  \end{pmatrix}.
\]
Their reductions modulo $2$ form a conjugacy class in
$\GL_2(\mathbb{F}_2)$.
For matrices $A, A_1, A_2$ over any commutative ring,
denote by $A \mapsto A^{\iota}$ the main anti-involution
(thus $A A^{\iota} = \det(A)$)
and by $\langle A_1, A_2 \rangle := \tr(A_1 A_2^{\iota})$
the trace pairing.
For $\alpha \in S^0$, the pairings
$\langle \rho(\alpha), e_i/2 \rangle \in \mathbb{Z}/2$ are then
well-defined,
and one has
\begin{equation}\label{eq:chi-fourier-interpretation}
  \chi_i(\alpha) =
  (-1) ^{ \langle \rho(\alpha), e_i/2 \rangle }.
\end{equation}

\subsection{Fourier transform and adjoint action}
Let $B := R \otimes_{\mathbb{Z}} \mathbb{Q}$
be the quaternion algebra generated by $R$.
Introduce the subscript $2$ to denote ``$2$-adic completion,''
so that $B_2 = B \otimes_{\mathbb{Q}} \mathbb{Q}_2$,
$R_2 = R \otimes_{\mathbb{Z}} \mathbb{Z}_2$,
$R_2^0 = R^0 \otimes_{\mathbb{Z}} \mathbb{Z}_2
= \{\alpha \in R_2 : \tr(\alpha) = 0\}$,
and so on.
Fix an identification
$B_2 = M_2(\mathbb{Q}_2)$
under which $R_2$ identifies with $M_2(\mathbb{Z}_2)$.
The characters $\chi_i : S^0 \rightarrow \{ \pm 1 \}$
extend by continuity
to $S_2^0$.
Extending them further by zero,
we obtain Schwartz--Bruhat functions $\chi_i : B_2^0 \rightarrow
\mathbb{C}$ as in \S\ref{sec:defn-ternary-thetas}.
For notational clarity, set $\mathfrak{o} := \mathbb{Z}_2,
\mathfrak{p} := 2 \mathbb{Z}_2$.
Fix an unramified character
$\psi : \mathbb{Q}_2 \rightarrow \mathbb{C}^{(1)}$,
i.e., one for which
$\{x \in \mathbb{Q}_2 : \psi(x \mathfrak{o}) \subseteq \{1\}\}
= \mathfrak{o}$.
Equip
$B_2$ with the measure $d \alpha$
with the property that the Fourier transform
$\mathfrak{F} \circlearrowright \mathcal{S}(B_2^0)$
defined by
\begin{equation}\label{eqn:fourier-B2}
  \mathfrak{F} \phi(\alpha) :=
  \int_{\alpha ' \in B_2}
  \phi(\alpha ')
  \psi(\langle \alpha, \alpha ' \rangle)
  \, d \alpha ',
\end{equation}
with
$\langle \alpha, \alpha ' \rangle$ the trace pairing as in \S\ref{sec:some-quadratic-characters},
satisfies
$\mathfrak{F} \mathfrak{F} \phi(\alpha) = \phi(-\alpha)$.
Define an inner product $\langle , \rangle_{L^2(B_2)}$
with respect to $d \alpha$.
Set $G := \PGL_2(\mathbb{Q}_2) = B_2^\times/\mathbb{Q}_2^\times$,
$K := \PGL_2(\mathbb{Z}_2) = R_2^\times/\mathbb{Z}_2^\times$.
For $X = B_2^0$ or $X = B_2$, the group $G$
acts on $X$,
hence on $\mathcal{S}(X)$,
by the adjoint action: for $g,\beta,\phi \in G,
X, \mathcal{S}(X)$
\begin{equation}\label{eq:adjoint-action}
  \Ad(g) \beta := g \beta g^{-1},
  \quad 
  \Ad(g) \phi(\beta) := \phi(\Ad(g)^{-1} \beta).
\end{equation}
Recall the Cartan decomposition
$G = \bigsqcup_{n \in \mathbb{Z}_{\geq 0}}
K a(2^n) K$
with
$a(y) := \diag(y,1)$.
For $g \in G$,
denote by
$\mathbf{n}(g)$ the integer $n \in \mathbb{Z}_{\geq
  0}$
for which 
$K a(2^n) K
= K g K$.
\subsection{Fourier-analytic calculations}
Define $\phi,\phi ',\phi '' \in
\mathcal{S}(B_2),\mathcal{S}(\mathbb{Q}_2),
\mathcal{S}(B_2^0)$
by
$\phi '(m) := 1_{\mathfrak{o}^\times}(m)$,
$\phi ''(\beta) := 2^{-3} \kappa_0 \sum_{i=1,2,3} \chi_i(4
\beta)$
and by $\phi(m+\beta) := \phi '(m) \phi ''(\beta)$
for $m,\beta \in \mathbb{Q}_2, B_2^0$.
Thus $\phi ''$
is
as in \S\ref{sec:defn-ternary-thetas}.
\begin{lemma}\label{lem:inner-product-involving-phi}
  Let $g \in G$.
  Set $n := \mathbf{n}(g)$.  Then
  \[\langle \Ad(g) \phi, \phi  \rangle_{L^2(B_2)} = 2^{-4} \kappa_0^2 (1_{n=0} 2 + 1) 2^{-n}.\]
\end{lemma}
\begin{proof}
  Observe first that,
  since $e_1,e_2,e_3$ form a conjugacy class mod $2$,\footnote{alternatively,
    since the vertices $v_1,v_2,v_3$
    in the definition of the characters $\chi_i$
    are permuted under tree automorphisms
    fixing the origin}
  the function
  $\phi''$ and hence also $\phi$ is
  $\Ad(K)$-invariant.
  We thereby reduce to the case
  $g = a(2^n)$.
  Equip $\mathbb{Q}_2$ with the Haar measure
  $d a$ assigning volume one to $\mathfrak{o}$.
  Equip $B_2^0$ with the Haar measure
  $d \beta$ so that for $f \in C_c(B_2^0)$,
  \[
    \int_{\beta \in B_2^0} f(\beta) \, d \beta
    :=
    \int_{a,b,c \in \mathbb{Q}_2}
    f
    (
    \begin{pmatrix}
      a & b \\
      c & -a
    \end{pmatrix}
    )
    \, d a \, d b \, d c.
  \]
  The Haar measure $d \alpha$
  on $B_2$ is then given for $f \in C_c(B_2)$
  by
  \[
    \int_{\alpha \in B_2} f(\alpha) \, d \alpha
    =
    \int_{t,\beta  \in \mathbb{Q}_2,B_2^0}
    f(t/2+\beta) \, d t \, d \beta
  \]
  because we may compute that with this normalization,
  $\mathfrak{F} 1_{M_2(\mathbb{Z}_2)}
  = 1_{M_2(\mathbb{Z}_2)}$.
  Thus
  $\langle \Ad(g) \phi, \phi \rangle_{L^2(B_2)}
  =
  I' \cdot I''(g)$
  with
  \[
    I'
    :=
    \int_{t \in \mathbb{Q}_2}
    |\phi '(t/2)|^2 \, d t
    =
    \int_{t \in \mathbb{Q}_2}
    1_{\mathfrak{o}^\times}(t/2)  \, d t
    =
    2^{-2}
  \]
  and
  $I''(g)
  :=
  \int_{\beta \in B_2^0}
  \langle \Ad(g) \phi '', \phi '' \rangle$.
  We must verify that
  $I''(g) = 2^{-2} 3\kappa_0^2$ for $n = 0$
  and $I''(g) = 2^{-2-n} \kappa_0^2$ for $n > 0$.
  We invoke the Fourier transform $\mathcal{F} ''$
  on $\mathcal{S}(B_2^0)$,
  defined by analogy to \eqref{eqn:fourier-B2} using $d \beta$.
  The measure $d \beta$
  is not self-dual for $\mathfrak{F} ''$
  but instead satisfies
  \begin{equation}\label{eq:plancherel-formula-B0}
    \|f\|^2 = 2^{-1} \|\mathcal{F}'' f\|^2 \quad \text{ for $f \in L^2(B_2^0)$} 
  \end{equation}
  as one verifies by taking (say) $f :=
  1_{R_2^0}$
  and evaluating $\|f\|^2 = 1, \mathfrak{F} f = 1_{2^{-1} S_2^0}, \|\mathfrak{F} f\|^2 = 2$.
  By \eqref{eq:chi-fourier-interpretation}, we have
  $\chi_i(4 \beta)
  = 1_{S^0}(4 \beta) \psi( \langle \beta, e_i \rangle)$.
  We compute
  \[
    \int_{\beta ' \in B_2^0} 1_{S^0}(4 \beta)
    \psi ( \langle \beta, \beta ' \rangle)
    = 2^4
    1_{2 R^0}(\beta)
  \]
  by expanding the LHS
  for
  $\beta = \begin{pmatrix}
    a & b \\
    c & -a
  \end{pmatrix}$
  to
  \[
    \int_{a',b',c' \in \mathbb{Q}_2}
    1_{\mathfrak{o}}(4 a')
    1_{2 \mathfrak{o}}(4 b')
    1_{2 \mathfrak{o}}(4 c')
    \psi(2 a a' + b b' + c c')
    \, d a'
    \, d b'
    \, d c'
  \]
  and evaluating the latter to
  $2^4
  1_{2 \mathfrak{o}}(a)
  1_{2 \mathfrak{o}}(b)
  1_{2 \mathfrak{o}}(c) = 2^4 1_{2 R^0}(\beta)$.
  Thus
  \[\mathcal{F} '' \phi ''
    = 2^{-3} \kappa_0 \sum_{i=1,2,3}
    2^4 1_{e_i + 2 R_2^0}
    =
    2 \kappa_0 \sum_{i=1,2,3} 1_{e_i + 2 R_2^0}.\]
  By \eqref{eq:plancherel-formula-B0}
  and the commutation
  $\Ad(g) \mathfrak{F} '' \phi ''
  = \mathfrak{F} '' \Ad(g) \phi ''$,
  \[
    I''(g) = 2^{-1}  \int_{\beta \in B_2^0}
    \langle \Ad(g) \mathfrak{F} '' \phi '', \mathfrak{F} '' \phi
    '' \rangle
    = 2 \kappa_0^2
    \sum_{i,j=1,2,3}
    \langle \Ad(g) 1_{e_i + 2 R_2^0}, 1_{e_j + 2 R_2^0} \rangle.
  \]
  For     $\beta = \begin{pmatrix}
    a & b \\
    c & -a
  \end{pmatrix}$, we see from the formulas
  \begin{align*}
    1_{e_1 + 2 R_2^0}(\beta)
    &= 1_{\mathfrak{p}}(a)
      1_{\mathfrak{o}^\times}(b)
      1_{\mathfrak{o}^\times}(c),
    \\
    1_{e_2 + 2 R_2^0}(\beta)
    &=
      1_{\mathfrak{o}^\times}(a)      1_{\mathfrak{o}^\times}(b)
      1_{\mathfrak{p}}(c),
    \\
    1_{e_3 + 2 R_2^0}(\beta)
    &=
      1_{\mathfrak{o}^\times}(a)
      1_{\mathfrak{p}}(b)
      1_{\mathfrak{o}^\times}(c).
  \end{align*}
  that the sets $e_i + 2 R_2^0$ ($i=1,2,3$) are disjoint
  and that their $d \beta$-volumes are $2^{-3}$.
  If $n = 0$, so that $g$ is the identity,
  it follows that
  \[
    I''(g)
    =
    2 \kappa_0^2 \sum_{i=1,2,3} \|1_{e_i + 2 R_2^0}\|^2
    = 2 \kappa_0^2 \cdot 3 \cdot 2^{-3}
    = 2^{-2} 3 \kappa_0^2.
  \]
  Suppose $n \geq 1$.
  Write $\varpi := 2 \hookrightarrow \mathfrak{p}$.
  We have
  $\Ad(g) 1_{e_i + 2 R^0}
  = 1_{\Ad(g)(e_i + 2 R_2^0)}$ and so
  \begin{align*}
    \Ad(g) 1_{e_1 + 2 R_2^0}(\beta)
    &= 1_{\mathfrak{p}}(a)
      1_{\varpi^n \mathfrak{o}^\times}(b)
      1_{\varpi^{-n} \mathfrak{o}^\times}(c),
    \\
    \Ad(g) 1_{e_2 + 2 R_2^0}(\beta)
    &=
      1_{\mathfrak{o}^\times}(a)      1_{\varpi^{n} \mathfrak{o}^\times}(b)
      1_{\varpi^{-n} \mathfrak{p}}(c),
    \\
    \Ad(g) 1_{e_3 + 2 R_2^0}(\beta)
    &=
      1_{\mathfrak{o}^\times}(a)
      1_{\varpi^n \mathfrak{p}}(b)
      1_{\varpi^{-n} \mathfrak{o}^\times}(c).
  \end{align*}
  Thus
  $\langle \Ad(g) 1_{e_i + 2 R_2^0}, 1_{e_j + 2 R_2^0} \rangle =
  1_{(i,j)=(2,3)} \int_{a,b,c \in k}
  1_{\mathfrak{o}^\times}(a)
  1_{\varpi^n \mathfrak{o}^\times}(b)
  1_{\mathfrak{o}^\times}(c) \, d a \, d b \, d c
  =
  1_{(i,j)=(2,3)} 2^{-3-n}$
  and so
  $I''(g) = 2 \kappa_0^2 \cdot 2^{-3-n} = 2^{-2-n} \kappa_0^2$, as
  required.
\end{proof}

Now equip $G$ with the Haar measure
assigning volume one to $K$.
For $k=1,2$,
denote by $\Xi_k : K \backslash G / K \rightarrow
\mathbf{C}$
be the normalized ($\Xi_k(1) = 1$)
spherical matrix coefficient attached to $\Psi_k$
(see Lemma \ref{lem:rallis-ip-local-unramified-integral}).
\begin{lemma}\label{lem:local-spherical-unram-integral-dummy}
  The integral
  $I_0 := \int_{g \in G}
  \langle \Ad(g) 1_{M_2(\mathbb{Z}_2)}, 1_{M_2(\mathbb{Z}_2)}  \rangle_{L^2(B_2)}
  \Xi_k(g)$
  converges absolutely
  and is given by
  $I_0 = L_2(\Psi_k,\tfrac{1}{2})/\zeta_2(2)$.
\end{lemma}
\begin{proof}
  By direct calculation;
  see Lemma \ref{lem:rallis-ip-local-unramified-integral} below.
\end{proof}

\begin{proposition}
  \label{prop:classical-correlations}
  For $\phi$ as in Proposition \ref{lem:inner-product-involving-phi},
  the integral
  \[I := \int_{g \in G}
    \langle \Ad(g) \phi, \phi   \rangle_{L^2(B_2)}
    \Xi_k(g)\]
  converges absolutely and is given by
  \[
    I
    =
    2^{-4} \kappa_0^2
    \left(
      2 + 
      \frac{L_2(\Psi_k,\tfrac{1}{2})}{\zeta_2(2)}
    \right).
  \]
\end{proposition}

\begin{proof}
  We use the Cartan decomposition,
  writing $G = K \bigsqcup (\sqcup_{n > 0} K a(2^n) K)$.
  For $g = a(2^n)$,
  one has
  $\langle \Ad(g) 1_{M_2(\mathbb{Z}_2)}, 1_{M_2(\mathbb{Z}_2)}
  \rangle_{L^2(B_2)}
  = 2^{-n}$
  and thus
  $\langle \Ad(g) \phi, \phi  \rangle
  = 2^{-4} \kappa_0^2 (1_{n=0} 2 + 
  \langle \Ad(g) 1_{M_2(\mathbb{Z}_2)}, 1_{M_2(\mathbb{Z}_2)}
  \rangle_{L^2(B_2)} )$.
  Since $\vol(K) = 1$,
  we conclude by Lemmas \ref{lem:inner-product-involving-phi} and \ref{lem:local-spherical-unram-integral-dummy}.
\end{proof}

\section{Rallis inner product formula}
\label{sec-7}
The main purpose of this section is to prove the identity
\eqref{eqn:sketch-4-precise},
which is needed to
determine the diagonal entries of the limiting
quantum variance $V_\infty$ in Theorem \ref{thm:special-main}.
The reader who is satisfied with a less precise result in which
the diagonal matrix
$V_\infty$ is left unspecified may skip this
section; on the other hand, the proportionality constants in the
quantum variance problem are of basic interest,
so it would be disappointing if our method were
incapable of determining them.

To prove \eqref{eqn:sketch-4-precise}, we specialize the Rallis
inner product formula from its general adelic formulation and
evaluate some carefully normalized local integrals.  The
  contents of this section are, in principal, straightforward,
  and should not be confused with representing a primary novelty
  of this paper.  On the other hand, some careful book-keeping
is required, and we are not aware of any previous work in which
such a specialization was carried out.  Indeed, although many
classical cases of the Rallis inner product formula are recorded
in the literature (see e.g. \cite[p285]{MR1147818},
\cite[p55]{MR1292796}, \cite[\S6]{MR2504745} and its references,
\cite{MR2604368}, \cite[p969]{MR2776068}), none seem to apply
``off-the-shelf'' to the theta functions $h_1, h_2$ considered
here.  Their methods could likely be adapted, but we have chosen
instead to derive a ``ready-to-use'' specialization of the
general adelic form
whose application here
reduces the proof of
\eqref{eqn:sketch-4-precise}
to the
evaluation of the $2$-adic local integral considered in
\S\ref{sec:classical-correlations}.
We work in generality here
% section in generality
so as not to duplicate effort in future
applications of the method.

% Any difficulty lies in the book-keeping.

% We specialize
% We do so in a manner that should suit future
% studies of quantum variance (on any quaternion algebra)
% by specializing the Rallis inner product formula from its general adelic formulation
% and evaluating some carefully normalized local integrals.

\subsection{Local preliminaries}
\label{sec-7-1}
Let $k$ be a local field of characteristic $\neq 2$.
Denote by $\Mp_2(k)$ the metaplectic
double cover of $\SL_2(k)$,
realized as a set of pairs $(\sigma,\zeta) \in \SL_2(k) \times
\{\pm 1\}$
with the multiplication law
$(\sigma_1,\zeta_1) (\sigma_2,\zeta_2)
=(\sigma_1,\sigma_2, \zeta_1 \zeta_2 c(\sigma_1,\sigma_2))$
for the Kubota cocycle
$c : k^\times/k^{\times 2} \times  k^\times/k^{\times 2}
\rightarrow \{\pm 1\}$;
see e.g. \cite[\S4]{nelson-theta-squared} for details on this and what follows.
As generators for $\Mp_2(k)$ we take
for $a,b,\zeta \in k^\times, k, \{\pm 1\}$ the elements
\[
  n(b) :=
  (\begin{pmatrix}
    1 & b \\
    & 1
  \end{pmatrix}, 1), \quad
  t(a) :=
  (\begin{pmatrix}
    a &  \\
    & a^{-1}
  \end{pmatrix}, 1),
  \quad
  w := (\begin{pmatrix}
    & 1 \\
    -1 & 
  \end{pmatrix}, 1) \]
and $\eps(\zeta) := (1,\zeta)$.
Let $\psi : k \rightarrow \mathbb{C}^{(1)}$ be a nontrivial
character.
Let $(V,q)$ be a quadratic space over $k$.
For $x,y \in V$,
set
$\langle x,y \rangle := q(x+y) - q(x) - q(y)$.
Equip $V$ with the Haar measure $d x$
self-dual for the Fourier transform $\mathcal{F} \phi(y) :=
\int_V
\phi(x) \psi (\langle x,y \rangle) \, d x$
for $\phi$ in the Schwartz--Bruhat space $\mathcal{S}(V)$.
The Weil representation
$\omega_{V,\psi}$ of $\Mp_2(k)$ on $\mathcal{S}(V)$,
abbreviated $\omega := \omega_{V,\psi}$
when $V,\psi$ are clear from context,
is defined on the generators:
there is a quartic character $\chi = \chi_{V,\psi}$ of $k^\times$
and an eighth root of unity $\gamma = \gamma_{V,\psi}$
(see \cite{MR0165033,MR0424695})
so that
$\omega(n(b)) \phi(x) = \psi(b q(x)) \phi(x)$,
$\omega(t(a)) \phi(x) = \chi(a) |a|^{d/2} \phi(a x)$,
$\omega(w) \phi = \gamma \mathcal{F} \phi$,
and $\omega(\eps(\zeta)) \phi = \zeta^{\dim V} \phi$.
The assignment $V \mapsto \omega_{V,\psi}$
is compatible with orthogonal direct sums:
if $V = V' \oplus V''$,
then $\omega_{V,\psi} = \omega_{V',\psi} \otimes
\omega_{V'',\psi}$
wrt the dense inclusion
$\mathcal{S}(V') \otimes \mathcal{S}(V'') \hookrightarrow
\mathcal{S}(V)$.
\subsection{Global preliminaries}
\label{sec-7-2}
Now let $k$ be a number field with adele
ring $\mathbb{A}$.
Denote by $v$ a typical place of $k$.
Denote by $k_v$ the completion.
For an algebraic group $H$ over $k$,
set $H_v := H(k_v)$.
Fix a nontrivial character
$\psi : \mathbb{A}/k \rightarrow \mathbb{C}^{(1)}$.
Let $(V,q)$ be a quadratic space over $k$.
The metaplectic group $\Mp_2(\mathbb{A})$ consists of pairs
$(\sigma,\zeta) \in \SL_2(\mathbb{A}) \times \{\pm 1\}$
with the law $(\sigma_1,\zeta_1) (\sigma_2,\zeta_2)
=(\sigma_1,\sigma_2, \zeta_1 \zeta_2 c(\sigma_1,\sigma_2))$
for the product $c = \prod c_v$
of the local cocycles.
The map $\gamma \mapsto (\gamma,1)$ defines
a splitting $\SL_2(k) \hookrightarrow \Mp_2(\mathbb{A})$.
Equip $V(\mathbb{A})$ with the product
of the local Haar measures on the completions $V(k_v)$.
The space
$\mathcal{S}(V(\mathbb{A}))$
is a restricted tensor product
of the local spaces
$\mathcal{S}(V(k_v))$;
we assume here that the distinguished elements
defining the restricted tensor product
are unit vectors for almost all places.
The Weil representation
$\omega_{\psi,V}$ of $\Mp_2(\mathbb{A})$ on
$\mathcal{S}(V(\mathbb{A}))$
is the restricted tensor product of the local Weil
representations.
\subsection{The main global result}
\label{sec-7-3}
Let $B$ be a quaternion algebra over $k$.  Denote by $B^0$ the
trace zero subgroup and by $G$ the algebraic group with
$G(k) = B^\times / k^\times$.  Set
$[\Mp_2] := \SL_2(k) \backslash \Mp_2(\mathbb{A}), [\SL_2] :=
\SL_2(k) \backslash \SL_2(\mathbb{A})$,
$[G] := G(k) \backslash G(\mathbb{A})$.  Then
$(k, x \mapsto x^2)$, $(B^0, x \mapsto \nr(x))$ and
$(B, x \mapsto \nr(x))$ are quadratic spaces with
$B = k \oplus B^0$.
The above discussion applies,
giving Weil representations
of $\Mp_2(\mathbb{A})$
on $\mathcal{S}(\mathbb{A}), \mathcal{S}(B_\mathbb{A}^0),
\mathcal{S}(B_\mathbb{A})$
which
we denote by $\omega$ for notational simplicity.
The adjoint action $\Ad : G \rightarrow \SO(B^0)$
defined as in \eqref{eq:adjoint-action}
is an isomorphism.
Let $\pi \subset L^2([G])$
be a cuspidal automorphic representation.
Fix a unitary factorization
$\pi \cong \otimes \pi_v$.
Equip $G_v$ with arbitrary
Haar measures;
assume only that for non-archimedean $v$ outside some finite
set,
the maximal compact in $G_v$
has volume one.
Equip $G(\mathbb{A})$ with the restricted product
measure and $[G]$ the quotient measure.\footnote{
  It might appear natural to consider
  here some specific measure on $G(\mathbb{A})$,
  such as the Tamagawa measure,
  but the main identities
  to be considered are self-normalizing
  for the choice of Haar measure on $G(\mathbb{A})$, so we do not.
}
For Schwartz functions $\phi ', \phi '' \in
\mathcal{S}(\mathbb{A}),
\mathcal{S}(B^0(\mathbb{A}))$
and $\Psi \in \pi$,
define for $\sigma \in [\Mp_2]$ and $g \in [G]$
the adelic theta functions
$\theta_{\phi '}(\sigma) := \sum_{m \in k} \omega(\sigma) \phi '(m)$,
$\theta_{\phi ''}(\sigma,g) := \sum_{\beta \in B^0} \omega(\sigma) \Ad(g) \phi '(m)$,
$\theta_{\phi '',\Psi}(\sigma) := \int_{g \in [G]}
\Psi(g) \theta_{\phi ''}(\sigma,g)$.
Denote simply by $\int$ an integral over $[\SL_2]$
with respect to the Tamagawa measure,
which is the probability Haar.
\begin{theorem}\label{thm:quaternary-rallis-adelic}
  For $i=1,2$,
  let
  $\phi'_i = \otimes \phi_{i,v}' \in \mathcal{S}(\mathbb{A})^{(+)}$,
  $\phi''_i = \otimes \phi_{i,v}'' \in \mathcal{S}(B^0(\mathbb{A}))$
  and
  $\Psi_i = \otimes \Psi_{i,v} \in \pi$.
  For a sufficiently large finite set $S$ of places,
  \[
  \int (\theta_{\phi_1'', \Psi_1} \overline{\theta_{\phi_2'',
    \Psi_2}})
  \int (\theta_{\phi_1'} \overline{\theta_{\phi_2'}})
  = 
  2
  \frac{L^{(S)}(\pi,\tfrac{1}{2})}{\zeta_k^{(S)}(2)}
  \prod_{v \in S}
  I_v
  \]
  where
  $I_v := 
  \int_{g \in G_v}
  \langle \pi_v(g) \Psi_{1,v}, \Psi_{2,v} \rangle_{\pi_v}
  \langle \Ad(g) \phi_{1,v}, \phi_{2,v} \rangle_{L^2(B_v)}$
  with $\phi_{i,v} := \phi_{i,v}' \otimes \phi_{i,v}'' \in \mathcal{S}(B_v)$.
\end{theorem}
\begin{proof}
  By polarization, it suffices to consider the notationally simpler case
  $\Psi_{1,v} = \Psi_{2,v} =: \Psi_v$,
  $\phi_{1,v}' = \phi_{2,v}' =: \phi_v'$,
  $\phi_{1,v}'' = \phi_{2,v}'' =: \phi_v''$,
  $\phi_{1,v} = \phi_{2,v} =: \phi_v$, and similarly
  without the subscripts ``$v$.''
  From the factorization of inner products
  \[\langle \Ad(g) \phi_{v}, \phi_{v} \rangle
  = \langle \phi_{v}', \phi_{v}' \rangle \langle \Ad(g)
  \phi_{v}'', \phi_{v}'' \rangle\]
  according to the orthogonal decomposition $B_v = k_v \oplus
  B^0_v$,
  we see that $I_v = I_v' I_v''$, where
  \begin{align*}
    I_v'
    &:=
      \langle \phi_{v}', \phi_{v}' \rangle_{L^2(k_v)},
    \\
    I_v'' &:=
            \int_{g \in G_v}
            \langle \pi_v(g) \Psi_{v}, \Psi_{v} \rangle_{\pi_v}
            \langle \Ad(g) \phi_{v}'', \phi_{v}''
            \rangle_{L^2(B_v^0)}.
  \end{align*}
  It thus suffices to show separately that
  \begin{align*}
    \int |\theta_{\phi'}|^2
    &=
      2
      \prod_{v \in S}
      I_v',
%      \langle \phi_{v}', \phi_{v}' \rangle _{L^2(k_v)},
    \\
    \int |\theta_{\phi'', \Psi}|^2
    &=
      \frac{L^{(S)}(\pi,\tfrac{1}{2})}{\zeta_k^{(S)}(2)}
      \prod_{v \in S}
      I_v''.
      % \int_{g \in G_v}
      % \langle \pi_v(g) \Psi_{v}, \Psi_{v} \rangle_{\pi_v}
      % \langle \Ad(g) \phi_{v}'', \phi_{v}'' \rangle _{L^2(B_v^0)}.
  \end{align*}
  The first identity follows from \cite[Thm 4]{nelson-theta-squared},
  while the second identity follows from specializing the Rallis
  inner product
  given
  by \cite[Thm 6.6]{MR2837015};
  compare with \cite[Prop 2.8 (i)]{MR3291638} for a
  more explicit formulation
  in the representative case $B = M_2(k)$.
\end{proof}
\subsection{Some local integrals}
\label{sec-7-4}
\begin{lemma}[The unramified case]
  \label{lem:rallis-ip-local-unramified-integral}
  Let $k$ be a non-archimedean local field of characteristic
  zero with ring of integers $\mathfrak{o}$.
  Let $\psi : k \rightarrow \mathbb{C}^{(1)}$
  be an unramified nontrivial character.
  Equip $M_2(k)$ with the Haar measure attached to $\psi$ as in \S\ref{sec-7-1},
  so that $\vol(M_2(\mathfrak{o})) = 1$.
  Set $\phi := 1_{M_2(\mathfrak{o})} \in \mathcal{S}(M_2(k))$.
  Equip $\PGL_2(k)$ with any Haar measure.
  Let $\pi$ be an unramified generic irreducible unitary representation
  of $\PGL_2(k)$ with $L^2$-normalized $\PGL_2(\mathfrak{o})$-invariant vector $v$
  and normalized spherical function $\Phi : = \PGL_2(k)
  \rightarrow \mathbb{C}$
  given by $\Phi(g) := \langle \pi(g) v, v \rangle$, so that $\Phi(1) = 1$.
  Then
  \[
    \int_{g \in \PGL_2(k)}
    \langle \Ad(g) \phi, \phi  \rangle_{L^2(M_2(k))}
    \Phi(g)
    = \vol(\PGL_2(\mathfrak{o})) \frac{L(\pi,\tfrac{1}{2})}{\zeta_k(2)},
  \]
  and the LHS converges absolutely.
\end{lemma}
\begin{proof}
  This is implicit in Theorem \ref{thm:quaternary-rallis-adelic}.
  A direct proof follows from the Cartan
  decomposition as in \S\ref{sec:classical-correlations}
  and the Macdonald formula as in \cite[Thm 4.6.6]{MR1431508}; we omit the routine details.
\end{proof}

\begin{lemma}[The non-split case for the trivial representation]\label{lem:local-rallis-integral-real-case}
  Let $k$ be a local field $\neq \mathbb{C}$ of characteristic
  zero,
  $\psi : k \rightarrow \mathbb{C}^{(1)}$ a nontrivial
  character,
  and $B$ the non-split quaternion algebra over $k$.
  Set $H := P B ^\times$; it is compact.
  Equip $B$ with the Haar measure attached to $\psi$ as in
  \S\ref{sec-7-1}, and equip $H$ with any Haar measure.
  If $k = \mathbb{R}$,
  set
  $\phi(x) = e^{- 2 \pi \nr(x)}$ and $q := 1$;
  if $k$ is non-archimedean, set $\phi := 1_{R}$ and $q := \# \mathfrak{o}/\mathfrak{p}$,
  where $R \subset B$ and $\mathfrak{o} \subset k$ are maximal
  orders and $\mathfrak{p} \subset \mathfrak{o}$ the maximal ideal.
  Let $\Phi$ be the constant function $\Phi(g) := 1$.
  Let $c \in k^\times$ be such that $\psi(x) = \psi_1(c x)$, 
  where $\psi_1(x) := e^{4 \pi i x}$ if $k = \mathbb{R}$
  and $\psi_1$ is unramified if $k$ is non-archimedean.
  Then
  \[
  \int_{g \in H}
  \langle \Ad(g) \phi, \phi  \rangle_{L^2(B)}
  \Phi(g)
  =
  \vol(H) q^{-1} |c|^2.
  \]
\end{lemma}
\begin{proof}
  We have $\Ad(g) \phi = \phi$,
  so the integrand is constant, and the formula to be proved is simply that
  $\|\phi\|^2 = q^{-1} |c|^2$.
  Consider first the archimedean case $k = \mathbb{R}$.
  We have $\|\phi\|^2 =
  \int_{x \in B}
  e^{- 4 \pi \nr(x)}
  =
  (|c|^{1/2})^{4}
  \int_{x \in B}
  e^{- 4 \pi |c| \nr(x)}
  = |c|^2 q^{-1} \mathcal{F} \phi_c(0)$,
  where $\phi_c(x) := \psi(\pm i \nr(x))$ with $\pm 1 := \sgn(c)$;
  it remains only to verify that $\mathcal{F} \phi_c (0) = 1$.
  For that, we expand out in terms of Gaussians
  using coordinates $x = (x_0,x_1,x_2,x_3)$ on $B$
  for which $\nr(x) = \sum x_j^2$:
  \[
  \phi_c(x)
  =
  \prod e^{- 4 \pi |c| x_j^2},
  \quad
  \psi(\langle x, y \rangle)
  =
  \prod e^{8 \pi i |c| x_j y_j}.
  \]
  From
  this
  and the known Fourier transform
  of the Gaussian,\footnote{
    For $y \in \mathbb{R}$,
    one has
    $\int_{x \in \mathbb{R}} e^{- 4 \pi |c| x^2 + 8 \pi i |c| x
      y}\, d x
    = (2 |c|^{1/2})^{-1} \int_{x \in \mathbb{R}} e^{-  \pi x^2 + 2 \pi i  x (2 |c|^{1/2} y)}\, d x
    =
    (2 |c|^{1/2})^{-1}
    e^{-4 \pi |c| y^2}$.
  }
  we see that $\mathcal{F} \phi_c$ is some multiple
  of $\phi_c$.
  By the measure normalization,
  it follows that $\mathcal{F} \phi_c = \phi_c$
  and hence $\mathcal{F} \phi_c(0) = 1$, as required.
  We turn now to the non-archimedean case.
  Let $j$ be a uniformizer for $R$.
  The lattice dual of $R$ with respect to $\psi$
  is\footnote{
    Write
    $B = E + E j$, where $E/k$ is the unramified quadratic field
    extension, $j \alpha j^{-1} = \overline{\alpha}$ for $\alpha
    \in E$,
    and $\varpi := j^2$ is a uniformizer of $k$.
    Then $R = \mathfrak{o}_E + \mathfrak{o}_E j$ with
    $\mathfrak{o}_E \subset E$
    the maximal order.
    For $x = x_0 + x_1 j, y = y_0 + y_j \in B$,
    we have
    $\langle x, y \rangle = \tr(x_0 \overline{y_0}) - \tr(x_1
    \overline{y_1}) \varpi$.
    Thus $x \in R$ satisfies
    $\psi (\langle x, R \rangle) = \{1\}$
    iff
    $\psi_1 (c \langle x, R \rangle) = \{1\}$
    iff
    $\tr(x_0 \mathfrak{o}_E) \subseteq c^{-1} \mathfrak{o}$
    and
    $\tr(x_1 \mathfrak{o}_E) \subseteq c^{-1}  \varpi^{-1} \mathfrak{o}$
    iff $x \in c^{-1} j^{-1} R$.
  }
  $\{ x \in R : \psi(\langle x, R \rangle) = \{ 1 \} \} = c^{-1} j^{-1} R$, so $\mathcal{F} 1_R = \gamma 1_{c^{-1} j^{-1} R}$ for some $\gamma > 0$.
  The volume of $c^{-1} j^{-1} R$ is that of $R$ multiplied by
  $|\nr(c^{-1} j^{-1})|^2 = |c|^{-4} q^2$,
  hence
  $1_R = \mathcal{F} \mathcal{F} 1_R = \gamma^2 |c|^{-4} q^2 1_R$,
  giving $\gamma = q^{-1} |c|^2$.
  Therefore $\|1_R\|^2 = \mathcal{F} 1_R (0) = q^{-1} |c|^2$, as required.
\end{proof}
  \subsection{Specialization\label{sec:rallis-ipf-specialization}}
  \label{sec-7-5}
  Suppose now $k := \mathbb{Q}$,
  that $\psi$ is the standard additive character
  (thus $\psi = \prod \psi_v$
  with $\psi_\infty(x) := e^{2 \pi i x}$),
  that $B$ is a definite quaternion algebra over $\mathbb{Q}$.
  Denote by $\ram(B) \ni \infty$ the set of places at which it ramifies.
  Fix a maximal order $R \subset B$.
  Let $K := \prod K_v = G_\infty \times K_{\fin}, K_{\fin} =
  \prod_p K_p$
  where $K_p = R_p^\times / \mathbb{Z}_p^\times$.
  Assume that the chosen Haar measure on $G_v$
  assigns volume one to each $K_v$.
  Then
  \begin{equation}\label{eq:non-split-volume-total}
    \vol(G_p) = 2 \text{ for all } p \in \ram(B).
  \end{equation}
  Let $S$ be a finite set of finite primes
  containing all $p \in \ram(B)$.
  Let
  $\phi' = \otimes \phi '_v \in \mathcal{S}(\mathbb{A})^{(+)}$,
  $\phi'' = \otimes \phi ''_v \in \mathcal{S}(B^0(\mathbb{A}))$
  satisfy:
  for $p \in \ram(B)$ or $p \notin S$,
  \begin{equation}\label{eq:unram-local-data-assumptions}
    \phi_\infty'(x) = e^{- 2 \pi x^2},
    \,
    \phi_{\infty}''(x) = e^{- 2 \pi \nr(x)},
    \,
    \phi_p' = 1_{\mathbb{Z}_p},
    \,
    \phi_p'' = 1_{R_p^0}.
  \end{equation}
  Thus $\phi_v', \phi_v''$ is determined except when
  $v  \in S - \ram(B)$.
  Let
  $\Psi = \otimes \Psi_{v} \in \pi = \otimes \pi_v \subseteq L^2([G])$
  be a factorized cusp form
  that is \emph{trivial} at all $v \in \ram(B)$.
  Assume that the factorization
  of unitary structures on $\pi$
  is normalized
  so that $\|\Psi_{v}\|_{\pi_{v}} = 1$
  for $v \in \ram(B)$ or  $v \notin S$.
  Then
  $\|\Psi\|_{L^2([G])}
    =
    \prod_{p \in S - \ram(B)}
    \|\Psi_p\|_{\pi_p}$.
  Set
  $\tilde{H}_0
  :=
  \theta_{\phi '}$,
  $\tilde{H}
  := \theta_{\phi '', \Psi}$.
  Theorem \ref{thm:quaternary-rallis-adelic}
  and the results of \S\ref{sec-7-4} give
  \begin{equation}\label{eq:specialized-integral-of-stuff}
    \int |\tilde{H}_0|^2
  \int |\tilde{H}|^2
  =
  2 \frac{L^{(S)}(\pi,\tfrac{1}{2})}{\zeta^{(S)}(2)}
  I_\infty 
  \prod_{p \in S} I_p
  \end{equation}
  with $I_\infty = 1/4$ and $I_p = 2/p$ for $p \in \ram(B)$
  and $I_p$ as in Theorem \ref{thm:quaternary-rallis-adelic} for
  $p \in S - \ram(B)$.
\subsection{Unadelization}
\label{sec-7-6}
Let $\Phi : \mathbb{H} \rightarrow \mathbb{C}$ be an automorphic
form of weight $k/2$ ($k \in \mathbb{Z}_{\geq 1}$) on
$\Gamma(N)$ for some $N \equiv 0 \pmod{4}$.  It can be lifted
uniquely to an adelic automorphic form $\tilde{\Phi}$ on
$[\Mp_2]$ by requiring that
$\tilde{\Phi}(n(x) t(y^{1/2}) \eps(\zeta)) = \zeta^{k} \Phi(x +
i y)$
for $x,y,\zeta \in \mathbb{R}, \mathbb{R}_+^\times,\{ \pm 1\}$;
then $\tilde{\Phi}$ transforms under the $k$th standard character of
the standard maximal compact subgroup of $\Mp_2(\mathbb{R})$ and is invariant
by the standard splitting of the $N$th principal congruence
subgroup of $\SL_2(\mathbb{A}_f)$ (see e.g. \cite[\S4.5]{nelson-theta-squared}).
Any automorphic
form $\tilde{\Phi}$ on $[\Mp_2]$ with these properties arises in
this way from $\Phi(x + iy) := \tilde{\Phi}(n(x) t(y^{1/2})$,
and the bijection $\Phi \leftrightarrow \tilde{\Phi}$
intertwines the inner products defined in
\S\ref{sec:some-conventions} and \S\ref{sec-7-3}.
This discussion applies to $\tilde{H}_0, \tilde{H}$;
by inspection of the formulas defining the theta lift,
one finds that
the corresponding classical modular forms
$H_0, H$
have respective weights $1/2, 3/2, 3/2$
and are given with
$\phi_{\fin}' := \otimes \phi_p'$,
$\phi_{\fin}'' := \otimes \phi_p''$
by
\[
H_0(z)
= y^{1/4}
\sum_{m \in \mathbb{Q}}
\phi_{\fin}'(m) q^{m^2},
\,
H(z)
=
y^{3/4}
\int_{g \in [G]}
\Psi(g)
\sum_{\beta \in B^0}
\Ad(g) \phi_{\fin}''(\beta) q^{\nr(\beta)}.
\]
By \eqref{eq:unram-local-data-assumptions},
the summations may be restricted to $S$-integral elements.
\subsection{Completion of the calculation}
\label{sec-7-8}
Assume now that $B$ is as in \S\ref{sec:preliminaries}, that
\eqref{eq:unram-local-data-assumptions} is satisfied for
all $p \notin S := \{2\}$, and that $\phi_2', \phi_2''$
are as in Lemma \ref{lem:inner-product-involving-phi}.
These assumptions fully determine
$\phi', \phi''$.
They imply that
$H_0(z)$ is the Jacobi theta function
defined in \S\ref{sec:reduction-proof}
which we denote here by
$\theta_{\Jac}(z) := H_0(z)$.
Recall the set
$\mathbf{Y} = \Gamma \backslash
\GL_2(\mathbb{Q}_2)/\GL_2(\mathbb{Z}_2)$
defined in \S\ref{sec-1}.
By strong approximation, one has a
bijection
$\mathbf{Y} \cong [G]/K$
induced by
$\GL_2(\mathbb{Q}_2) \twoheadrightarrow \PGL_2(\mathbb{Q}_2)
\cong G(\mathbb{Q}_2) \hookrightarrow G(\mathbb{A})$
(see e.g. \cite[\S3]{MR894322}).
The measure on $\mathbf{Y}$ from \S\ref{sec:measures}
is compatible
with that induced on $[G]/K$ from \S\ref{sec:rallis-ipf-specialization}
because for $g \in \GL_2(\mathbb{Q}_2)$ with image $[g] \in
G(\mathbb{Q}_2)$,
one has
$(1/2) \# \Gamma \cap g \GL_2(\mathbb{Z}_2) g^{-1}
= \# G(\mathbb{Q}) \cap [g] K [g]^{-1}$.
Under this bijection, $\Psi_1, \Psi_2 : \mathbf{Y} \rightarrow \mathbb{R}$
identify with factorizable $K$-invariant vectors
(also denoted) $\Psi_k : [G] \rightarrow \mathbb{R}$
in automorphic representations $\pi_1, \pi_2
\subset L^2([G])$
for which
$L(\pi_k,s) = L(\Psi_k,s)$.
Let $k \in \{1,2\}$
and set $\Psi = \Psi_k$;
one finds using SAGE or otherwise\footnote{
e.g., using that the supersingular
polynomial $j(j+4)(j-1768)$
splits over
 $\mathbb{F}_{23}$
} that $\Psi$ is trivial at $23$
and using \S\ref{sec:defn-ternary-thetas}
that
$H(z) = h_k(z)$.
We have
$\|\Psi_{k,2}\|_{\pi_2}^2
=
\|\Psi_k\| = 1$
and hence for $g \in \PGL_2(\mathbb{Q}_2) =
B_2^\times/\mathbb{Q}_2^\times$
that
$\langle \pi_{k,2}(g) \Psi_{k,2}, \Psi_{k,2} \rangle_{\pi_{k,2}}
=
\Xi_k(g)$
with $\Xi_k$ as in \S\ref{sec:classical-correlations}.
By \eqref{eq:specialized-integral-of-stuff}, it follows that
\[
\|\theta_{\Jac}\|^2 \|h_k\|^2
=
\int |\tilde{H}_0|^2
\int |\tilde{H}|^2
= 2 \frac{L^{(S)}(\Psi_k,\tfrac{1}{2})}{\zeta^{(S)}(2)}
I_\infty I_{23}
I_2
\]
with
$I_2 =
\int_{g \in \PGL_2(\mathbb{Q}_2)}
\Xi_k(g)
\langle \Ad(g) \phi_2, \phi_2 \rangle_{L^2(B_2)}$
 and  $I_\infty = 1/4, , I_{23} = 2/23$.
The measure defining
$L^2(B_2)$ is compatible with
that in \S\ref{sec:classical-correlations},
so by Proposition \ref{prop:classical-correlations} and the
definition
(\S\ref{sec-1-3-7}) of $\kappa_0$,
\[
I_2
= 
2^{-4}
\kappa_0^2 \left( 2 + \frac{L_2(\Psi_k,\tfrac{1}{2})}{\zeta_2(2)} \right)
= \kappa_1^{-1} 2^{-6}
\zeta_2(1) \zeta_2(2)
\left(2 + \frac{L_2(\Psi_k,\tfrac{1}{2})}{ \zeta_2(2)}\right).
\]
Recalling the definition
(\S\ref{sec-1-3-7}) of $\kappa_1$, we obtain
\[
\|\theta_{\Jac}\|^2 \|h_k\|^2
= 
2 L^{(S)}(\Psi_k,\tfrac{1}{2})
\cdot \frac{(4 \pi)^2}{4} 
\cdot \frac{2     \zeta_{23}(1) }{23}
\cdot
2^{-6}
\zeta_2(1)^2
\zeta_2(2)
\left(2 + \frac{L_2(\Psi_k,\tfrac{1}{2})}{ \zeta_2(2)}\right).
\]
We  ``simplify'' this a bit using that $\zeta_{23}(1) =
L_{73}(\Psi_k,\tfrac{1}{2})$
and factoring $L_2(\Psi_k,\tfrac{1}{2})$ out of the rightmost
expression.
Either by inspection
or by the SAGE output pasted below,
we evaluate the RHS to
$L(\Psi_k,\tfrac{1}{2}) P(\lambda_{\Psi_k}(2))$,
as required.
\begin{verbatim}
lv = 1/(1 - x/sqrt(2) + 1/2); C0 = 2 * (4*pi)^2/4
C2 = 2^(-6)/((1 - 1/2)^2*(1-1/2^2)) * (2/lv + (1-1/2^2))
C23= 2/23; C = C0 * C23 * C2
C
       -1/69*pi^2*(4*sqrt(2)*x - 15)
\end{verbatim}

% -----------------
\bibliography{refs}{}
\bibliographystyle{plain}
% -----------------
\end{document}